\documentclass[reqno,12pt]{amsart}

\usepackage{color} 
\usepackage{ifpdf}
\ifpdf
    \usepackage[pdftex]{graphicx}
    \usepackage[pdftex]{hyperref}
    \hypersetup{
        unicode=false,          
        pdftoolbar=true,        
        pdfmenubar=true,        
        pdffitwindow=false,     
        pdfstartview={FitH},    
        pdftitle={MCP Article},      
        pdfauthor={Michael Holst},   
        pdfsubject={Mathematics},    
        pdfcreator={Michael Holst},  
        pdfproducer={Michael Holst}, 
        pdfkeywords={PDE, analysis, mathematical physics}, 
        pdfnewwindow=true,      
        colorlinks=true,        
        linkcolor=red,          
        citecolor=blue,         
        filecolor=magenta,      
        urlcolor=cyan           
    }
    \def\myfigpng#1#2{\includegraphics[height=#2]{#1.png}}
    \def\myfigpdf#1#2{\includegraphics[height=#2]{#1.pdf}}

\else
    \usepackage{graphicx}
    \usepackage{pstricks}
    
    \newcommand{\href}[2]{#2}
    \def\myfigpng#1#2{\includegraphics[height=#2]{#1}} 
    \def\myfigpdf#1#2{\includegraphics[height=#2]{#1}} 
\fi

\usepackage{times}
\usepackage{amsfonts}

\usepackage{amsmath}
\usepackage{amsthm}
\usepackage{amssymb}
\usepackage{amsbsy}
\usepackage{amscd}

\usepackage{enumerate}
\usepackage{verbatim}
\usepackage{subfigure}

\newtheorem{theorem}{Theorem}[section]

\newtheorem{lemma}[theorem]{Lemma}

\newtheorem{assumption}[theorem]{Assumption}

\numberwithin{equation}{section}  

  \newcounter{mnote}
  \let\oldmarginpar\marginpar
    \renewcommand\marginpar[1]{\-\oldmarginpar[\raggedleft\footnotesize #1]%
    {\raggedright\footnotesize #1}}

\newenvironment{algorithmX}
{\begin{trivlist}\em\item[\hskip\labelsep{\bf Algorithm:}]}{\end{trivlist}}

\definecolor{myblue}{rgb}{0.2,0.2,0.7}
\definecolor{mygreen}{rgb}{0,0.6,0}
\definecolor{mycyan}{rgb}{0,0.6,0.6}
\definecolor{myred}{rgb}{0.9,0.2,0.2}
\definecolor{mymagenta}{rgb}{0.9,0.2,0.9}
\definecolor{mywhite}{rgb}{1.0,1.0,1.0}
\definecolor{myblack}{rgb}{0.0,0.0,0.0}

\newcommand{\beq}{\begin{equation}}
\newcommand{\eeq}{\end{equation}}
\newcommand{\beqa}{\begin{eqnarray}}
\newcommand{\eeqa}{\end{eqnarray}}
\newcommand{\Ahatstar}{{}^{*}\!\!\hat{A}}

\newcommand{\PLTMG}{{\sc PLTMG}}
\newcommand{\MC}{{\sc MC}}

\newcommand{\MALOC}{{\sc MALOC}}
\newcommand{\SG}{{\sc SG}}
\newcommand{\MCLAB}{{\sc MCLab}}

\newcommand{\FETK}{{\sc FETK}}

\def\text#1{{\rm #1}}

\def\calg#1{{\mathcal #1}}
\def\bbbb#1{{\mathbb #1}}
\def\MC{{\sc MC}}

\begin{document}

\title[Adaptive Numerical Treatment of Elliptic Systems on Manifolds]
      {Adaptive Numerical Treatment of \\ Elliptic Systems on Manifolds}

\author[M. Holst]{Michael Holst}
\email{mholst@math.ucsd.edu}

\address{Department of Mathematics,
         University of California, San Diego
         9500 Gilman Drive, Dept. 0112,
         La Jolla, CA 92093-0112 USA}

\thanks{The author was supported in part by NSF CAREER Award~9875856, by NSF Grants~0225630, 0208449, 0112413, and by a UCSD Hellman Fellowship.}

\date{March 1, 2001}


\begin{abstract}
Adaptive multilevel finite element methods are developed and analyzed
for certain elliptic systems arising in geometric analysis and
general relativity.  
This class of nonlinear elliptic systems of tensor equations on manifolds
is first reviewed, and then adaptive multilevel finite element methods for
approximating solutions to this class of problems are considered in some
detail.
Two {\em a posteriori} error indicators are derived,
based on local residuals and on global linearized
adjoint or {\em dual} problems.
The design of Manifold Code (\MC) is then discussed;
\MC\ is an adaptive multilevel finite element software package for
2- and 3-manifolds developed over several years at Caltech and UC San Diego.
It employs {\em a posteriori} error estimation, adaptive simplex subdivision, 
unstructured algebraic multilevel methods, global inexact Newton methods, 
and numerical continuation methods for the
numerical solution of nonlinear covariant elliptic systems on
2- and 3-manifolds.
Some of the more interesting features of \MC\ are described in detail,
including some new ideas for topology and geometry representation
in simplex meshes, and an unusual partition of unity-based method
for exploiting parallel computers.
A short example is then given which involves the
Hamiltonian and momentum constraints in the Einstein equations,
a representative nonlinear 4-component covariant elliptic system on
a Riemannian 3-manifold which arises in general relativity.
A number of operator properties and solvability results recently
established 
are first summarized, making possible
two quasi-optimal {\em a priori} error estimates for Galerkin
approximations which are then derived.
These two results complete the theoretical framework for effective
use of adaptive multilevel finite element methods.
A sample calculation using the \MC\ software is then presented.
\end{abstract}

\maketitle


\vspace*{-1.2cm}
{\footnotesize
\tableofcontents
}

\section{Introduction}

In this paper we consider adaptive multilevel finite element
methods for certain elliptic systems arising in geometric analysis
and general relativity.  
Our interest is in developing adaptive approximation techniques
for the highly accurate and efficient numerical solution of this
class of problems.
We begin by giving a brief introduction to this class of nonlinear
elliptic systems of tensor equations on manifolds,
and then discuss adaptive multilevel finite element methods for
approximating solutions to this class of problems.
We derive two {\em a posteriori} error indicators,
the first of which is local residual-based, whereas the second
is based on a global linearized adjoint or {\em dual} problem.

The design of a computer program called Manifold Code (\MC) is then
described, which is an adaptive multilevel finite element software package for 
partial differential equations (PDEs) on 2- and 3-manifolds developed
over several years at Caltech and UC San Diego.
\MC\ employs {\em a posteriori} error estimation, adaptive simplex subdivision, 
unstructured algebraic multilevel methods, global inexact Newton methods, 
and numerical continuation methods for the accurate and efficient
numerical solution of nonlinear covariant elliptic systems on
2- and 3-manifolds.
We describe some of the more interesting features of \MC\ in detail,
including some new ideas for topology and geometry representation
in simplex meshes.
We also describe an unusual partition of unity-based method in \MC\ for
using parallel computers in an adaptive setting,
based on joint work with R. Bank~\cite{BaHo98a}.
Global $L^2$- and $H^1$-error estimates are derived for solutions
produced by MC's parallel algorithm by using
Babu\v{s}ka and Melenk's Partition of Unity Method (PUM)
error analysis framework~\cite{BaMe97} and by exploiting the recent results
of Xu and Zhou on local error estimation~\cite{XuZh97}.

We finish with an example involving the Hamiltonian and momentum
constraints in the Einstein equations, a representative nonlinear
4-component covariant elliptic system on a Riemannian 3-manifold which
arises in general relativity.
We first summarize a number of operator properties and solvability
results which were established recently in~\cite{HoBe01}.
We then derive two quasi-optimal {\em a priori} error estimates for
Galerkin approximations of the constrains, completing the theoretical
framework for effective use of adaptive multilevel finite element methods.
We then present a sample calculation using the \MC\ software for this
application.  More detailed examples involving the use of MC for the Einstein
constraints may be found in~\cite{HoBe99,BeHo99}.
Applications of \MC\ to problems in other areas such as biology
and elasticity can be found in~\cite{HBW99,BHW99,BaHo98a}.

\section[Adaptive Finite Element Methods
         for Nonlinear Equations]
        {Adaptive Multilevel Finite Element Methods
         for Nonlinear Elliptic Equations}
  \label{sec:elliptic_overview}

In this section we will first give an overview of nonlinear elliptic
equations on manifolds, followed by a brief description of adaptive 
multilevel finite element techniques for such equations.

\subsection{Nonlinear elliptic equations on manifolds}
  \label{sec:elliptic_general}

Let $(\calg{M},g_{ab})$ be a connected compact Riemannian $d$-manifold with
boundary $(\partial \calg{M},\sigma_{ab})$,
where the boundary
metric $\sigma_{ab}$ is inherited from $g_{ab}$.
To allow for general boundary conditions, 
we will view the boundary $(d-1)$-submanifold $\partial \calg{M}$
(which we assume to be oriented)
as being formed from two disjoint
submanifolds $\partial_0 \calg{M}$ and $\partial_1 \calg{M}$, 
i.e.,
\begin{equation}
   \label{eqn:bndry}
\partial_0 \calg{M} \cup \partial_1 \calg{M} = \partial \calg{M},
\ \ \ \ \ \ \ \ 
\partial_0 \calg{M} \cap \partial_1 \calg{M} = \emptyset.
\end{equation}
When convenient in the discussions below, one of the two submanifolds
$\partial_0 \calg{M}$ or $\partial_1 \calg{M}$ may be allowed to shrink to
zero measure, leaving the other to cover $\partial \calg{M}$.
Moreover, in what follows it will usually be necessary to make smoothness
assumptions about the boundary submanifold $\partial \calg{M}$, such as
Lipschitz continuity (for a precise definition see~\cite{Adam78}).
We will employ the abstract index notation (cf.~\cite{Wald84}) and
summation convention for tensor expressions below, with indices
running from $1$ to $d$ unless otherwise noted.
The summation convention is that all repeated symbols in products
imply a sum over that index.
Partial differentiation on a non-flat manifold must be covariant, meaning
that application of gradient and divergence operators require the use of
a connection due to the curvilinear nature of the coordinate
system used to describe the domain manifold.
Christoffel symbols formed with respect to the given
metric $g_{ab}$ (at times denoted $\hat{\gamma}_{ab}$)
provide default connection coefficients.

Covariant partial differentiation of a tensor
$t^{a_1\cdots a_p}_{~~~~~~~b_1\cdots b_q}$
using the connection provided by the metric $g_{ab}$ will be denoted
as $t^{a_1\cdots a_p}_{~~~~~~~b_1\cdots b_q;c}$
or as $D_c t^{a_1\cdots a_p}_{~~~~~~~b_1\cdots b_q}$.
Denoting the outward unit normal to $\partial \calg{M}$ as $n_b$,
recall the Divergence Theorem for a vector
field $w^b$ on $\calg{M}$ (cf.~\cite{Lee97}):
\begin{equation}
   \label{eqn:divergence_vector}
\int_{\calg{M}} w^b_{~;b} ~dx = \int_{\partial \calg{M}} w^b n_b ~ds,
\end{equation}
where $dx$ denotes the measure on $\calg{M}$ generated by the volume
element of $g_{ab}$:
\begin{equation}
   \label{eqn:volume_element}
dx = \sqrt{\text{det}~g_{ab}}~dx^1 \cdots dx^d,
\end{equation}
and where $ds$ denotes the boundary measure on $\partial \calg{M}$
generated by the boundary volume element of $\sigma_{ab}$.
Making the choice $w^b = u_{a_1\ldots a_k} v^{a_1\ldots a_kb}$
in~(\ref{eqn:divergence_vector})
and forming the divergence $w^b_{~;b}$ by applying the product rule
leads to a useful integration-by-parts formula for certain
contractions of tensors:
\begin{eqnarray}
   \label{eqn:divergence}
  \int_{\calg{M}}
    u_{a_1\ldots a_k} v^{a_1\ldots a_kb}_{~~~~~~~~~;b} ~dx
&=& \int_{\partial \calg{M}}
    u_{a_1\ldots a_k} v^{a_1\ldots a_kb} n_b ~ds
\\
& &
- \int_{\calg{M}}
    v^{a_1\ldots a_kb} u_{a_1\ldots a_k;b} ~dx.
\nonumber
\end{eqnarray}
When $k=0$ this reduces to the familiar case where $u$ and $v$ are scalars.

\subsubsection{Coupled elliptic systems and augumented systems}

Consider now a general second-order elliptic system
of tensor equations in strong divergence form
over $\calg{M}$:
\begin{eqnarray}
\label{eqn:str1}
  - A^{ia}(x^b,u^j,u^k_{~;c},\lambda)_{;a} + B^i(x^b,u^j,u^k_{~;c},\lambda)
      & = & 0 ~\text{~in~} \calg{M}, \\ 
\label{eqn:str2}
    A^{ia}(x^b,u^j,u^k_{~;c},\lambda) n_a + C^i(x^b,u^j,u^k_{~;c},\lambda)
      & = & 0 ~\text{~on~} \partial_1 \calg{M}, \\ 
\label{eqn:str3}
u^i(x^b) & = & E^i(x^b,\lambda) ~\text{~on~} \partial_0 \calg{M},
\end{eqnarray}
where
$$
\lambda \in \bbbb{R}^m,
\ \ \ 
1 \le a,b,c \le d,
\ \ \ 
1 \le i,j,k \le n,
$$
$$
A : \calg{M} \times \bbbb{R}^n \times \bbbb{R}^{nd} \times \bbbb{R}^m
            \mapsto \bbbb{R}^{nd},
\ \ \ 
B : \calg{M} \times \bbbb{R}^n \times \bbbb{R}^{nd} \times \bbbb{R}^m
            \mapsto \bbbb{R}^n,
$$
$$
C : \partial_1 \calg{M} \times \bbbb{R}^n \times \bbbb{R}^{nd} \times \bbbb{R}^m
            \mapsto \bbbb{R}^n,
\ \ \ 
E : \partial_0 \calg{M} \times \bbbb{R}^m \mapsto \bbbb{R}^n.
$$
The divergence-form system~(\ref{eqn:str1})--(\ref{eqn:str3}), together
with the boundary conditions, can be viewed as 
an operator equation of the form
\begin{equation}
   \label{eqn:parm_pre}
G(u,\lambda) = 0,
\ \ \ \ \
G : \calg{B}_1 \times \bbbb{R}^m \mapsto \calg{B}_2^*,
\end{equation}
for some Banach spaces $\calg{B}_1$ and $\calg{B}_2$, where $\calg{B}_2^*$
denotes the dual space of $\calg{B}_2$.
Analysis and numerical techniques often require the Gateaux-linearization 
operator $D_uG(u) \in \calg{L}(\calg{B}_1,\calg{B}_2^*)$.
If $D_uG(u_0,\lambda_0)$ is a linear homeomorphism from $\calg{B}_1$ to
$\calg{B}_2^*$, then the Implicit Function Theorem guarantees that there is
a neighborhood of $(\lambda_0,u_0) \in \bbbb{R}^m \times \calg{B}_1$ 
containing regular solutions to~(\ref{eqn:parm_pre}).
If $D_uG(u_0,\lambda_0)$ is singular, so that the Implicit Function
Theorem does not apply, then the standard approach is to expand the
solution spaces in such a way that the expanded problem has a regular solution.
This is called the {\em augmented} or {\em bordered system} approach to
handling folds and bifurcations, and is the basis for sophisticated numerical
path-following algorithms~\cite{Kell87,Kell92}.
In the case of a simple limit point (or {\em fold}),
where $D_uG(u_0)$ is a Fredholm
operator of $\calg{B}_1$ into $\calg{B}_2^*$ with index $m$,
then the augmented system approach involves simply adding a set of $m$ linear 
constraints to the original system, producing:
\begin{equation}
   \label{eqn:parm}
F(u,\lambda,s) = \left[ \begin{array}{c}
                      G(u,\lambda) \\
                      N(u,\lambda,s)
                      \end{array}
                \right] = 0,
\ \ \ \
F: \calg{B}_1 \times \bbbb{R}^m \times \bbbb{R}^m
     \mapsto \calg{B}_2^* \times \bbbb{R}^m.
\end{equation}
If the augmentation function
$N(u,\lambda,s) : \calg{B}_1 \times \bbbb{R}^m \times \bbbb{R}^m
         \mapsto \bbbb{R}^m$
is chosen correctly, then the linearization operator
$DF \in \calg{L}(\calg{B}_1 \times \bbbb{R}^m,
                 \calg{B}_2^* \times \bbbb{R}^m)$ for the whole system,
which can be written as
\begin{equation}
    \label{eqn:parm_linearized}
DF(u,\lambda,s) = \left[ \begin{array}{ll}
            D_u G(u,\lambda)   & D_{\lambda} G(u,\lambda)   \\ 
            D_u N(u,\lambda,s) & D_{\lambda} N(u,\lambda,s) \\ 
            \end{array} \right],
\end{equation}
becomes a homeomorphism again.

Our interest here is primarily in coupled systems
of one or more scalar field equations
and one or more $d$-vector field equations, possibly
augmented as in~(\ref{eqn:parm})--(\ref{eqn:parm_linearized}).
The unknown $n$-vector $u^i$ then in general consists
of $n_s$ scalars and $n_v$ $d$-vectors, so that $n = n_s + n_v \cdot d$.
To allow the $n$-component system~(\ref{eqn:str1})--(\ref{eqn:str3})
to be treated notationally as if it were a single $n$-vector equation,
it will be convenient to introduce the following notation for the
unknown vector $u^i$ and for the metric of the product space of
scalar and vector components of $u^i$:
\begin{equation}
   \label{eqn:product_metric}
\calg{G}_{ij} = \left[ \begin{array}{ccc}
         g_{ab}^{(1)} &        & 0      \\
                      & \ddots &              \\
         0            &        & g_{ab}^{(n_e)} \\
         \end{array} \right],
~~~~~~
u^i = \left[ \begin{array}{c}
         u^a_{(1)}      \\
         \vdots         \\
         u^a_{(n_e)}      \\
         \end{array} \right],
~~~~~n_e=n_s+n_v.
\end{equation}
If $u^a_{(k)}$ is a $d$-vector
we take $g_{ab}^{(k)} = g_{ab}$; if $u^a_{(k)}$ is a scalar
we take $g_{ab}^{(k)}=1$.

\subsubsection{Weak formulations}

The weak form of~(\ref{eqn:str1})--(\ref{eqn:str3}) is obtained by taking the
$L^2$-based duality pairing between a vector $v^j$
(vanishing on $\partial_0 \calg{M}$) lying in a product space
of scalars and tensors, and the residual of the
tensor system~(\ref{eqn:str1}), yielding:
\begin{equation}
   \label{eqn:weak_almost}
\int_{\calg{M}} \calg{G}_{ij} \left( B^i - A^{ia}_{\ \ ;a} \right) v^j 
  ~dx = 0.
\end{equation}
Due to the definition of $\calg{G}_{ij}$ in~(\ref{eqn:product_metric}),
this is simply a sum of integrals of scalars, each of which is a contraction
of the type appearing on the left side in~(\ref{eqn:divergence}).
Using then~(\ref{eqn:divergence}) and~(\ref{eqn:str2}) together
in~(\ref{eqn:weak_almost}),
and recalling that $v^i=0$ on $\partial_0 \calg{M}$
satisfying~(\ref{eqn:bndry}), yields
\begin{equation}
    \label{eqn:weak_general}
    \int_{\calg{M}} \calg{G}_{ij} A^{ia} v^j_{~;a}~dx
  + \int_{\calg{M}} \calg{G}_{ij} B^i v^j~dx
  + \int_{\partial_1 \calg{M}} \calg{G}_{ij} C^i v^j~ds
  = 0.
\end{equation}
Equation~(\ref{eqn:weak_general}) leads to a
covariant weak formulation of the problem:
\begin{equation}
   \label{eqn:weak}
\text{Find}~ u \in \bar{u} + \calg{B}_1 ~\text{s.t.}~
     \langle F(u),v \rangle = 0,
   \ \ \forall~ v \in \calg{B}_2,
\end{equation}
for suitable Banach spaces of functions $\calg{B}_1$ and $\calg{B}_2$,
where the nonlinear weak form $\langle F(\cdot),\cdot \rangle$
can be written as:
\begin{equation}
\label{eqn:weakForm}
  \langle F(u),v \rangle
    = \int_{\calg{M}} \calg{G}_{ij} (A^{ia} v^j_{~;a} + B^i v^j)~dx
           + \int_{\partial_1 \calg{M}} \calg{G}_{ij} C^i v^j ~ds.
\end{equation}
The notation $\langle w,v \rangle$ will represent the duality pairing
of a function $v$ in a Banach space $\calg{B}$ with
a bounded linear functional (or {\em form}) $w$
in the dual space $\calg{B}^*$.
Depending on the particular function spaces involved,
the pairing may be thought of as coinciding with the $L^2$-inner-product
through the Riesz Representation Theorem~\cite{Yosi80}.
The affine shift tensor $\bar{u}$ in~(\ref{eqn:weak})
represents the essential or Dirichlet part
of the boundary condition if there is one; the existence of $\bar{u}$ such that 
$E = \bar{u} |_{\partial_0 \calg{M}}$ in the sense
of the Trace operator is guaranteed by the Trace Theorem
for Sobolev spaces on manifolds with boundary~\cite{Wlok92},
as long as $E^i$ in~(\ref{eqn:str3}) and $\partial_0 \calg{M}$ are smooth
enough.
If normalization is required as in~(\ref{eqn:parm}),
then the weak formulation also reflects the normalization.


\subsubsection{Sobolev spaces of tensors}

The Banach spaces which arise naturally as solution spaces for the
class of nonlinear elliptic systems in~(\ref{eqn:weak}) are product
spaces of the Sobolev spaces $W_{0,D}^{k,p}(\calg{M})$.
This is due to the fact that under suitable growth conditions
on the nonlinearities in $F$, it can be shown
(essentially by applying the H\"{o}lder inequality)
that there exists $p_k,q_k,r_k$ satisfying $1 < p_k,q_k,r_k < \infty$
such that the choice
\begin{equation}
   \label{eqn:banach1}
\calg{B}_1 = W_{0,D}^{1,r_1}
      \times \cdots \times W_{0,D}^{1,r_{n_e}},
~~~~
\calg{B}_2 = W_{0,D}^{1,q_1}
      \times \cdots \times W_{0,D}^{1,q_{n_e}},
\end{equation}
\begin{equation}
   \label{eqn:banach2}
\frac{1}{p_k} + \frac{1}{q_k} = 1,
~~~~
r_k \ge \text{min} \{p_k,q_k\},
~~~~
k=1,\ldots,n_e,
\end{equation}
ensures $\langle F(u),v \rangle$ in~(\ref{eqn:weakForm})
remains finite for all arguments~\cite{FuKu80}.

The Sobolev spaces are also fundamental to the theory of the finite element
method, which is based essentially on subspace projection and best
approximation.
The Sobolev spaces $W^{k,p}(\calg{M})$ of tensors on manifolds, and the
various subspaces such as $W_{0,D}^{k,p}(\calg{M})$ which we will need to
make use of later in the paper, can be defined
as follows (cf.~\cite{HaEl73,Aubi82,Hebe91} for more complete discussions).
For a type $(r,s)$-tensor
$T^{a_1a_2 \cdots a_r}_{~~~~~~~~~b_1b_2\cdots b_s} = T^I_{~J}$,
where $I$ and $J$ are (tensor) multi-indices satisfying
$|I|=r$, $|J|=s$, define
\begin{equation}
   \label{eqn:l2_ext}
|T^I_{~J}| = \left( T^I_{~J} T^L_{~M} g_{IL}g^{JM} \right)^{1/2}.
\end{equation}
Here, $g_{IJ}$ and $g^{IJ}$ are generated from the
Riemannian $d$-metric $g_{ab}$ on $\calg{M}$ as follows:
\begin{equation}
   \label{eqn:metric_lebesgue}
g_{IJ} = g_{ab}g_{cd} \cdots g_{pq},
~~~~~~~~
g^{IJ} = g^{ab}g^{cd} \cdots g^{pq},
\end{equation}
where $|I|=|J|=m$, producing $m$ terms in each product.
Expression~(\ref{eqn:l2_ext}) is just an extension of the
Euclidean $l^2$-norm for vectors
in $\bbbb{R}^d$.
For example, in the case of a 3-manifold,
taking $|I|=1$, $|J|=0$, $g_{ab} = \delta_{ab}$, gives:
$$
|T^I_{~J}| = |T^a| = \left( T^a T^b g_{ab} \right)^{1/2}
= \left( T^a T^b \delta_{ab} \right)^{1/2}
= \| T^a \|_{l^2(\bbbb{R}^3)}.
$$

Covariant (distributional) differentiation of order $m=|K|$
(for some tensor multi-index $K$) using a connection generated by
$g_{ab}$, or generated by possibly a different metric, is denoted
as any of:
\begin{equation}
   \label{eqn:metric_distribution}
D^m T^I_{~J} = D_K T^I_{~J}
             = T^I_{~J;K},
\end{equation}
where $m$ should not be confused with a tensor index.
Employing the measure $dx$ on $\calg{M}$ defined
in~(\ref{eqn:volume_element}),
the $L^p$-norm of a tensor on $\calg{M}$ is defined as:
\begin{equation}
   \label{eqn:lp_norm}
   \|T^I_{~J}\|_{L^p(\calg{M})} 
     = \left( \int_{\calg{M}} | T^I_{~J} |^p ~dx \right)^{1/p},
\end{equation}
and the resulting $L^p$-spaces for $1 \le p < \infty$ are defined as:
\begin{equation}
   \label{eqn:lp}
L^p(\calg{M}) = \left\{~ T^I_{~J} ~|~
   \|T^I_{~J}\|_{L^p(\calg{M})} < \infty ~\right\}.
\end{equation}
When discussing the properties of $L^p$-functions over a manifold $\calg{M}$
we will use the notation {\em a.e.}, meaning that the property is understood
to hold ``almost everywhere'' in the sense of Lebesgue measure.
We will at times need to make use of the (extended) H\"{o}lder and Minkowski
inequalities for tensors in $L^p$-spaces:
\begin{equation}
   \label{eqn:holder}
    \|U^I_{~J} W^J_{~I} \|_{L^r(\calg{M})}
\le \|U^I_{~J}\|_{L^p(\calg{M})} 
    \|W^I_{~J}\|_{L^q(\calg{M})},
\end{equation}
\begin{equation}
   \label{eqn:minkowski}
    \|U^I_{~J} + V^I_{~J} \|_{L^p(\calg{M})}
\le \|U^I_{~J} \|_{L^p(\calg{M})}
  + \|V^I_{~J} \|_{L^p(\calg{M})},
\end{equation}
which hold when $U^I_{~J}, V^I_{~J} \in L^p(\calg{M})$,
$W^I_{~J} \in L^q(\calg{M})$, $1/p + 1/q = 1/r$,
$1 \le p,q,r < \infty$.
The H\"{o}lder inequality also extends to the case $p=1$, $q=\infty$, $r=1$,
where
$\|U^I_{~J}\|_{L^{\infty}(\calg{M})}
    = \text{ess}~\sup_{x \in \calg{M}} |U^I_{~J}(x)|$.

The Sobolev semi-norm of a tensor is defined through~(\ref{eqn:lp_norm}) as:
\begin{equation}
    \label{eqn:h1_norm_semi}
|T^I_{~J}|_{W^{m,p}(\calg{M})}^p
= \sum_{|K|=m} \| T^I_{~J;K} \|_{L^p(\calg{M})}^p,
\end{equation}
and the Sobolev norm is subsequently defined using~(\ref{eqn:h1_norm_semi}) as:
\begin{equation}
   \label{eqn:h1_norm}
\|T^I_{~J}\|_{W^{k,p}(\calg{M})}
= \left( \sum_{0\le m\le k}
     | T^I_{~J} |_{W^{m,p}(\calg{M})}^p \right)^{1/p}.
\end{equation}
The resulting Sobolev spaces of tensors are then defined
using~(\ref{eqn:h1_norm}) as:
\begin{equation}
    \label{eqn:h1}
W^{k,p}(\calg{M}) = \left\{~ T^I_{~J} ~|~
   \|T^I_{~J}\|_{W^{k,p}(\calg{M})} < \infty ~\right\},
\end{equation}
\begin{equation}
    \label{eqn:h1_0}
W^{k,p}_0(\calg{M}) 
     = \left\{~ \text{Completion~of}~ C_0^\infty(\calg{M})
    \text{~w.r.t.~} \|\cdot\|_{W^{k,p}(\calg{M})} ~\right\},
\end{equation}
where $C_0^{\infty}(\calg{M})$ is the space of $C^{\infty}$-tensors with
compact support in $\calg{M}$.
The space $W^{k,p}_0(\calg{M})$ in~(\ref{eqn:h1_0}) is a special case of
$W^{k,p}_{0,D}(\calg{M})$, which can be characterized as:
\begin{equation}
   \label{eqn:h1_0d}
W^{k,p}_{0,D}(\calg{M}) = \left\{ T^I_{~J} \in W^{k,p} ~|~
             \text{tr}~T^I_{~J;K} = 0 ~\text{on}~ \partial_0 \calg{M},
   |K| \le k-1 \right\}.
\end{equation}
Note that if the metric used to define covariant
differentiation in~(\ref{eqn:metric_distribution}) is taken to be different
from the metric $g_{ab}$ used in~(\ref{eqn:metric_lebesgue}),
it can still be shown that the norms generated
by~(\ref{eqn:h1_norm}) are equivalent,
so that the resulting Sobolev spaces have exactly the same
topologies~\cite{HaEl73}.

The Hilbert space special case of $p=2$ is given a simplified notation:
\begin{equation}
   \label{eqn:hk}
H^k(\calg{M}) = W^{k,2}(\calg{M}),
\end{equation}
with the same convention used for the various subspaces of $H^k(\calg{M})$
such as $H_0^k(\calg{M})$ and $H_{0,D}^k(\calg{M})$.
The norm on $H^k(\calg{M})$ defined above is then actually induced
by an inner-product as follows:
$\| T^I_{~J} \|_{H^k(\calg{M})}
   = (T^I_{~J},T^I_{~J})_{H^k(\calg{M})}^{1/2}$,
where
\begin{equation}
    \label{eqn:l2_innerproduct}
(T^I_{~J},S^I_{~J})_{L^2(\calg{M})}
     = \int_{\calg{M}}
      T^I_{~J} S^L_{~M} g_{IL}g^{JM} ~dx,
\end{equation}
and where
\begin{equation}
    \label{eqn:hk_innerproduct}
(T^I_{~J},S^I_{~J})_{H^k(\calg{M})}
 = \sum_{0\le |K| \le k} ( T^I_{~J;K}, S^I_{~J;K} )_{L^2(\calg{M})}.
\end{equation}

Finally, note that Sobolev trace spaces of tensors living on boundary
submanifolds as
needed for discussing boundary-value problems can be defined under some
smoothness assumptions on the boundary,
and spaces based on fractional-order differentiation
(take $k \in \bbbb{R}$ in the discussion above)
can be defined in several different ways (cf.~\cite{Adam78,Aubi82}).

\subsection[Adaptive FE methods
            for nonlinear elliptic systems]
           {Adaptive multilevel finite element methods
            for nonlinear elliptic systems}

A {\em Petrov-Galerkin} approximation of the solution to
(\ref{eqn:weak}) is the solution to the following subspace problem:
\begin{equation}
\text{Find}~ u_h \in \bar{u}_h + U_h \subset \calg{B}_1 ~\text{s.t.}~
     \langle F(u_h),v \rangle = 0,
   \ \ \forall~ v \in V_h \subset \calg{B}_2,
   \label{eqn:galerkin}
\end{equation}
for some chosen subspaces $U_h$ and $V_h$,
where $\text{dim}(U_h) = \text{dim}(V_h) = n$.
A {\em Galerkin} approximation refers to the case that $U_h = V_h$.
A {\em finite element} method is a Petrov-Galerkin or Galerkin
method in which the subspaces $U_h$ and $V_h$ are chosen
to have the extremely simple form of continuous piecewise polynomials
with local support, defined over a disjoint covering of the domain 
manifold $\calg{M}$ by {\em elements}.
A global $C^0$-basis on the manifold may be defined element-wise
from local basis functions defined on a reference simplex by use of the
chart structure provided with the manifold.
For example, in the case of continuous piecewise linear polynomials 
on 2-simplices (triangles) or 3-simplices (tetrahedra),
the reference element is equipped with the usual basis as
shown in Figure~\ref{fig:bases}.
\begin{figure}
   \label{fig:bases}
\begin{center}
\begin{picture}(300,60)
\put(50,10){
    \mbox{\myfigpdf{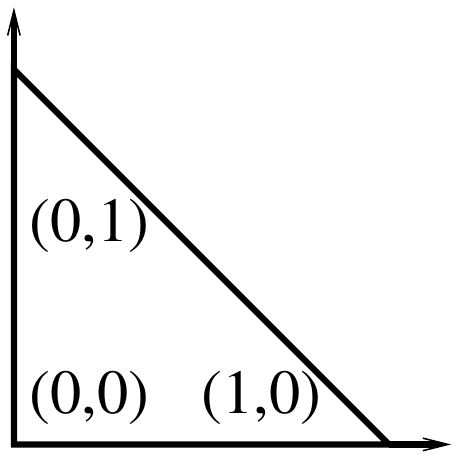}{0.8in}}
}
\put(150,30){
    \mbox{\small
    $
    \begin{array}{ccl}
    \tilde{\phi}_0(\tilde{x},\tilde{y}) & = & 1 - \tilde{x} - \tilde{y} \\
    \tilde{\phi}_1(\tilde{x},\tilde{y}) & = & \tilde{x} \\
    \tilde{\phi}_2(\tilde{x},\tilde{y}) & = & \tilde{y}
    \end{array}
    $
    }
}
\end{picture}
\begin{picture}(300,60)
\put(50,0){
    \mbox{\myfigpdf{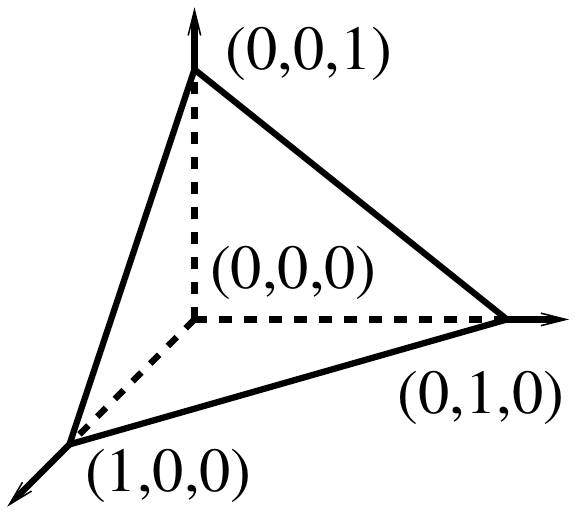}{0.8in}}
}
\put(150,25){
    \mbox{\small
    $
    \begin{array}{ccl}
    \tilde{\phi}_0(\tilde{x},\tilde{y},\tilde{z}) & = & 1
                       - \tilde{x} - \tilde{y} - \tilde{z} \\
    \tilde{\phi}_1(\tilde{x},\tilde{y},\tilde{z}) & = & \tilde{y} \\
    \tilde{\phi}_2(\tilde{x},\tilde{y},\tilde{z}) & = & \tilde{x} \\
    \tilde{\phi}_3(\tilde{x},\tilde{y},\tilde{z}) & = & \tilde{z}
    \end{array}
    $
    }
}
\end{picture}
\end{center}
\caption{Canonical linear references bases.}
\end{figure}
The chart structure provides mappings between the elements
contained in each coordinate patch and the unit simplex.
If the manifold domain can be triangulated exactly with simplex elements
(possibly as a polyhedral approximation to an underlying smooth surface),
then the coordinate transformations are simply affine transformations.
In this sense, finite element methods are by their very nature
defined in a chart-wise manner.
Algorithms for smooth ($C^k$) 2-surface representations using manifolds
have been considered recently in~\cite{GrHu95,Grim96};
some interesting related work appeared in~\cite{DaHm89,DaMi84}.

Due to the non-smooth behavior of their derivatives along simplex vertices,
edges, and faces in the disjoint simplex covering of $\calg{M}$, such
continuous piecewise polynomial bases clearly do not span a subspace
of $\calg{C}^1(\calg{M})$; however, one can show~\cite{Ciar78} that in fact:
$$
V_h = \text{span}\{\phi_1,\ldots,\phi_n\} \subset W^{1,p}_{0,D}(\calg{M}),
~~ \calg{M} \subset \bbbb{R}^d,
$$
so that continuous, piecewise defined, low-order
polynomial spaces do in fact form a subspace of the solution space
to the weak formulation of the class of second order elliptic equations of
interest.
Making then the choice
$
U_h = \text{span}\{\phi_1,\phi_2,\ldots,\phi_n\},
$
$
V_h = \text{span}\{\psi_1,\psi_2,\ldots,\psi_n\},
$
equation~(\ref{eqn:galerkin}) in the case of a scalar unknown
reduces to a set of $n$ nonlinear algebraic 
relations (implicitly defined) for the $n$ coefficients $\{\alpha_j\}$
in the expansion
\begin{equation}
   \label{eqn:soln}
u_h = \bar{u}_h + \sum_{j=1}^n \alpha_j \phi_j,
\end{equation}
with suitable modification for a vector unknown.
In particular, regardless of the complexity of the form
$\langle F(u),v \rangle$,
as long as we can evaluate it for given $u$ and $v$, then we
can evaluate the discrete nonlinear residual of the finite element 
approximation $u_h$ as:
$$
    F_i = \langle F(\bar{u}_h + \sum_{j=1}^n \alpha_j \phi_j),\psi_i \rangle,
         \ \ \ \ i=1,\ldots,n.
$$
Since the form $\langle F(u),v \rangle$
involves an integral in this setting, if we
employ quadrature then we can simply sample the integrand at quadrature points;
this is a standard technique in finite element technology.
Given the local support nature of the functions $\phi_j$ and $\psi_i$,
all but a small constant number of terms in the sum 
$\sum_{j=1}^n \alpha_j \phi_j$ are zero at a particular spatial point
in the domain, so that the residual $F_i$ is inexpensive to evaluate when
quadrature is employed.

The two primary issues in using the approximation method are:
\begin{enumerate}
\item Functionals $\calg{E}(u-u_h)$ of the error $u-u_h$ (such as norms), and
\item Complexity of solving the $n$ nonlinear algebraic equations.
\end{enumerate}
The first of these issues represents the core of finite element approximation
theory, which itself rests on the results of classical approximation theory.
Classical references to both topics include~\cite{Ciar78,DeLo93,Davi63}.
The second issue is addressed by the complexity theory of direct and
iterative solution methods for sparse systems of linear and nonlinear
algebraic equations, cf.~\cite{Hack94,OrRh70}.

\subsubsection{Approximation quality: error estimation and adaptive methods}

{\em A priori} error analysis for the finite element method for addressing
the first issue is now a very well-understood subject~\cite{Ciar78,BrSc94}.
Much activity has recently been centered around
{\em a posteriori} error estimation and the use of error indicators based
on such estimates in conjunction with adaptive mesh
refinement algorithms~\cite{PLTMG,BaRh78a,BaRh78b,Verf94,Verf96,XuZh97}.
These indicators include weak and strong residual-based 
indicators~\cite{BaRh78a,BaRh78b,Verf94}, indicators based on
the solution of local problems~\cite{BaSm93,BaWe85}, and indicators
based on the solution of global (but linearized) adjoint or {\em dual}
problems~\cite{EHM2001}.
The challenge for a numerical method is to be as efficient as possible,
and {\em a posteriori} estimates are a basic tool in deciding which
parts of the solution require additional attention.
While the majority of the work on {\em a posteriori} estimates and indicators
has been for linear problems, nonlinear extensions are possible through
linearization theorems~(cf.~\cite{Verf94,Verf96}).
The typical solve-estimate-refine structure in simplex-based adaptive finite
element codes exploiting these {\em a posteriori} indicators is illustrated
in Algorithm~\ref{alg:adapt}.
\begin{algorithmX}
   \label{alg:adapt}
{\em (Adaptive multilevel finite element approximation)}
\begin{itemize}
\item While ($\calg{E}(u - u_h)$ is ``large'') do:
    \begin{enumerate}
    \item Find $u_h \in \bar{u}_h + U_h \subset \calg{B}_1$ such that $\langle F(u_h),v \rangle=0, ~\forall~ v \in V_h \subset \calg{B}_2$.
    \item Estimate $\calg{E}(u-u_h)$ over each element.
    \item Initialize two temporary simplex lists as empty: $Q1=Q2=\emptyset$.
    \item Simplices which fail an indicator test using equi-distribution of the chosen error functional $\calg{E}(u-u_h)$ are placed on the ``refinement'' list~$Q1$.
    \item Bisect all simplices in~$Q1$ (removing them from~$Q1$), and place any nonconforming simplices created on the list~$Q2$.
    \item $Q1$ is now empty; set $Q1$ = $Q2$, $Q2 = \emptyset$.
    \item If $Q1$ is not empty, goto (5).
    \end{enumerate}
 \item End While.
\end{itemize}
\end{algorithmX}
The conformity loop (5)--(7), required to produce a globally ``conforming''
mesh (described below) at the end of a refinement step, is guaranteed to 
terminate in a finite number of steps (cf.~\cite{Riva84,Riva91}),
so that the refinements remain local.
Element shape is crucial for approximation quality; the bisection 
procedure in step~(5) is guaranteed to produce nondegenerate families 
if the longest edge is bisected in two dimensions~\cite{RoSt75,Styn80}, 
and if marking or homogeneity methods are used 
in three dimensions~\cite{AMP97,Mukh96,Bans91a,Bans91b,LiJo95,Maub95}.
Whether longest edge bisection is nondegenerate in three dimensions apparently
remains an open question.
Figure~\ref{fig:refine} shows a single subdivision of a 2-simplex or a
3-simplex using either 4-section (left-most figure),
8-section (fourth figure from the left),
or bisection (third figure from the left, and the right-most figure).
\begin{figure}
\begin{center}
\mbox{\myfigpdf{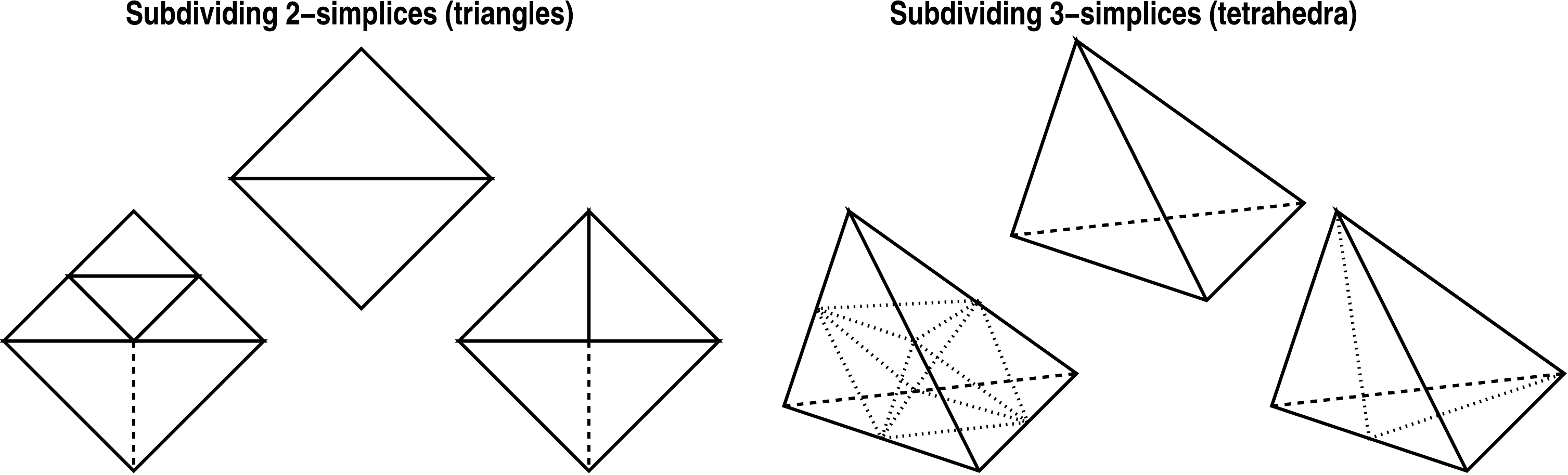}{1.3in}}
\end{center}
   \caption{Refinement of 2- and 3-simplices using 4-section,
            8-section, and bisection.}
   \label{fig:refine}
\end{figure}
The paired triangle in the 2-simplex case of Figure~\ref{fig:refine} 
illustrates the nature of conformity and its violation during refinement.
A globally conforming simplex mesh is defined as a collection of simplices
which meet only at vertices and faces; for example, removing the dotted
bisection in the third group from the left in Figure~\ref{fig:refine}
produces a non-conforming mesh.
Non-conforming simplex meshes create several theoretical as well as practical
implementation difficulties;
while the queue-swapping presented in Algorithm~\ref{alg:adapt} above
is a feature unique to \MC~(see Section~\ref{sec:mc}),
an equivalent approach is taken in \PLTMG\ ~\cite{PLTMG}
and similar packages~\cite{Mukh96,BER95,Bey96,BBJL98}.

\subsubsection{Computational complexity: solving linear and nonlinear systems}

Addressing the complexity of Algorithm~\ref{alg:adapt},
Newton-like methods as illustrated in Algorithm~\ref{alg:newton}
are often the most effective.
\begin{algorithmX}
   \label{alg:newton}
{\em (Damped-inexact-Newton)}
\begin{itemize}
 \item Let an initial approximation $u$ be given.
 \item While ($|\langle F(u),v \rangle| > \epsilon$ for any $v$) do:
    \begin{enumerate}
    \item Find $w$ such that $\langle DF(u)w,v \rangle = - \langle F(u),v \rangle + r, ~\forall~ v$.
    \item Set $u = u + \lambda w$.
    \end{enumerate}
 \item End While.
\end{itemize}
\end{algorithmX}
The bilinear form $\langle DF(u)w,v \rangle$ which appears
in Algorithm~\ref{alg:newton} is simply
the (Gateaux) linearization of the nonlinear form $\langle F(u),v \rangle$,
defined formally as:
$$
\langle DF(u)w,v \rangle=\left.\frac{d}{d\epsilon}
    \langle F(u+\epsilon w),v \rangle \right|_{\epsilon=0}.
$$
This form is easily computed from most nonlinear forms
$\langle F(u),v \rangle$ which
arise from second order nonlinear elliptic problems, although the
calculation can be tedious in some cases (the example we consider
later in the paper is in this category).
   The possibly nonzero ``residual'' term $r$ is to allow for inexactness
in the linearization solve for efficiency, which is quite effective in many
cases (cf.~\cite{BaRo82,DES82,EiWa92}).
The parameter $\lambda$ brings robustness to the 
algorithm~\cite{EiWa92,BaRo80,BaRo81}.
If folds or bifurcations are present, then the iteration is modified 
to incorporate path-following~\cite{Kell87,BaMi89}.

As was the case for the Petrov-Galerkin discretized nonlinear residual
$\langle F(\cdot),\cdot \rangle$, the matrix
representing the bilinear form in the Newton iteration is easily assembled,
regardless of the complexity of the bilinear form 
$\langle DF(\cdot)\cdot,\cdot \rangle$.
In particular, the matrix equation for $w = \sum_{j=1}^n \beta_j \phi_j$
has the form:
$$
A U = F,
\ \ \ \ \ 
U_i = \beta_i,
$$
where
$$
A_{ij} = \langle DF( \bar{u}_h + \sum_{k=1}^n \alpha_k \phi_k )\phi_j,
         \psi_i \rangle,
\ \ \ \ \ 
F_i = \langle F( \bar{u}_h + \sum_{j=1}^n \alpha_j \phi_j),\psi_i \rangle.
$$
As long as the integral-based forms
$\langle F(\cdot),\cdot \rangle$ and $\langle DF(\cdot)\cdot,\cdot \rangle$
can be evaluated at individual points in the domain, then quadrature can
be used to build the Newton equations, regardless of the complexity of
the forms.
This is one of the most powerful features of the finite element method.
It should be noted that there is a subtle difference between the approach
outlined here (typical for a nonlinear finite element approximation) and
that usually taken when applying a Newton-iteration to a nonlinear finite
difference approximation.
In particular, in the finite difference setting the discrete equations
are linearized explicitly by computing the jacobian of the system of 
nonlinear algebraic equations.
In the finite element setting, the commutativity of linearization and
discretization is exploited; the Newton iteration is actually
performed in function space, with discretization occurring
``at the last moment'' in Algorithm~\ref{alg:newton} above.

   It can be shown that the Newton iteration above
is dominated by the computational 
complexity of solving the $n$ linear algebraic equations 
in each iteration (cf.~\cite{BaRo82,Hack85}).
   Multilevel methods are the only known provably optimal or nearly optimal
methods for solving these types of linear algebraic equations resulting
from discretizations of a large class of general linear elliptic 
problems~\cite{Hack85,BaDu81,Xu92a}.
   Unfortunately, the need to accurately represent complicated
PDE coefficient, domain features, and domain boundaries with an adapted
mesh requires the use of very fine mesh simply to describe the complexities
of the problem, which often precludes
the simple solve-estimate-refine approach in Algorithm~\ref{alg:adapt}.
   In Section~\ref{sec:mc_solvers} we describe the algebraic multilevel
approach we take in the \MC\ implementation to adress this,
similar to that taken in~\cite{CSZ94,CGZ97,RuSt87,VMB94}.

\subsection{Residual-based {\em a posteriori} error indicators}
  \label{sec:indicators}

There are several approaches to adaptive error control, although the
approaches based on {\em a posteriori} error estimation are usually the
most effective and most general.
While most existing work on {\em a posteriori} estimates has been for linear 
problems, extensions to the nonlinear case can be made through linearization.
For example, consider the nonlinear problem in~(\ref{eqn:parm}), which
we will write as follows (ignoring the parameters for simplicity):
\begin{equation}
    \label{eqn:abstract_strong}
F(u) = 0, ~~~ F \in C^1(\calg{B}_1,\calg{B}_2^*),
~~~ \calg{B}_1, \calg{B}_2 ~\text{Banach~spaces},
\end{equation}
and a discretization:
\begin{equation}
    \label{eqn:abstract_strong_discrete}
F_h(u_h) = 0, ~~~ F_h \in C^0(U_h,V_h^*),
~~~
U_h \subset \calg{B}_1, 
~~~
V_h \subset \calg{B}_2.
\end{equation}
The nonlinear residual $F(u_h)$ can be used to estimate the error
$\|u-u_h\|_{\calg{B}_1}$, through the use of a
{\em linearization theorem}~\cite{LiRh94,Verf94}.
An example of such a theorem due to Verf\"urth is the following.
\begin{theorem} \cite{Verf94}
   \label{thm:verf}
Let $u\in X$ be a regular solution of $F(u)=0$, 
so that the Gateaux derivative $DF(u)$ is a linear
homeomorphism of $\calg{B}_1$ onto $\calg{B}_2^*$.
Assume $DF$ is Lipschitz continuous at $u$, so that there exists $R_0$ such 
that
$$
\gamma = \sup_{u_h \in B(u,R_0)} \frac{
   \| DF(u) - DF(u_h) \|_{\calg{L}(\calg{B}_1,\calg{B}_2^*)}}
            {\|u-u_h\|_{\calg{B}_1}} < \infty.
$$
Let $R=\min \{R_0, \gamma^{-1} \|DF(u)^{-1}\|_{\calg{L}(\calg{B}_2^*,\calg{B}_1)}, 2 \gamma^{-1} \|DF(u)\|_{\calg{L}(\calg{B}_1,\calg{B}_2^*)} \}$.
Then for all $u_h \in B(u,R)$,
\begin{equation}
   \label{eqn:verf}
   C_1 \|F(u_h)\|_{\calg{B}_2^*}
   \le \| u - u_h \|_{\calg{B}_1}
   \le C_2 \|F(u_h)\|_{\calg{B}_2^*},
\end{equation}
where $C_1=\frac{1}{2}\|DF(u)\|^{-1}_{\calg{L}(\calg{B}_1,\calg{B}_2^*)}$
and $C_2=2 \| DF(u)^{-1} \|_{\calg{L}(\calg{B}_2^*,\calg{B}_1)}$.
\end{theorem}
\begin{proof}
See~\cite{Verf94}.\qed
\end{proof}
The effect of linearization is swept under the rug somewhat by
the choice of $R$ sufficiently small, where $R$ is the radius of an open
ball in $\calg{B}_1$ about $u$, denoted as $B(u,R)$ in the theorem above.
One then ignores the
factors in~(\ref{eqn:verf}) involving the linearization $DF(u)$ and
its inverse, and focuses on two-sided estimates for the nonlinear residual
$\| F(u_h) \|_{\calg{B}_2^*}$ appearing on each side of~(\ref{eqn:verf}).
Since one typically constructs highly refined meshes where needed,
such local linearized estimates are thought to reasonable, although
much evidence to the contrary has been assembled by the dual-problem
error indicator community (see the discussion later in this section).
Note that $\|F(u_h)\|_{\calg{B}_2^*}$ can be estimated in 
different ways, including
\begin{enumerate}
\item Approximation by $\|F_h(u_h)\|_{\calg{B}_2^*}$ (residual estimates)~\cite{Verf94,Verf96,LiRh94,LiRh96},
\item Solution of local Neumann (or Dirichlet) problems~\cite{BaSm93,BaWe85}.
\end{enumerate}
The approaches can be shown to be essentially equivalent (up to constants;
cf.~\cite{Verf94,BaSi95}).
For reasons of efficiency, estimation by strong residuals is often used
rather than the solution of local problems in the case of elliptic systems
and/or in the setting of three-dimensional problems.
In particular, one employs the linearization theorem above, together with
some derived (and computable) upper and lower bounds on the nonlinear residual
$\|F(u_h)\|_{\calg{B}_2^*}$ given by the following pair of inequalities:
$$
C_3 \le \|F(u_h)\|_{\calg{B}_2^*} \le C_4.
$$
While it is clear that the upper bound $C_4$ is the key to bounding the
error, the lower bound $C_3$ can also be quite useful; it can help to ensure
that the adaptive procedure doesn't do too much work by over-refining an
area where it is unnecessary.
The effectiveness of an adaptive finite element code can
hinge on the implementation details of the estimator,
and implementing it efficiently can be quite an
art form (cf.~\cite{Verf94,BaSm93,BaWe85}).

We now consider the first two of these approaches in more detail.
First, we derive a strong residual-based {\em a posteriori} error indicator
for general Petrov-Galerkin approximations~(\ref{eqn:galerkin}) to the
solutions of general nonlinear elliptic systems of tensors of the
form~(\ref{eqn:str1})--(\ref{eqn:str3}).
The analysis involves primarily the weak
formulation~(\ref{eqn:weak})--(\ref{eqn:weakForm}).
Our derivation follows closely that of Verf\"{u}rth~\cite{Verf94,Verf96}
in the flat, Cartesian case.
In the next section, we will consider the second alternative, namely
an indicator based on a duality approach.

It should be noted that while the discussions throughout the paper are
generally valid for domains which are connected compact Riemannian manifolds
with Lipschitz continuous boundaries, several of the results we will
need to employ here have only been shown to hold in the case of bounded
$2$- and $3$-manifolds with smooth boundaries, with an atlas consisting of
only one chart (i.e., bounded open subsets of $\bbbb{R}^2$ and $\bbbb{R}^3$).
Examples are convex polyhedra in $\bbbb{R}^d$, which automatically
satisfy the Lipschitz continuity assumption.
The extensions of some of these results from open sets to Riemannian manifolds 
are not immediate; a number of subtle issues arise when manifold domains are
considered in conjunction with Sobolev spaces, the subject of two recent
monographs~\cite{Aubi82,Hebe91}.
Some of the difficulties encountered impact
approximation theory on manifolds~\cite{Adam78,Aubi82,Rose97,Schw91,Hebe91}.
However, if a sufficient amount of the function space framework is in place, 
and if the underlying manifold admits a partition of unity 
(e.g., if it is paracompact), then it should be possible to extend
finite element approximation theory by operating chartwise, 
and then globalizing the local results using the partition of unity.
We will assume here that such extensions are possible;
understanding the manifold case of multiple charts and other
complications is work in progress~\cite{Hols97c}.

The starting point for our residual-based error indicator is the 
linearization inequality~(\ref{eqn:verf}).
In our setting of the weak formulation~(\ref{eqn:weak})--(\ref{eqn:weakForm}),
we make the appropriate choice~(\ref{eqn:banach1})--(\ref{eqn:banach2}),
where we restrict our discussion
here to a single elliptic system for a scalar or a $d$-vector
(i.e., the product space has dimension $n_e=1$), which includes
the examples presented later in the paper.
The linearization inequality then involves standard Sobolev norms:
\begin{equation}
   \label{eqn:verfagain}
   C_1 \|F(u_h)\|_{W^{-1,q}(\calg{M})}
 \le \| u - u_h \|_{W^{1,r}(\calg{M})} \le
   C_2 \|F(u_h)\|_{W^{-1,q}(\calg{M})},
\end{equation}
for $1/p + 1/q = 1$, $r \ge \text{min}\{ p, q \}$,
where $W^{-1,q}(\calg{M}) = (W^{1,q}(\calg{M}))^*$ denotes the dual
space of bounded linear functionals on $W^{1,q}(\calg{M})$.
The norm of the nonlinear residual $F(\cdot)$ in the dual space of bounded
linear functionals on $W^{1,q}(\calg{M})$ is defined in the usual way:
\begin{equation}
  \label{eqn:dualnorm}
   \| F(u) \|_{W^{-1,q}(\calg{M})} = \sup_{ 0 \ne v \in W^{1,q}(\calg{M})}
      \frac{ | \langle F(u),v \rangle | }{ \|v\|_{W^{1,q}(\calg{M})} }.
\end{equation}
The numerator is the nonlinear weak form $\langle F(u),v \rangle$ appearing
in~(\ref{eqn:weakForm}).
(We will consider only the case of no parameters;
see~\cite{Verf94} for the case of parameters.)
In order to derive a bound on the weak form in the numerator
we must first introduce quite a bit of notation that
we have managed to avoid until now.

To begin, we assume that the $d$-manifold $\calg{M}$ has been exactly
triangulated with a set $\calg{S}$ of shape-regular $d$-simplices
(the finite dimension $d$ is arbitrary throughout this discussion).
A family of simplices will be referred to here as shape-regular if
for all simplices in the family the ratio of the diameter of the
circumscribing sphere to that of the inscribing sphere
is bounded by an absolute fixed constant, independent of the numbers and
sizes of the simplices that may be generated through refinements.
(For a more careful definition of shape-regularity and related concepts,
see~\cite{Ciar78}.)
It will be convenient to introduce the following notation:
\begin{center}
\begin{tabular}{lcl}
$\calg{S}$    & = & Set of shape-regular simplices
                    triangulating $\calg{M}$ \\
$\calg{N}(s)$ & = & Union of faces in simplex set
                    $s$ lying on $\partial_N \calg{M}$ \\
$\calg{I}(s)$ & = & Union of faces in simplex set
                    $s$ not in $\calg{N}(s)$ \\
$\calg{F}(s)$ & = & $\calg{N}(s) \cup \calg{I}(s)$ \\
$\omega_s$    & = & $~\bigcup~ \{~ \tilde{s} \in \calg{S} ~|~
                         s \bigcap \tilde{s} \ne \emptyset,
                     ~\text{where}~s \in \calg{S} ~\}$ \\
$\omega_f$    & = & $~\bigcup~ \{~ \tilde{s} \in \calg{S} ~|~
                     f \bigcap \tilde{s} \ne \emptyset,
                     ~\text{where}~f \in \calg{F} ~\}$ \\
$h_s$         & = & Diameter (inscribing sphere) of the simplex $s$ \\
$h_f$         & = & Diameter (inscribing sphere) of the face $f$.
\end{tabular}
\end{center}
When the argument to one of the face set
functions $\calg{N}$, $\calg{I}$, or $\calg{F}$ is in fact the entire
set of simplices $\calg{S}$, we will leave off the explicit dependence on
$\calg{S}$ without danger of confusion.
Referring forward briefly to Figure~\ref{fig:river} will be convenient.
The two darkened triangles in the left picture in Figure~\ref{fig:river}
represents the set $w_f$ for the face $f$ shared
by the two triangles.
The clear triangles in the right picture in Figure~\ref{fig:river}
represents the set $w_s$ for the darkened triangle $s$ in the center
(the set $w_s$ also includes the darkened triangle).

Finally, we will also need some notation to represent
discontinuous jumps in function values across faces interior to the
triangulation.
To begin, for any face $f \in \calg{N}$, let $n_f$ denote the
unit outward normal;
for any face $f \in \calg{I}$, take $n_f$ to be an arbitrary
(but fixed) choice of one of the two possible face normal orientations.
Now, for any $v \in L^2(\calg{M})$ such that
$v \in C^0(s) ~\forall s \in \calg{S}$, define the {\em jump function}:
$$
[v]_f(x) = \lim_{\epsilon \rightarrow 0^+} v(x+\epsilon n_f)
         - \lim_{\epsilon \rightarrow 0^-} v(x-\epsilon n_f).
$$

We now begin the analysis by splitting the volume and surface integrals
in~(\ref{eqn:weakForm})
into sums of integrals over the individual elements and faces, and we then
employ the divergence theorem~(\ref{eqn:divergence}) to work
backward towards the strong form in each element:
\begin{eqnarray*}
  \langle F(u),v \rangle
    &=& \int_{\calg{M}} \calg{G}_{ij} (A^{ia} v^j_{~;a} + B^i v^j)~dx
           + \int_{\partial_N \calg{M}} \calg{G}_{ij} C^i v^j ~ds
\\
    &=& \sum_{s \in \calg{S}}
         \int_{s} \calg{G}_{ij} (A^{ia} v^j_{~;a} + B^i v^j)~dx
    + \sum_{f \in \calg{N}}
         \int_{f} \calg{G}_{ij} C^i v^j ~ds
\nonumber \\
    &=& \sum_{s \in \calg{S}}
         \int_{s} \calg{G}_{ij} (B^i - A^{ia}_{~~;a}) v^j~dx
    + \sum_{s \in \calg{S}}
         \int_{\partial s} \calg{G}_{ij} A^{ia} n_a v^j~ds
\nonumber \\
    & & \quad
    + \sum_{f \in \calg{N}}
         \int_{f} \calg{G}_{ij} C^i v^j ~ds.
\nonumber
\end{eqnarray*}
Using the fact that~(\ref{eqn:galerkin}) holds for the solution to the
discrete problem, we employ the jump function and write
\begin{eqnarray}
\langle F(u_h),v \rangle 
    &=& \langle F(u_h),v-v_h \rangle
\label{eqn:estim5} \\
    &=& \sum_{s \in \calg{S}}
         \int_{s} \calg{G}_{ij} (B^i - A^{ia}_{~~;a}) (v^j - v^j_h)~dx
\nonumber \\
    & & + \sum_{s \in \calg{S}}
         \int_{\partial s} \calg{G}_{ij} A^{ia} n_a (v^j - v^j_h)~ds
\nonumber \\
    & & + \sum_{f \in \calg{N}}
         \int_{f} \calg{G}_{ij} C^i (v^j-v^j_h) ~ds
\nonumber \\
    &=& \sum_{s \in \calg{S}}
         \int_{s} \calg{G}_{ij} (B^i - A^{ia}_{~~;a}) (v^j - v^j_h)~dx
\nonumber \\
    & & + \sum_{f \in \calg{I}}
         \int_{f} \calg{G}_{ij} \left[ A^{ia} n_a \right]_f (v^j - v^j_h)~ds
\nonumber \\
    & & \quad + \sum_{f \in \calg{N}}
         \int_{f} \calg{G}_{ij} (C^i + A^{ia} n_a) (v^j-v^j_h) ~ds
\nonumber \\
  &\le& \sum_{s \in \calg{S}} \left(
         \| B^i - A^{ia}_{~~;a} \|_{L^p(s)}
         \| v^j - v^j_h \|_{L^q(s)}
      \right)
\nonumber \\
    & & + \sum_{f \in \calg{I}} \left(
         \| \left[ A^{ia} n_a \right]_f \|_{L^p(f)}
         \| v^j - v^j_h \|_{L^q(f)}
      \right)
\nonumber \\
   & & \quad + \sum_{f \in \calg{N}} \left(
         \| C^i + A^{ia} n_a \|_{L^p(f)}
         \| v^j-v^j_h \|_{L^q(f)}
      \right),
\nonumber 
\end{eqnarray}
where we have applied the H\"older inequality~(\ref{eqn:holder})
three times with $1/p + 1/q = 1$.

In order to bound the sums on the right, we will employ a standard
tool known as a $W^{1,p}$-quasi-interpolant $I_h$.
An example of such an interpolant is due to Scott and Zhang~\cite{ScZh90},
which we refer to as the SZ-interpolant (see also Cl\'{e}ment's interpolant
in~\cite{Clem75}).
Unlike point-wise polynomial interpolation, which is not well-defined for
functions in $W^{1,p}(\calg{M})$ when the embedding
$W^{1,p}(\calg{M}) \hookrightarrow C^0(\calg{M})$ fails,
the SZ-interpolant $I_h$ can be constructed quite generally
for $W^{1,p}$-functions on shape-regular meshes of $2$- and $3$-simplices.
Moreover, it can be shown to have the following remarkable
local approximation properties:
For all $v \in W^{1,q}(\calg{M})$, it holds that
\begin{eqnarray}
\| v - I_h v \|_{L^q(s)} &\le C_s h_s \| v \|_{W^{1,q}(\omega_s)},
\label{eqn:sz_approx} \\
\| v - I_h v \|_{L^q(f)} &\le C_f h_f^{1-1/q} \| v \|_{W^{1,q}(\omega_f)}.
\label{eqn:sz_approx_bnd}
\end{eqnarray}
For the construction of the SZ-interpolant, and for a proof of the
approximation inequalities in $L^p$-spaces for $p \ne 2$, see~\cite{ScZh90}.
A simple construction and the proof of the first inequality can also be
found in the appendix of~\cite{HoTi96b}.

Employing now the SZ-interpolant by taking $v_h = I_h v$ in~(\ref{eqn:estim5}),
using~(\ref{eqn:sz_approx})--(\ref{eqn:sz_approx_bnd}),
and noting that $1-1/q = 1/p$, we have
\begin{eqnarray}
\langle F(u_h),v \rangle
&\le& \sum_{s \in \calg{S}} C_s h_s
         \| B^i - A^{ia}_{~~;a} \|_{L^p(s)}
         \| v^j \|_{W^{1,q}(\omega_s)}
\nonumber \\
& & + \sum_{f \in \calg{I}} C_f h_f^{1/p}
         \| \left[ A^{ia} n_a \right]_f \|_{L^p(f)}
         \| v^j \|_{W^{1,q}(\omega_f)}
\nonumber \\
& & + \sum_{f \in \calg{N}} C_f h_f^{1/p}
         \| C^i + A^{ia} n_a \|_{L^p(f)}
         \| v^j \|_{W^{1,q}(\omega_f)}
\nonumber \\
&\le& 
    \left( \sum_{s \in \calg{S}} C_s^p h_s^p
             \| B^i - A^{ia}_{~~;a} \|_{L^p(s)}^p 
    \right.
\label{eqn:residual_sz} \\
& &
 \hspace*{-2.0cm}
    \left.
         + \sum_{f \in \calg{I}} C_f^p h_f
             \| \left[ A^{ia} n_a \right]_f \|_{L^p(f)}^p
         + \sum_{f \in \calg{N}} C_f^p h_f
             \| C^i + A^{ia} n_a \|_{L^p(f)}^p
    \right)^{1/p}
\nonumber \\
& &
 \hspace*{-2.0cm}
\cdot
    \left( \sum_{s \in \calg{S}} \| v^j \|_{W^{1,q}(\omega_s)}^q
         + \sum_{f \in \calg{I}} \| v^j \|_{W^{1,q}(\omega_f)}^q
         + \sum_{f \in \calg{N}} \| v^j \|_{W^{1,q}(\omega_f)}^q
    \right)^{1/q}
\nonumber
\end{eqnarray}
where we have used the discrete H\"{o}lder inequality 
to obtain the last inequality.

It is not difficult to show (cf.~\cite{Verf94}) that the simplex
shape regularity assumption bounds the number of possible overlaps of the
sets $\omega_s$ with each other, and also bounds the number of possible
overlaps of the sets $\omega_f$ with each other.
This makes it possible to
establish the following two inequalities:
\begin{eqnarray}
           \sum_{s \in \calg{S}} \| v^j \|_{W^{1,q}(\omega_s)}^q
    &\le D_s \| v \|_{W^{1,q}(\calg{M})}^q,
\label{eqn:loc1} \\
           \sum_{f \in \calg{F}} \| v^j \|_{W^{1,q}(\omega_f)}^q
    &\le D_f \| v \|_{W^{1,q}(\calg{M})}^q,
\label{eqn:loc2}
\end{eqnarray}
where $D_s$ and $D_f$ depend on the shape regularity constants reflecting
these overlap bounds.
Therefore, since $\calg{I} \subset \calg{F}$ and $\calg{N} \subset \calg{F}$,
we employ~(\ref{eqn:loc1})--(\ref{eqn:loc2}) in~(\ref{eqn:residual_sz})
which gives
\begin{eqnarray}
   \label{eqn:residual_almost}
\langle F(u_h),v \rangle
&\le& C_5 \| v \|_{W^{1,q}(\calg{M})}
    \cdot
    \left(
         \sum_{s \in \calg{S}} h_s^p
             \| B^i - A^{ia}_{~~;a} \|_{L^p(s)}^p
    \right.
\\
&+&
\hspace*{-0.1cm}
 \sum_{f \in \calg{I}} h_f
             \| \left[ A^{ia} n_a \right]_f \|_{L^p(f)}^p
    \left.
       + \sum_{f \in \calg{N}} h_f
             \| C^i + A^{ia} n_a \|_{L^p(f)}^p
    \right)^{1/p},
\nonumber
\end{eqnarray}
where $C_5 = \max_{\calg{S},\calg{F}} \{ C_s, C_f \}
 \cdot \max_{\calg{S},\calg{F}} \{ D_s^{1/q}, D_f^{1/q} \}$
depends on the shape regularity of the simplices in $\calg{S}$.

We finally now use~(\ref{eqn:residual_almost}) in~(\ref{eqn:dualnorm})
to achieve the upper bound in~(\ref{eqn:verfagain}):
\begin{eqnarray}
 \| u - u_h \|_{W^{1,r}(\calg{M})} 
&\le& C_2 \|F(u_h)\|_{W^{-1,q}(\calg{M})}
\nonumber \\
&=& C_2 \sup_{ 0 \ne v \in W^{1,q}(\calg{M})}
      \frac{ | \langle F(u_h),v \rangle| }{ \|v\|_{W^{1,q}(\calg{M})} }
\nonumber \\
&\le& C_2 C_5
    \left(
         \sum_{s \in \calg{S}} h_s^p
             \| B^i - A^{ia}_{~~;a} \|_{L^p(s)}^p
    \right.
\label{eqn:residual_final} \\
& &    + \sum_{f \in \calg{I}} h_f
             \| \left[ A^{ia} n_a \right]_f \|_{L^p(f)}^p
\nonumber \\
& & \left.
       + \sum_{f \in \calg{N}} h_f
             \| C^i + A^{ia} n_a \|_{L^p(f)}^p
    \right)^{1/p}.
\nonumber
\end{eqnarray}
We will make one final transformation that will turn this into a sum of
element-wise error indicators that will be easier to work with in an
implementation.
We only need to account for the interior face integrals (which would
otherwise be counted twice) when we combine the sum over the faces
into the sum over the elements.
This leave us with the following
\begin{theorem}
Let $u \in W^{1,r}(\calg{M})$ be a regular solution
of~(\ref{eqn:str1})--(\ref{eqn:str3}), or equivalently
of~(\ref{eqn:weak})--(\ref{eqn:weakForm}),
where~(\ref{eqn:banach1})--(\ref{eqn:banach2}) holds.
Then under the same assumptions as in Theorem~\ref{thm:verf},
the following {\em a posteriori} error estimate holds for
a Petrov-Galerkin approximation $u_h$ satisfying~(\ref{eqn:galerkin}):
\begin{equation}
   \label{eqn:indicator_residual}
 \| u - u_h \|_{W^{1,r}(\calg{M})} \le
   C \left( \sum_{s \in \calg{S}} \eta_s^p \right)^{1/p},
\end{equation}
where
$$
C = 2 \cdot \max_{\calg{S},\calg{F}} \{ C_s, C_f \} 
      \cdot \max_{\calg{S},\calg{F}} \{ D_s^{1/q}, D_f^{1/q} \}
\cdot \| DF(u)^{-1} \|_{\calg{L}(W^{-1,q},W^{1,p})},
$$
and where the element-wise error indicator $\eta_s$ is defined as:
\begin{eqnarray}
\eta_s 
  &=& \left(
      h_s^p \| B^i - A^{ia}_{~~;a} \|_{L^p(s)}^p
  + \frac{1}{2} \sum_{f \in \calg{I}(s)} h_f
     \| \left[ A^{ia} n_a \right]_f \|_{L^p(f)}^p
   \right.
\label{eqn:estimator} \\
& & \left.
   + \sum_{f \in \calg{N}(s)} h_f
     \| C^i + A^{ia} n_a \|_{L^p(f)}^p
   \right)^{1/p}.
\nonumber
\end{eqnarray}
\end{theorem}
\begin{proof}
The proof follows from~(\ref{eqn:residual_final})
and the discussion above.\qed
\end{proof}
The element-wise error indicator in~(\ref{eqn:estimator}) provides an error
bound in the $W^{1,r}$-norm for a general covariant nonlinear elliptic
system of the form~(\ref{eqn:str1})--(\ref{eqn:str3}), with
$1/p+1/q=1$, $r \ge \text{min}\{p,q\}$, which may be
more appropriate than the $r=p=q=2$ case for some nonlinear problems.
Issues related to this topic are discussed in~\cite{Verf94}.
Following~\cite{Verf94}, it is possible to use a similar analysis
to construct lower bounds, dual to~(\ref{eqn:indicator_residual}),
of the form
$$
   \tilde{C} \left( \sum_{s \in \calg{S}} \eta_s^p \right)^{1/p}
 \le \| u - u_h \|_{W^{1,r}(\calg{M})}.
$$
Such results are useful for
performing unrefinement and in accessing the quality of an error-indicator.

\subsection{Duality-based {\em a posteriori} error indicators}
  \label{sec:dual_indicators}

We now derive an alternative {\em a posteriori} error indicator for general
Petrov-Galerkin approximations~(\ref{eqn:galerkin}) to the solutions of
general nonlinear elliptic systems of tensors of the
form~(\ref{eqn:str1})--(\ref{eqn:str3}).
The indicator is based on the solution of a global linearized adjoint
(or {\em dual}) problem; again the analysis involves primarily the weak
formulation~(\ref{eqn:weak})--(\ref{eqn:weakForm}), and follows
closely that of~\cite{BaSi95,EHM2001}.
This approach can be viewed as simply another
way to bound the nonlinear residual $\|F(u_h)\|_{\calg{B}_2^*}$ after
employing the (possibly quite crude) one-time linearization in
Theorem~\ref{thm:verf}.
However, the approach can be used to avoid the one-time linearization step,
bringing the stability properties of the differential operator into
the error indicator by updating the linearization as the solution is improved,
and by incorporating the linearization operator itself
into the error indicator.

As before, we are interested in the solution to the operator
equation~(\ref{eqn:abstract_strong})
and also in error estimates for approximations $u_h$
satisfying~(\ref{eqn:abstract_strong_discrete}).
We begin with the generalized Taylor remainder in integral form:
\begin{equation}
   \label{eqn:taylor}
F(u+h) = F(u) + \left\{ \int_0^1 DF(u+\xi h) d\xi \right\} h.
\end{equation}
Taking $h=u_h-u$, the error $e=u-u_h$ can be expressed as follows:
$$
R = -F(u_h) = -F(u+[u_h-u]) = -F(u) - A(u_h-u)
                            = 0 - Ae,
$$
where the linearization operator $A$ is defined from~(\ref{eqn:taylor}) as:
\begin{equation}
   \label{eqn:dual_operator}
A = \int_0^1 DF(u+\xi h) d\xi.
\end{equation}
If a linear functional of the error $l(e) = \langle e,\psi \rangle$
is of interest rather than the error itself, where $\psi$ is the
Riesz-representer of $l(\cdot)$, then we can exploit the linearization
operator $A$ in~(\ref{eqn:dual_operator}),
and its (unique) adjoint $A^T$, to produce
an error indicator:
$$
|\langle e,\psi \rangle | = |\langle e,A^T \phi \rangle |
   = | \langle Ae,\phi \rangle | = | \langle R,\phi \rangle |
   = | \langle F(u_h),\phi \rangle |.
$$
The indicator requires the solution of the linearized {\em dual problem}:
\begin{equation}
   \label{eqn:dual_problem}
A^T \phi = \psi
\end{equation}
for the residual weights $\phi$, where the data for the dual
problem $\psi$ is the Riesz-representer of the functional of interest.
Strong norm estimates of the form~(\ref{eqn:verfagain})
can be established using duality (cf.~\cite{BaSi95}), but
the operator information represented by the dual solution $\phi$ is then
lost (it appears in the constants).
If a functional of the error is of interest
(e.g., the error along a curve or surface in the domain),
then a more delicate
approach is to instead employ the dual solution $\phi$ as part
of the indicator:
\begin{equation}
   \label{eqn:dual_indicator}
| \langle e,\psi \rangle | = | \langle F(u_h),\phi \rangle |
     \le error~estimate.
\end{equation}
The dual solution $\phi$ obtained by solving~(\ref{eqn:dual_problem})
is used locally (element-wise) in~(\ref{eqn:dual_indicator}), with the
dual solution as residual weights (cf.~\cite{EHM2001}).

To construct such estimates for general Petrov-Galerkin
approximations~(\ref{eqn:galerkin}) to the solutions of general nonlinear
elliptic systems of tensors of the form~(\ref{eqn:str1})--(\ref{eqn:str3}),
we first need some simple identities to help identify the form of
the linearized dual problem:
$$
A^{ia}(u^k,u^k_{~;c}) - A^{ia}(U^k,U^k_{~;c})
~~~~~~~~~~~~~~~~~~~~~~~~~~~~~~~~~~~~~~~~~~~~~~~~~~~~~~~~~~~~~~~~~~~~
$$
\vspace*{-0.6cm}
\begin{eqnarray*}
~~
 &=& \int_0^1 \frac{d}{ds} A^{ia}(su^k+ (1-s)U^k,su^k_{~;c} + (1-s)U^k_{~;c}) ds
\nonumber \\
 &=& \int_0^1 \left\{
              D_1 A^{ia}(su^k+ (1-s)U^k,su^k_{~;c} + (1-s)U^k_{~;c})
      \right.
\nonumber \\
 & & ~~~~ \cdot \frac{d}{ds} [ su^k+ (1-s)U^k ]
\nonumber \\
 & & + D_2 A^{ia}(su^k+ (1-s)U^k,su^k_{~;c} + (1-s)U^k_{~;c})
\nonumber \\
 & & ~~~~ \cdot \left. \frac{d}{ds} [ su^k_{~;c} + (1-s)U^k_{~;c} ] \right\} ds
\nonumber \\
 &=& \left\{ \int_0^1
       D_1 A^{ia}(su^k+ (1-s)U^k,su^k_{~;c} + (1-s)U^k_{~;c}) 
    ds \right\} ( u^k - U^k )
\nonumber \\
 & & + \left\{ \int_0^1
        D_2 A^{ia}(su^k+ (1-s)U^k,su^k_{~;c} + (1-s)U^k_{~;c})
     ds \right\} ( u^k_{;c} - U^k_{;c} )
\nonumber \\
 &=& \calg{A}^{ia}_{~~b} e^b + \calg{A}^{ia~c}_{~~b} e^b_{~;c}.
\nonumber \\
\end{eqnarray*}
$$
B^i(u^k,u^k_{~;c}) - B^i(U^k,U^k_{~;c})
~~~~~~~~~~~~~~~~~~~~~~~~~~~~~~~~~~~~~~~~~~~~~~~~~~~~~~~~~~~~~~~~~~~~
$$
\vspace*{-0.6cm}
\begin{eqnarray*}
~~
 &=& \int_0^1 \frac{d}{ds} B^i(su^k+ (1-s)U^k,su^k_{~;c} + (1-s)U^k_{~;c}) ds
\nonumber \\
 &=& \int_0^1 \left\{
              D_1 B^i(su^k+ (1-s)U^k,su^k_{~;c} + (1-s)U^k_{~;c})
     \right.
\nonumber \\
 & & ~~~~ \cdot \frac{d}{ds} [ su^k+ (1-s)U^k ]
\nonumber \\
& & + D_2 B^i(su^k+ (1-s)U^k,su^k_{~;c} + (1-s)U^k_{~;c})
\nonumber \\
& & ~~~~ \cdot \left. \frac{d}{ds} [ su^k_{~;c} + (1-s)U^k_{~;c} ] \right\} ds
\nonumber \\
 &=& \left\{ \int_0^1
          D_1 B^i(su^k+ (1-s)U^k,su^k_{~;c} + (1-s)U^k_{~;c})
     ds \right\} ( u^k - U^k )
\nonumber \\
 & & 
   + \left\{ \int_0^1
            D_2 B^i(su^k+ (1-s)U^k,su^k_{~;c} + (1-s)U^k_{~;c})
     ds \right\} ( u^k_{;c} - U^k_{;c} )
\nonumber \\
 &=& \calg{B}^i_{~b} e^b + \calg{B}^{i~c}_{~b} e^b_{~;c}.
\end{eqnarray*}
Similarly,
$$
C^i(u^k) - C^i(U^k) = \left\{ \int_0^1
                           D_1 C^i(su^k+ (1-s)U^k)
                      ds \right\} ( u^k - U^k )
                    = \calg{C}^i_{~b} e^b.
$$
Therefore, given our original weak form in~(\ref{eqn:weakForm}), we have
\begin{eqnarray*}
\langle F(u) - F(U),\phi \rangle
   &=& \int_{\calg{M}} \calg{G}_{ij}
       \left\{ [A^{ia}(u^k,u^k_{~;c}) - [A^{ia}(U^k,U^k_{~;c})] \phi^j_{~;a}
       \right.
\nonumber \\
& & \left. +~ [B^i(u^k,u^k_{~;c}) - B^i(U^k,U^k_{~;c}) ] \phi^j \right\} ~dx
\nonumber \\
& & + \int_{\partial_1 \calg{M}} \calg{G}_{ij}
           [ C^i(u^k) - C^i(U^k) ] \phi^j ~ds
\nonumber \\
&=& \int_{\calg{M}} \calg{G}_{ij} \left\{
     ( \calg{A}^{ia}_{~~b} e^b + \calg{A}^{ia~c}_{~~b} e^b_{~;c} ) \phi^j_{~;a}
    \right.
\nonumber \\
& &
    \left.
 +~ ( \calg{B}^i_{~b} e^b + \calg{B}^{i~c}_{~b} e^b_{~;c} ) \phi^j \right\}~dx
 + \int_{\partial_1 \calg{M}} \calg{G}_{ij} \calg{C}^i_{~b} e^b \phi^j ~ds
\nonumber \\
&=& \langle Ae,\phi \rangle 
\nonumber \\
&=&  \langle e,A^T\phi \rangle .
\nonumber
\end{eqnarray*}
The weak form of the linearized dual problem is then:
\begin{equation}
   \label{eqn:dual_weak}
\text{Find}~ \phi \in \calg{B}_1 ~\text{such~that}~
     \langle A^T \phi, v \rangle = \langle \psi, v \rangle,
     ~~~~ \forall v \in \calg{B}_2,
\end{equation}
where the adjoint form is
\begin{equation}
    \label{eqn:adjoint_form}
\langle A^T \phi, v \rangle
= \int_{\calg{M}} \calg{G}_{ij} \left\{
     \calg{A}^{ia~c}_{~~b} \phi^j_{~;a} v^b_{~;c}
   + \calg{A}^{ia}_{~~b} \phi^j_{~;a} v^b
   + \calg{B}^{i~c}_{~b} \phi^j v^b_{~;c}
   \right.
\end{equation}
$$
   \left.
   +~ \calg{B}^i_{~b} \phi^j v^b
     \right\}~dx
+ \int_{\partial_1 \calg{M}} \calg{G}_{ij} \calg{C}^i_{~b} \phi^j v^b ~ds.
$$
The strong form of the linearized dual problem
in~(\ref{eqn:dual_weak})--(\ref{eqn:adjoint_form}) is then:
\begin{eqnarray*}
  \calg{G}_{ij} \left\{
       \calg{A}^{ia}_{~~b}\phi^j_{~;a}
     - \left( \calg{A}^{ia~c}_{~~b}\phi^j_{~;a} \right)_{;c}
     + \calg{B}^i_{~b} \phi^j
     - \left( \calg{B}^{i~c}_{~b} \phi^j \right)_{;c}
  \right\} 
   &=&0 \text{~in~} \calg{M}, \\
  \calg{G}_{ij} \left\{
       \calg{A}^{ia~c}_{~~b} \phi^j_{~;a} n_c
     + \left( \calg{B}^{i~c}_{~b} n_c + \calg{C}^i_{~b} \right) \phi^j
  \right\} 
   &=&0 \text{~on~} \partial_1 \calg{M}, \\
       u^i(x^b)
   &=&0 \text{~on~} \partial_0 \calg{M}.
\end{eqnarray*}
This leads to the following error representation:
\begin{theorem}
Given a projector $P_h : \calg{B}_1 \mapsto U_h$ onto the finite element
subspace $U_h \subset \calg{B}_1$, the functional error is:
$$
\langle e, \psi \rangle = \langle \calg{R}(U), \phi \rangle,
$$
where
\begin{eqnarray}
   \label{eqn:estimator_duality}
\langle \calg{R}(U), v \rangle
&=&   \sum_{s \in \calg{S}}
         \int_{s} \calg{G}_{ij} (B^i - A^{ia}_{~~;a}) (v^j - P_h v^j)~dx
\\
& & + \sum_{f \in \calg{I}}
         \int_{f} \calg{G}_{ij} \left[ A^{ia} n_a \right]_f 
             (v^j - P_h v^j)~ds
\nonumber \\
& & + \sum_{f \in \calg{N}}
         \int_{f} \calg{G}_{ij} (C^i + A^{ia} n_a) (v^j- P_h v^j) ~ds.
\nonumber
\end{eqnarray}
\end{theorem}
\begin{proof}
We begin by working backward toward the strong form:
\begin{eqnarray*}
\langle e,\psi \rangle
&=& \langle F(U),\phi \rangle
\\
&=& \int_{\calg{M}} \calg{G}_{ij} (A^{ia} \phi^j_{~;a} + B^i \phi^j)~dx
           + \int_{\partial_N \calg{M}} \calg{G}_{ij} C^i \phi^j ~ds
\nonumber \\
&=& \sum_{s \in \calg{S}}
         \int_{s} \calg{G}_{ij} (A^{ia} \phi^j_{~;a} + B^i \phi^j)~dx
    + \sum_{f \in \calg{N}}
         \int_{f} \calg{G}_{ij} C^i \phi^j ~ds
\nonumber \\
&=& \sum_{s \in \calg{S}}
         \int_{s} \calg{G}_{ij} (B^i - A^{ia}_{~~;a}) \phi^j~dx
    + \sum_{s \in \calg{S}}
         \int_{\partial s} \calg{G}_{ij} A^{ia} n_a \phi^j~ds
\nonumber \\
& & + \sum_{f \in \calg{N}}
         \int_{f} \calg{G}_{ij} C^i \phi^j ~ds.
\nonumber
\end{eqnarray*}
Using then Galerkin orthogonality and a jump function gives:
\begin{eqnarray*}
\langle F(U),\phi \rangle
&=& \langle F(U),\phi-P_h \phi \rangle
\\
&=& \sum_{s \in \calg{S}}
         \int_{s} \calg{G}_{ij} (B^i - A^{ia}_{~~;a}) (\phi^j - P_h \phi^j)~dx
\nonumber \\
& & + \sum_{s \in \calg{S}}
         \int_{\partial s} \calg{G}_{ij} A^{ia} n_a (\phi^j - P_h \phi^j)~ds
\nonumber \\
& & + \sum_{f \in \calg{N}}
         \int_{f} \calg{G}_{ij} C^i (\phi^j- P_h \phi^j) ~ds
\nonumber \\
&=& \sum_{s \in \calg{S}}
         \int_{s} \calg{G}_{ij} (B^i - A^{ia}_{~~;a}) (\phi^j - P_h \phi^j)~dx
\nonumber \\
& & + \sum_{f \in \calg{I}}
         \int_{f} \calg{G}_{ij} \left[ A^{ia} n_a \right]_f 
             (\phi^j - P_h \phi^j)~ds
\nonumber \\
& & + \sum_{f \in \calg{N}}
         \int_{f} \calg{G}_{ij} (C^i + A^{ia} n_a) (\phi^j- P_h \phi^j) ~ds
\nonumber \\
&=& \langle \calg{R}(U),\phi \rangle.
\nonumber
\end{eqnarray*}
\qed
\vspace*{-0.2cm}
\end{proof}

The error representation can be used as an error indicator as
illustrated in Algorithm~\ref{alg:dual_indicator}.
\begin{algorithmX}
   \label{alg:dual_indicator}
{\em (Linearized dual error indicator)}
\begin{enumerate}
\item Decide which linear functional(s) of the error $l(e) = \langle e, \psi \rangle$ is of interest.
\item Pose and solve the linearized dual problem $A^T \phi = \psi$ for the dual weight function $\phi$.
\item Numerically approximate $\langle \calg{R}(U),\phi \rangle$ within each element as an element-wise error indicator.
\item Elements which fail an indicator test (using equi-distribution) are marked for refinement.
\end{enumerate}
\end{algorithmX}

\section[Manifold Code (MC): finite element methods on manifolds]
        {Manifold Code (MC): adaptive multilevel finite element methods
         on manifolds}
   \label{sec:mc}

\MC (see also~\cite{HBW99,BHW99,BaHo98a,HoBe99,BeHo99})
is an adaptive multilevel finite element software package, written in ANSI C,
which was developed by the author over several years at
Caltech and UC San Diego.
It is designed to produce highly accurate numerical solutions to
nonlinear covariant elliptic systems of tensor equations
on 2- and 3-manifolds in an optimal or nearly-optimal way.
\MC\ employs {\em a posteriori} error estimation, adaptive simplex subdivision, 
unstructured algebraic multilevel methods, global inexact Newton methods, 
and numerical continuation methods for the highly accurate numerical solution
of nonlinear covariant elliptic systems on (Riemannian) 2- and 3-manifolds.

\subsection{The overall design of MC}

\MC\ is an implementation of Algorithm~\ref{alg:adapt},
where Algorithm~\ref{alg:newton} is employed for solving
nonlinear elliptic systems that
arise in Step~1 of Algorithm~\ref{alg:adapt}.
The linear Newton equations in each iteration of Algorithm~\ref{alg:newton}
are solved with algebraic multilevel methods, and the algorithm is
supplemented with a continuation technique when necessary.
Several of the features of \MC\ are somewhat unusual, allowing for the treatment
of very general nonlinear elliptic systems of tensor equations on domains
with the structure of 2- and 3-manifolds.
In particular, some of these features are:
\begin{itemize}
\item {\em Abstraction of the elliptic system}: The elliptic system 
     is defined only through a nonlinear weak form over the domain manifold,
     along with an associated linearization form, also defined everywhere on
     the domain manifold
     (precisely the forms $\langle F(u),v \rangle$
     and $\langle DF(u)w,v \rangle$ in the discussions above).
     To use the {\em a posteriori} error indicators, a third function
     $F(u)$ must also be provided (essentially the strong form of the problem).
\item {\em Abstraction of the domain manifold}: The domain manifold is
     specified by giving a polyhedral representation of the topology, along
     with an abstract set of coordinate labels of the user's interpretation,
     possibly consisting of multiple charts.
     \MC\ works only with the topology of the domain, the connectivity
     of the polyhedral representation.
     The geometry of the domain manifold is provided only through the form
     definitions, which contain the manifold metric information, and through
     a {\tt oneChart()} routine that the user provides to resolve chart
     boundaries.
\item {\em Dimension independence}: Exactly the same code paths in \MC\
     are taken for both two- and three-dimensional problems (as well as for
     higher-dimensional problems).
     To achieve this dimension independence, \MC\ employs the simplex as its 
     fundamental geometrical object for defining finite element bases.
\end{itemize}
As a consequence of the abstract weak form approach to defining the problem,
the complete definition of a complex nonlinear tensor system such as large
deformation nonlinear elasticity requires writing only a few hundred lines
of C to define the two weak forms, and to define the {\tt oneChart()} routine.
Changing to a different tensor system (e.g. the example later in the
paper involving the constraints in the Einstein equations)
involves providing only a different definition of the forms and a different
domain description.

\subsection[Topology and geometry representation in MC]
           {Topology and geometry representation in MC:
            The Ringed Vertex}

A datastructure referred to as the {\em ringed-vertex} (cf.~\cite{Hols97c})
is used to represent meshes of $d$-simplices of arbitrary topology.
This datastructure is illustrated in Figure~\ref{fig:river}.
\begin{figure}
\begin{center}
\mbox{\myfigpdf{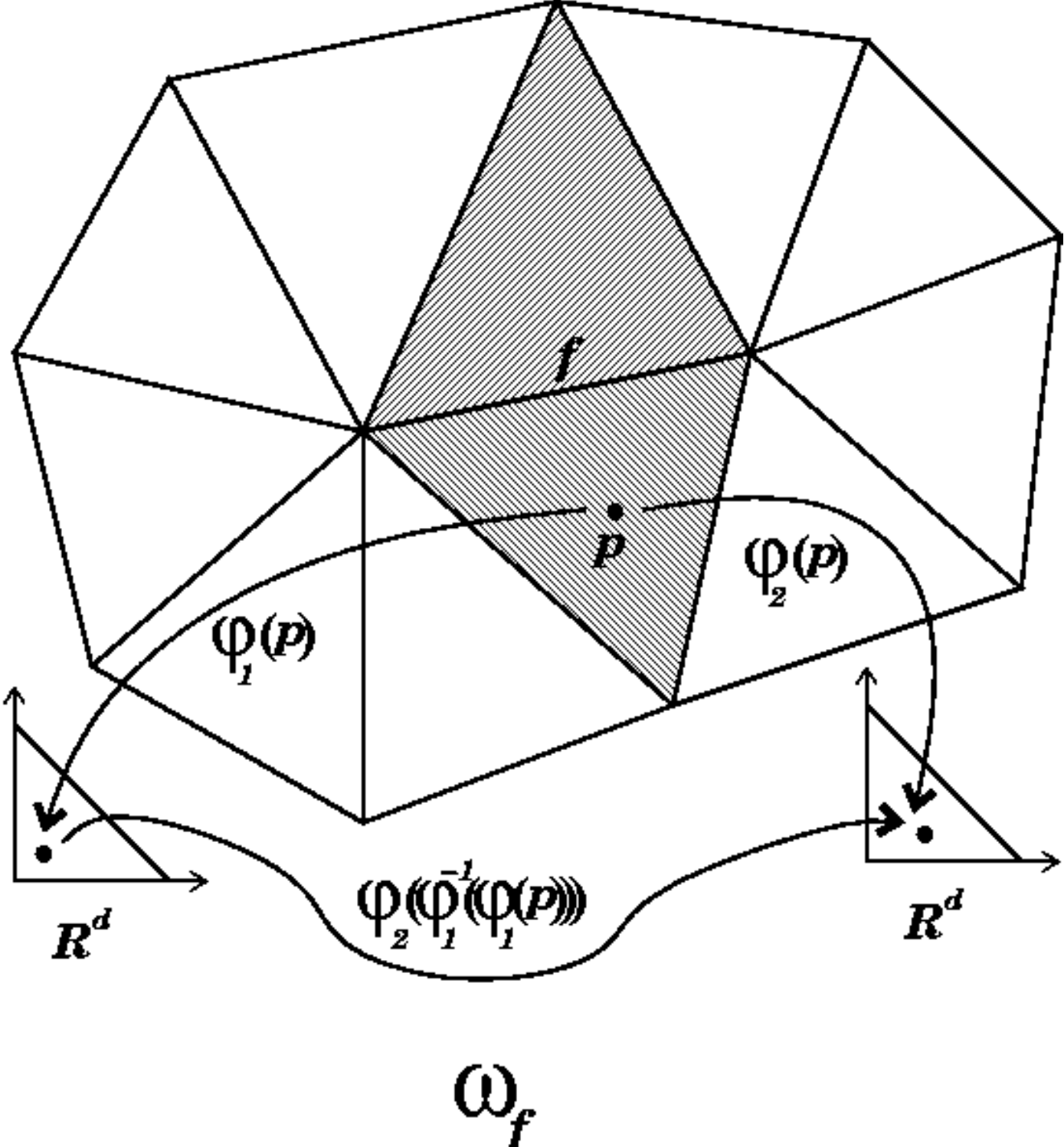}{2.3in}}
\mbox{\myfigpdf{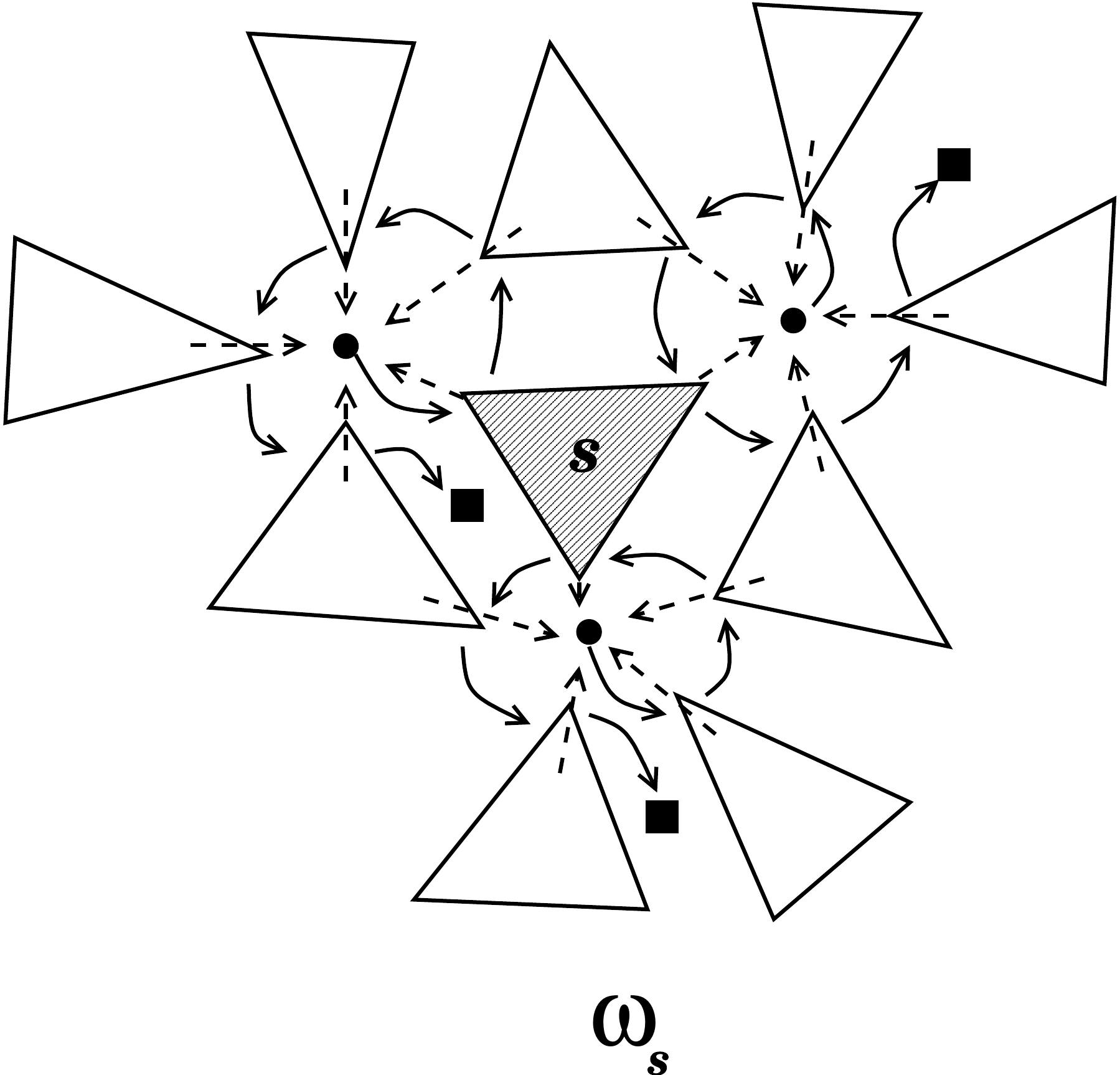}{2.3in}}
\end{center}
\caption{Polyhedral manifold representation.
The figure on the left shows two overlapping polyhedral (vertex) charts
consisting of the two rings of simplices around two vertices sharing an edge.
The region consisting of the two darkened triangles around the face $f$ is
denoted $\omega_f$, and represents the overlap of the two vertex charts.
Polyhedral manifold topology is represented by MC using the
{\em ringed-vertex} (or {\em RIVER}) datastructure.
The datastructure is illustrated for a given simplex $s$ in the figure on the
right; the topology primitives are vertices and $d$-simplices.
The collection of the simplices which meet the simplex $s$ at its vertices
(which then includes those simplices that share faces as well) is denoted
as $\omega_s$.
(The set $\omega_s$ includes $s$ itself.)
Edges are temporarily created during subdivision but are then destroyed 
(a similar ring datastructure is used to represent the edge topology). }
\label{fig:river}
\end{figure}
The ringed-vertex datastructure
is similar to the winged-edge, quad-edge, and edge-facet
datastructures commonly used in the computational geometry community for
representing 2-manifolds~\cite{Muck93}, but it can be used more generally
to represent arbitrary $d$-manifolds, $d \ge 2$.
It maintains a mesh of $d$-simplices with near minimal storage,
yet for shape-regular (non-degenerate) meshes, it provides $O(1)$-time access
to all information necessary for refinement, un-refinement, and
Petrov-Galerkin discretization of a differential operator.
The ringed-vertex datastructure also allows for dimension independent
implementations of mesh refinement and mesh manipulation, with one
implementation (the same code path) covering arbitrary dimension $d$.
An interesting feature of this datastructure is that the C structures used
for vertices, simplices, and edges are all of fixed size, so that a fast
array-based implementation is possible, as opposed to a less-efficient
list-based approach commonly taken for finite element implementations
on unstructured meshes.
A detailed description of the ringed-vertex datastructure, along with a
complexity analysis of various traversal algorithms, can be found
in~\cite{Hols97c}.

Since \MC\ is based entirely on the $d$-simplex, for adaptive refinement it
employs simplex bisection, using one of the simplex bisection strategies
outlined earlier.
Bisection is first used to refine an initial subset of the simplices in the 
mesh (selected according to some error indicator combined with
equi-distribution, discussed below), and then a
closure algorithm is performed in which bisection is used recursively on
any non-conforming simplices, until a conforming mesh is obtained.
If it is necessary to improve element shape, \MC\ attempts to optimize the
following simplex shape measure function for a given $d$-simplex $s$, in an
iterative fashion, similar to the approach taken in~\cite{BaSm97}:
\begin{equation}
\eta(s,d) = \frac{2^{2(1-\frac{1}{d})} 3^{\frac{d-1}{2}} |s|^{\frac{2}{d}}}
                 {\sum_{0 \le i < j \le d} |e_{ij}|^2}.
   \label{eqn:shape}
\end{equation}
The quantity $|s|$ represents the (possibly negative) volume of the
$d$-simplex, and $|e_{ij}|$ represents the length of the edge that
connects vertex $i$ to vertex $j$ in the simplex.
For $d=2$ this is the shape-measure used in~\cite{BaSm97} with
a slightly different normalization.
For $d=3$, the measure in~(\ref{eqn:shape}) is the shape-measure
developed in~\cite{LiJo94} 
again with a slightly different normalization.
The shape measure above can be shown to be equivalent to the sphere ratio
shape measure commonly used (cf.~\cite{LiJo94}).

\subsection{Discretization, adaptivity, and error estimation in MC}

Given a nonlinear weak form $\langle F(u),v \rangle$,
its linearization bilinear form
$\langle DF(u)w,v \rangle$,
a Dirichlet function $\bar{u}$, and a collection of simplices
representing the domain, \MC\ uses a default linear element to produce and then
solve the implicitly defined nonlinear algebraic equations for the
basis function coefficients in the
expansion~(\ref{eqn:soln}).
The user can also provide their own element, specifying the number of degrees
of freedom to be located on vertices, edges, faces, and in the interior of
simplices, along with a quadrature rule, and the values of the trial (basis)
and test functions at the quadrature points on the master element.
Different element types may be used for different components of a
coupled elliptic system.
The availability of a user-defined general element makes it possible to,
for example, use quadratic elements as would be required in elasticity
applications to avoid locking.

Once the equations are assembled and solved (discussed below),
{\em a posteriori} error estimates are computed from the discrete
solution to drive adaptive mesh refinement.
The idea of adaptive error control in finite element methods is to estimate
the behavior of the actual solution to the problem using only a previously
computed numerical solution, and then use the estimate to build an improved
numerical solution by upping the polynomial order
($p$-refinement) or refining the mesh ($h$-refinement) where appropriate.
Note that this approach to adapting the mesh (or polynomial order) to the
local solution behavior affects not only approximation quality, but also
solution complexity: if a target solution accuracy can be obtained with
fewer mesh points by their judicious placement in the domain, the cost of
solving the discrete equations is reduced (sometimes dramatically) because
the number of unknowns is reduced (again, sometimes dramatically).
Generally speaking, if an elliptic equation has a solution with local singular
behavior, such as would result from the presence of abrupt changes in the
coefficients of the equation, or a domain singularity,
then adaptive methods tend to give dramatic improvements over
non-adaptive methods in terms of accuracy achieved
for a given complexity price.
Two examples illustrating bisection-based adaptivity patterns
(driven by a completely geometrical ``error'' indicator simply
for illustration) are shown in Figure~\ref{fig:refinements}.
\begin{figure}[tbh]
\begin{center}
\mbox{\myfigpng{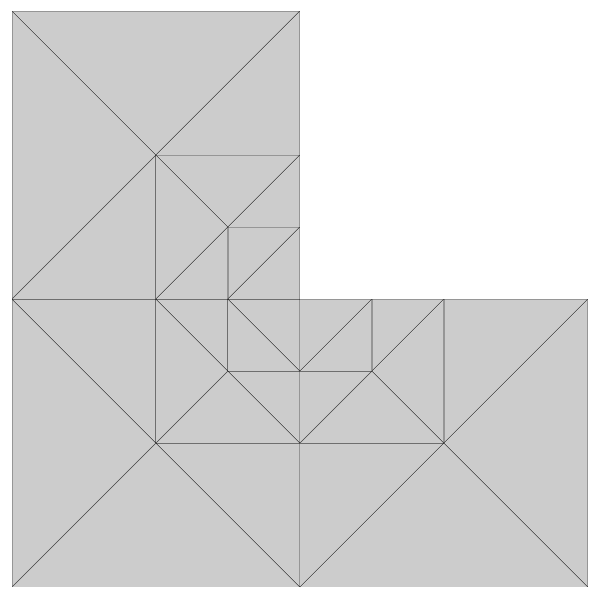}{1.5in}}
\mbox{\myfigpng{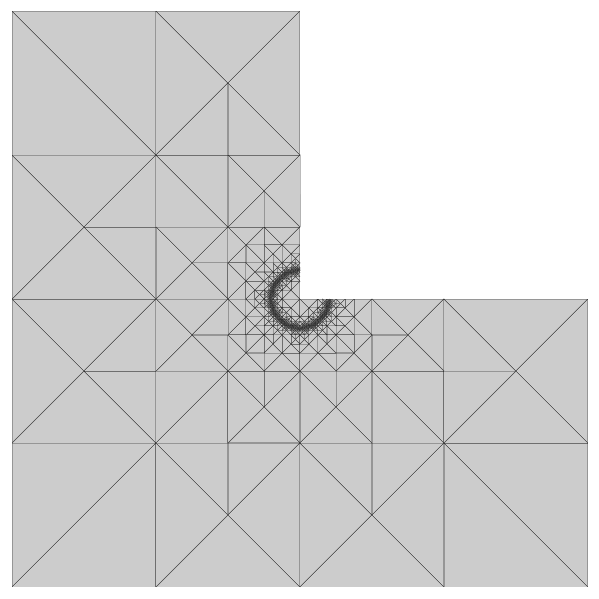}{1.5in}}
\mbox{\myfigpng{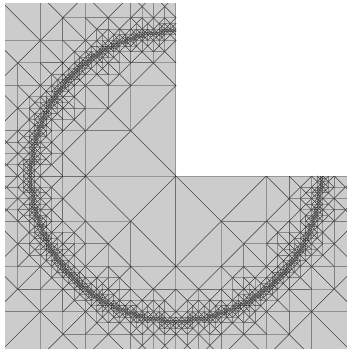}{1.5in}} \\
\mbox{\myfigpng{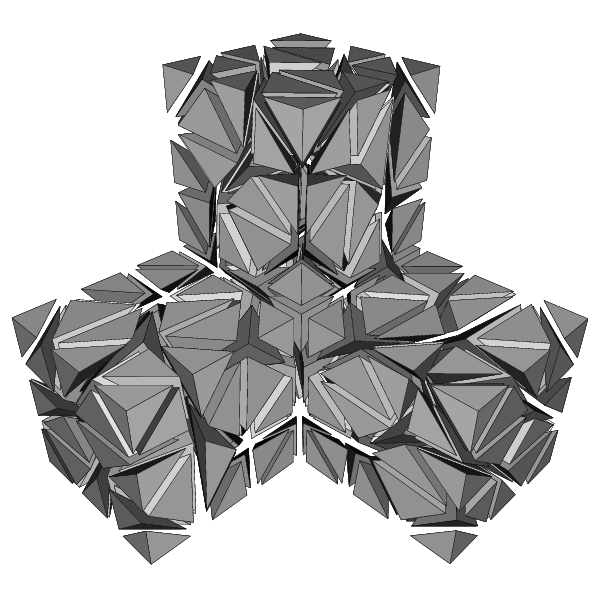}{1.5in}}
\mbox{\myfigpng{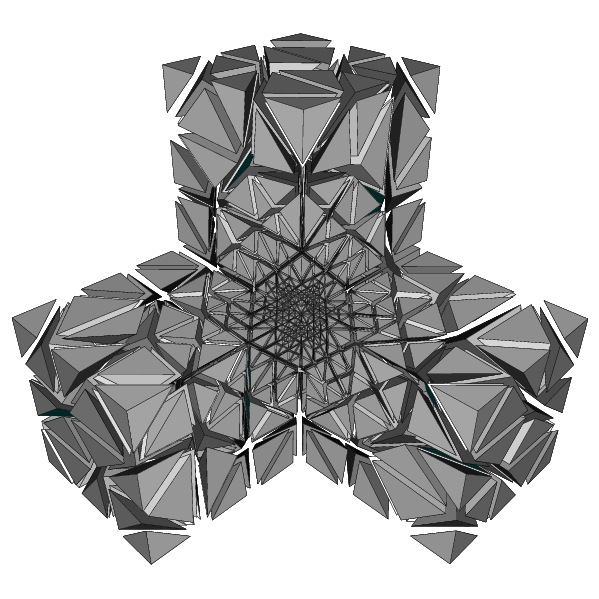}{1.5in}}
\mbox{\myfigpng{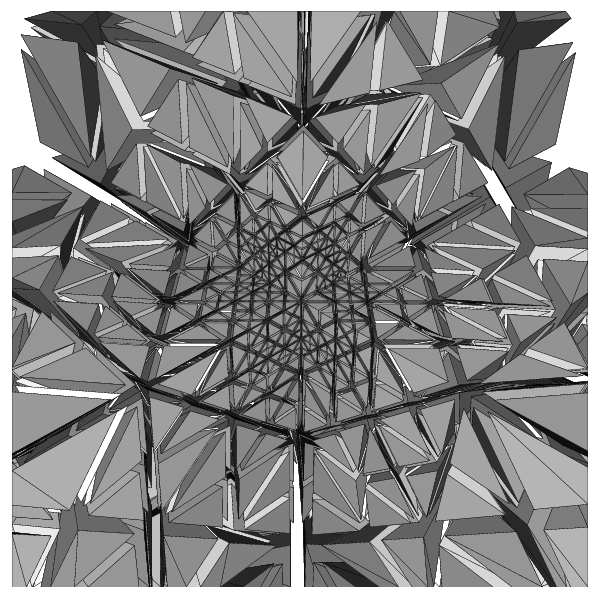}{1.5in}}
\end{center}
\caption{Examples illustrating the 2D and 3D adaptive mesh refinement
         algorithms in MC.  The right-most figure in each row shows
         a close-up of the area where most of the refinement occured
         in each example.}
  \label{fig:refinements}
\end{figure}

MC employs the error indicators derived in Section~\ref{sec:indicators}
adaptive solution of nonlinear elliptic systems of the
form~(\ref{eqn:str1})--(\ref{eqn:str3}).
In particular, the indicators are used to adaptively construct
Galerkin solutions satisfying~(\ref{eqn:galerkin}), which
approximate weak solutions satisfying~(\ref{eqn:weak}).
MC can be directed to use either the local residual
indicator~(\ref{eqn:estimator}) together with the
principle of error equi-distribution, or the duality-based
weighted residual indicator~(\ref{eqn:estimator_duality}),
again together with equi-distribution.

Of course, we can't perform the integrals in~(\ref{eqn:estimator})
or~(\ref{eqn:estimator_duality}) exactly in most cases,
so we employ quadrature in \MC.
Another option is to project the data onto the finite element spaces
involved and then to perform the integrals exactly;
this approach is analyzed carefully in~\cite{Verf94}.
Note that $\eta_s$ is computable by quadrature, since all terms appearing in
the definition depend only on the (available) computed solution $u_h$.
In particular, each of the terms
$$
   A^{ia}(x^b,(u_h)^j,(u_h)^k_{~;c})_{;a},
~~~
   A^{ia}(x^b,(u_h)^j,(u_h)^k_{~;c}) n_a,
$$
$$
   B^i(x^b,(u_h)^j,(u_h)^k_{~;c}),
~~~
   C^i(x^b,(u_h)^j,(u_h)^k_{~;c}),
$$
depend only on $u_h$, its first derivatives, and the normal vector $n_q$
(which is known from simple geometrical calculations).
All of the terms but the first are already provided by the user as
part of the weak form $\langle F(u),v \rangle$ required to use \MC.
The only problematic term is the first one; this represents the strong
form of the principle part of the equation, and must be supplied by the
user as a separate piece of information.

In order to understand this more completely, we will briefly describe
some of the problem specification details in \MC.
To use \MC\ to discretize problems of the
form~(\ref{eqn:weak})--(\ref{eqn:weakForm}), 
the user is expected to provide the Dirichlet function $\bar{u}$
by providing the function $E$ in~(\ref{eqn:str3}).
In addition, the user provides the nonlinear weak form:
\begin{eqnarray}
\langle F(u),v \rangle
    &=& \int_{\calg{M}} \calg{G}_{ij} (A^{ia} v^j_{~;a} + B^i v^j)~dx
     + \int_{\partial_N \calg{M}} \calg{G}_{ij} C^i v^j ~ds
\nonumber \\
    &=& \int_{\calg{M}} F_0(u)(v) ~dx
    + \int_{\partial_N \calg{M}} F_1(u)(v) ~ds,
\label{eqn:weak_app}
\end{eqnarray}
by providing the integrand function $F_t(u)(v)$ defined as:
\begin{equation}
  \label{eqn:fu_v}
F_t(u)(v)
~=~
   \left\{ \begin{array}{ll}
    \calg{G}_{ij} (A^{ia} v^j_{~;a} + B^i v^j), & ~~\text{if}~ t = 0, \\
    \calg{G}_{ij} C^i v^j,                      & ~~\text{if}~ t = 1. \\
    \end{array}
    \right.
\end{equation}
In order to use the inexact Newton iteration in \MC\ to produce a
Petrov-Galerkin approximation satisfying~(\ref{eqn:galerkin}),
the user must also provide a corresponding bilinear linearization form:
\begin{equation}
  \label{eqn:weakLinearized_app}
\langle DF(u)w,v \rangle
    = \int_{\calg{M}} DF_0(u)(w,v) ~dx
    + \int_{\partial_N \calg{M}} DF_1(u)(w,v) ~ds,
\end{equation}
where the integrand function $DF_t(u)(w,v)$ is defined
through Gateaux differentiation as described in
Section~\ref{sec:elliptic_general}.
In order to use the {\em a posteriori} error estimator in \MC,
the user must provide
an additional vector-valued function $SF_t(u)$, defined as:
\begin{equation}
  \label{eqn:sf_u}
SF_t(u)
~=~
   \left\{ \begin{array}{ll}
    B^i - A^{ia}_{~~;a}, & ~~\text{if}~ t = 0, \\
    C^i + A^{ia} n_a,    & ~~\text{if}~ t = 1, \\
    A^{ia} n_a,          & ~~\text{if}~ t = 2. \\
    \end{array}
    \right.
\end{equation}
The key point that must be emphasized here is the following:
since \MC\ employs
quadrature to evaluate the integrals appearing
in each of~(\ref{eqn:weak_app}), (\ref{eqn:weakLinearized_app}),
and~(\ref{eqn:estimator}),
{\em 
the user-provided functions $F_t(u)(v)$,
$DF_t(u)(w,v)$, and $SF_t(u)$
only need to be evaluated at a single point $x^p \in \calg{M}$ at a time.
}
In other words, the user can simply evaluate the expressions
in~(\ref{eqn:fu_v}) and in~(\ref{eqn:sf_u}) as if they were point vectors
and point tensors, rather than vector and tensor fields.
This is one of the most powerful features of \MC, and of nonlinear
finite element software in general.
It implies that the user-defined functions
$F_t(u)(v)$,
$DF_t(u)(w,v)$, and $SF_t(u)$ can usually be implemented
to appear in software exactly as they do on paper.

The remaining quantities appearing in the estimator~(\ref{eqn:estimator}),
namely the normal vector $n_q$ and the mesh parameters $h_s$ and $h_f$,
are completely geometrical and can be computed from the
local simplex geometry information.
The indicator is very inexpensive when compared to the typical cost of
producing the discrete solution $u_h$ itself; the number of function
evaluations and arithmetic operations (for performing quadrature)
is always linear in the total number of simplices.
Moreover, the indicator is completely {\em local};
it can be computed chart-wise when multiple coordinate
systems are employed.
These ideas are explored more fully in~\cite{Hols97c}.

\subsection{Solution of linear and nonlinear systems in MC}
   \label{sec:mc_solvers}

When a system of nonlinear finite element equations must be solved in \MC,
the global inexact-Newton Algorithm~\ref{alg:newton} is employed,
where the linearization systems are solved by linear multilevel methods.
When necessary, the Newton procedure in Algorithm~\ref{alg:newton} is
supplemented with a user-defined normalization equation for performing
an augmented system continuation algorithm.
The linear systems arising as the Newton equations in each iteration of
Algorithm~\ref{alg:newton} are solved using a completely algebraic multilevel
algorithm.
Either refinement-generated prolongation matrices $P_k$, or user-defined
prolongation matrices $P_k$ in a standard YSMP-row-wise sparse matrix format,
are used to define the multilevel hierarchy algebraically.
In particular, once the single ``fine'' mesh is used to produce the
discrete nonlinear problem $F(u)=0$ along with its linearization $Au=f$
for use in the Newton iteration in Algorithm~\ref{alg:newton}, a $J$-level
hierarchy of linear problems is produced algebraically using the following
recursion:
$$
A_{k+1} = P_k^T A_k P_k,
  \ \ \ \ \ \ \ \ k=1,\ldots,J-1,
  \ \ \ \ \ \ \ \ A_1 \equiv A.
$$
As a result, the underlying multilevel algorithm is provably convergent
in the case of self-adjoint-positive matrices~\cite{HoVa97a}.
Moreover, the multilevel algorithm has provably optimal $O(N)$ convergence
properties under the standard assumptions for uniform
refinements~\cite{Xu92a},
and is nearly-optimal $O(N \log N)$ under very weak assumptions on
adaptively refined problems~\cite{BDY88,Akso01}.
In the adaptive setting, a stabilized (approximate wavelet) hierarchical
basis method is employed~\cite{Akso01}.
External software can also be used to generate the prolongation matrices,
so that a number of different graph theory-based algebraic multilevel
coarsening algorithms may be used to generate the subspace hierarchy.

Coupled with the superlinear convergence properties of the outer inexact
Newton iteration in Algorithm~\ref{alg:newton}, this leads to
an overall complexity of $O(N)$ or $O(N \log N)$ for the solution of the
discrete nonlinear problems in Step~1 of Algorithm~\ref{alg:adapt}.
Combining this low-complexity solver with the judicious placement of unknowns
only where needed due to the error estimation in Step~2 and the subdivision
algorithm in Steps~3-6 of Algorithm~\ref{alg:adapt}, leads to a very
effective low-complexity approximation technique for solving a general
class of nonlinear elliptic systems on 2- and 3-manifolds.

\subsection[Parallel computing in MC: the PPUM method]
           {Parallel computing in MC:
            The Parallel Partition of Unity Method (PPUM)}
\label{sec:ppum}

\MC\ incorporates a new approach to the use of parallel computers
with adaptive finite element methods, based on combining the
{\em Partition of Unity Method (PUM)} of Babu\v{s}ka and Melenk~\cite{BaMe97}
with local error estimate techniques of Xu and Zhou~\cite{XuZh97}.
The algorithm, which we refer to as the {\em Parallel Partition of
Unity Method (PPUM)}, is described in detail in~\cite{BaHo98a,BHMP98}.
The idea of the algorithm is as follows.
\begin{algorithmX}
   \label{alg:ppum}
{\em (PPUM - Parallel Partition of Unity Method~\cite{BaHo98a})}
\label{alg:bh}
\begin{center}
\begin{enumerate}
\item Discretize and solve the problem using a global coarse mesh.
\item Compute {\em a posteriori} error estimates using the coarse solution, and decompose the mesh to achieve equal error using weighted spectral or inertial bisection.
\item Give the entire mesh to a collection of processors, where each processor will perform a completely independent solve-estimate-refine loop (Step~2 through Step~6 in Algorithm~\ref{alg:adapt}), restricting local refinement to only an assigned portion of the domain.  The portion of the domain assigned to each processor coincides with one of the domains produced by spectral bisection with some overlap (produced by conformity algorithms, or by explicitly enforcing substantial overlap).  When a processor has reached an error tolerance locally, computation stops on that processor.
\item Combine the independently produced solutions using a partition of unity subordinate to the overlapping subdomains.
\end{enumerate}
\end{center}
\end{algorithmX}
While the algorithm above seems to ignore the global coupling of the
elliptic problem, some recent theoretical results~\cite{XuZh97}
support this as provably good, and even optimal in some cases.
The principle idea underlying the results in~\cite{XuZh97} is that while
elliptic problems are globally coupled, this global coupling is essentially
a ``low-frequency'' coupling, and can be handled on the initial mesh which
is much coarser than that required for approximation accuracy considerations.
This idea has been exploited, for example, in~\cite{Xu92b,Xu92c}, 
and is in fact 
why the construction of a coarse problem in overlapping domain decomposition
methods is the key to obtaining convergence rates which are independent
of the number of subdomains (c.f.~\cite{Xu92a}).
A more complete description can be found in~\cite{BaHo98a},
along with examples using \MC\ and the 2D adaptive finite element
package PLTMG~\cite{PLTMG}. 
An analysis of the global $L^2$- and $H^1$-error in solutions produced by
the algorithm appears in the next section.
An example showing the types of local refinements that occur within each
subdomain is depicted in Figure~\ref{fig:parallel}.
\begin{figure}[tbh]
\begin{center}
\mbox{\myfigpdf{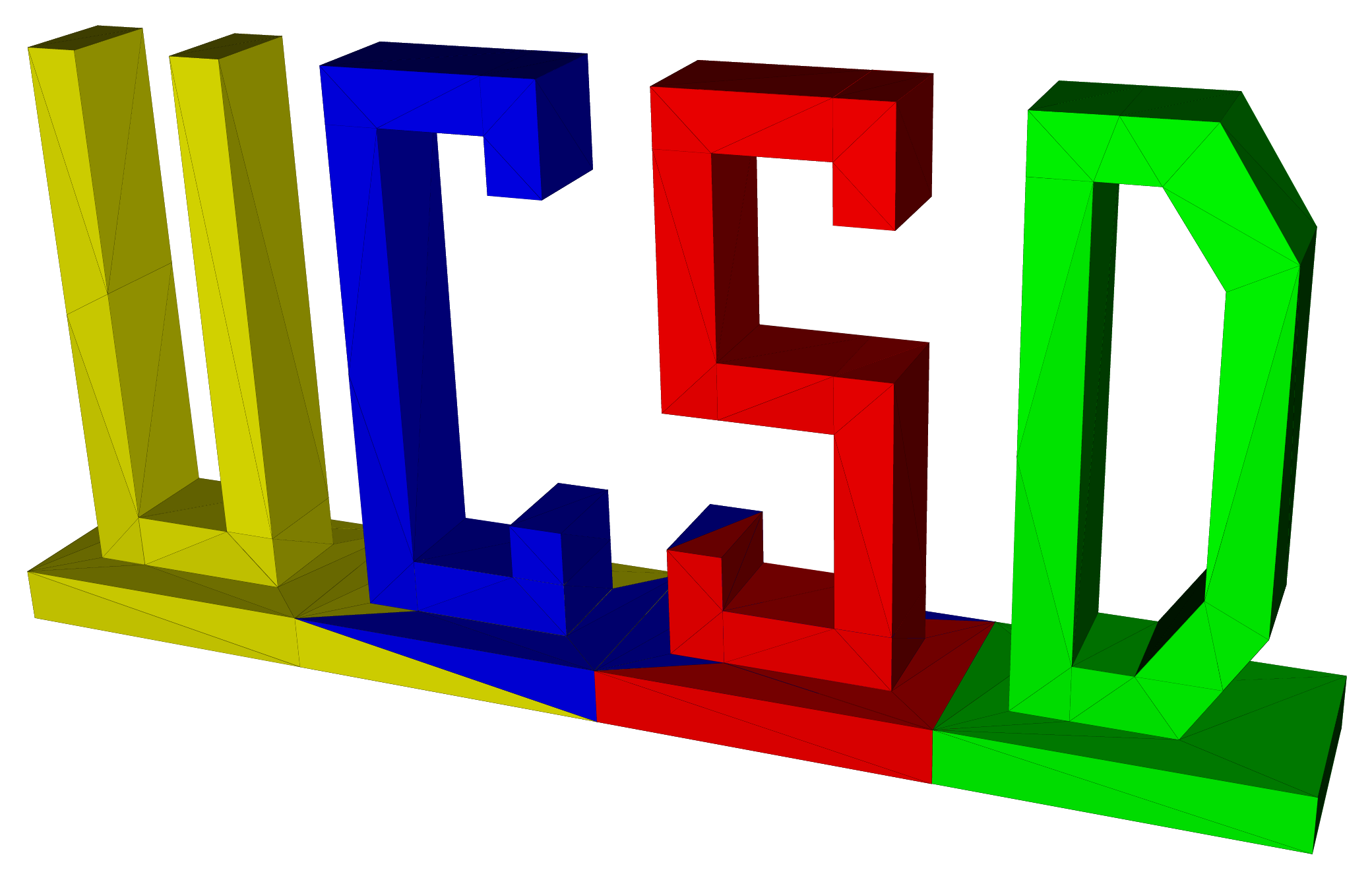}{1.1in}}
\mbox{\myfigpdf{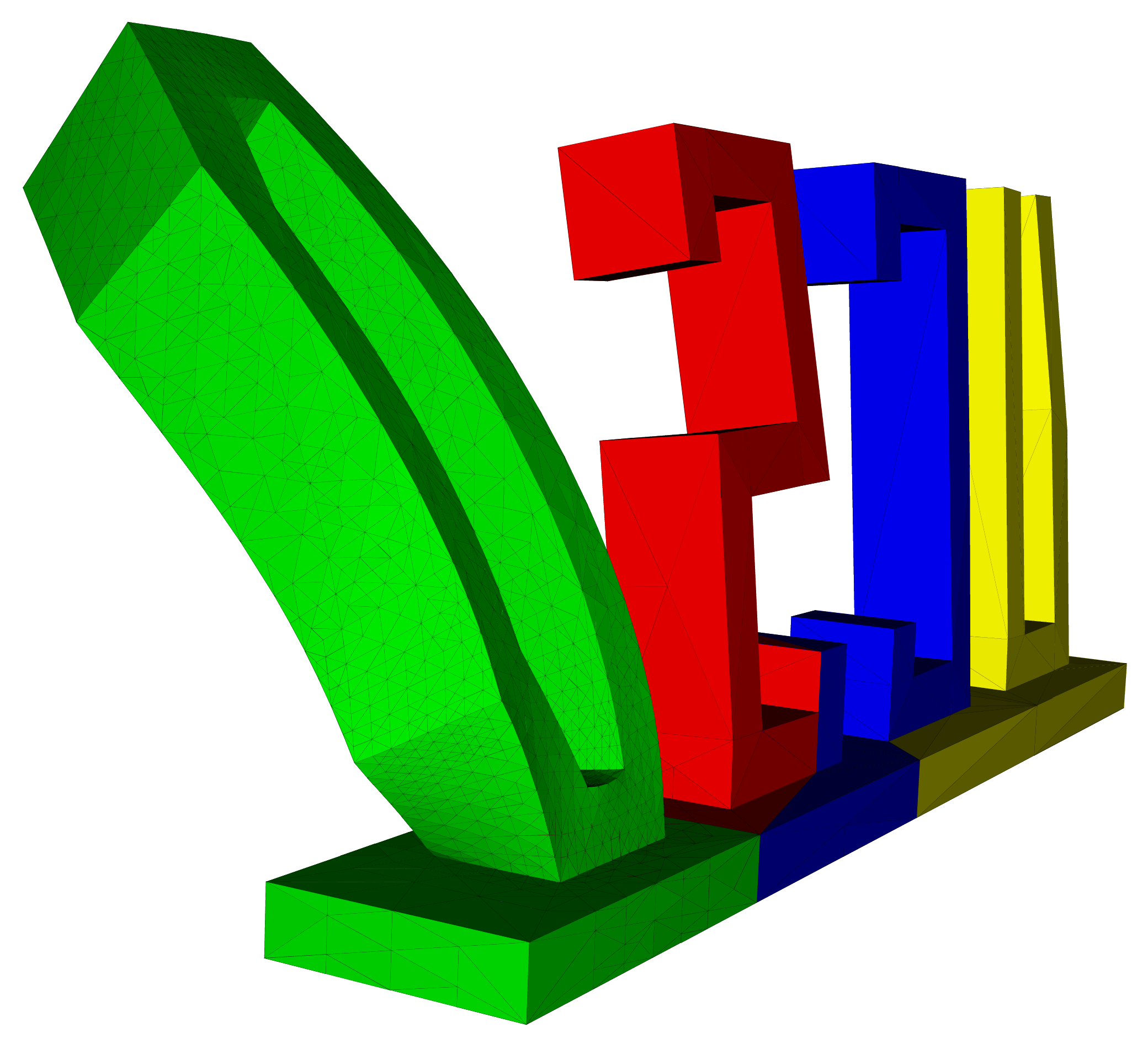}{1.2in}}
\mbox{\myfigpdf{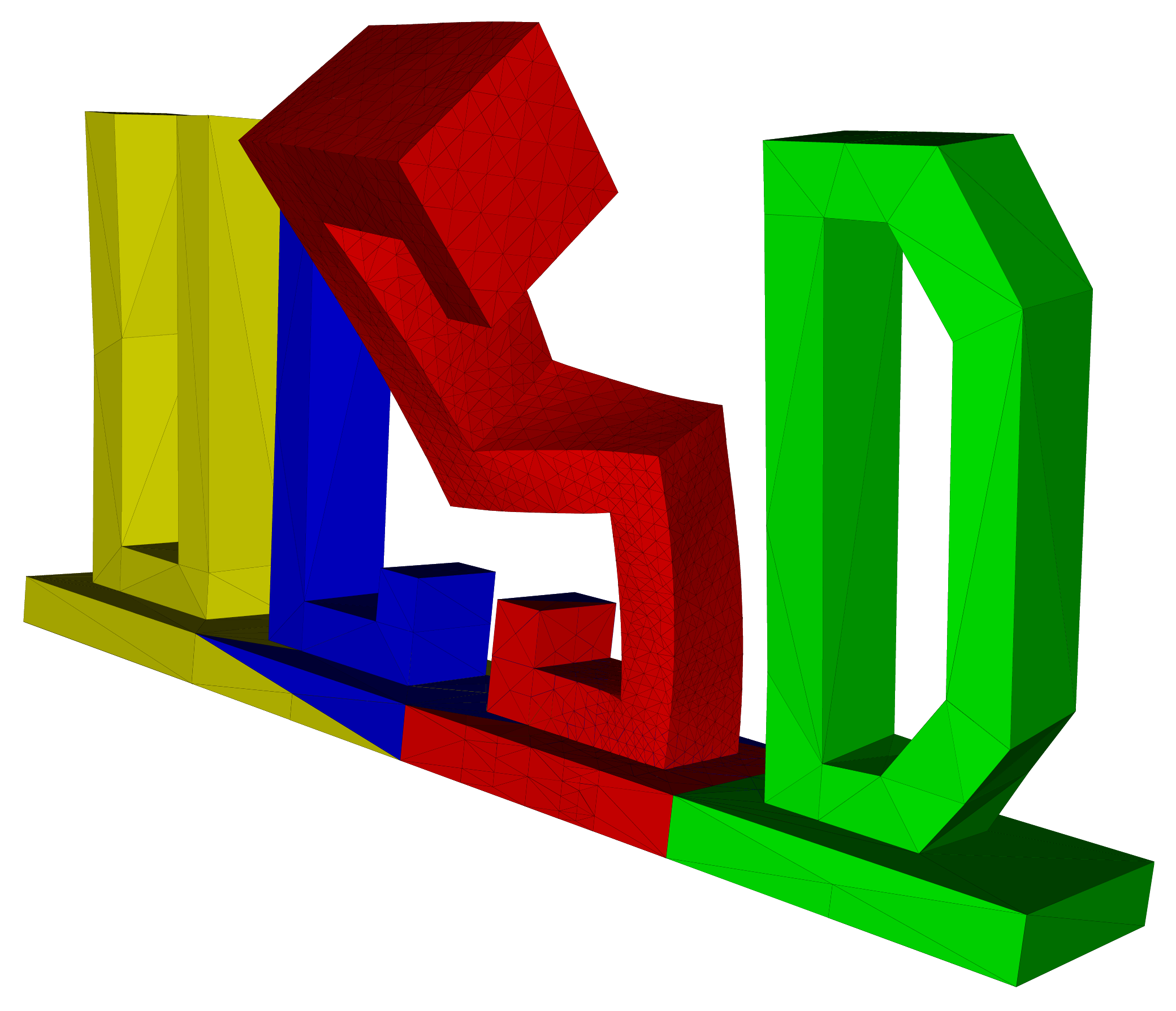}{1.2in}}
\end{center}
\caption{An example showing the types of local refinements that are
         created by PPUM.}
  \label{fig:parallel}
\end{figure}

\subsection{Global $L^2$- and $H^1$-error estimates for PPUM}

In order to analyze the error behavior in PPUM, we first review the
partition of unity method (PUM) of
Babu\v{s}ka and Melenk~\cite{BaMe97}.
Let $\Omega \subset \bbbb{R}^d$ be an open set and let $\{ \Omega_i \}$
be an open cover of $\Omega$ with a bounded local overlap property:
For all $x \in \Omega$, there exists a constant $M$ such that
\begin{equation}
   \label{eqn:pum_overlap}
\sup_i \{~ i ~|~ x \in \Omega_i ~\} \le M.
\end{equation}
A Lipschitz {\em partition of unity} $\{ \phi_i \}$
subordinate to the cover $\{ \Omega_i \}$ satisfies the
following five conditions:
\begin{eqnarray}
\sum_i \phi_i(x) &\equiv& 1, ~~~ \forall x \in \Omega, 
\label{eqn:pum_pou1}    \\
\phi_i &\in& C^k(\Omega) ~~~ \forall i, ~~~ (k \ge 0),
    \\
\sup \phi_i &\subset& \overline{\Omega}_i, ~~~ \forall i,
    \\
\| \phi_i \|_{L^{\infty}(\Omega)} &\le& C_{\infty}, ~~~ \forall i,
    \\
\| \nabla \phi_i \|_{L^{\infty}(\Omega)} 
   &\le& \frac{C_G}{\text{diam}(\Omega_i)}, ~~~ \forall i.
\label{eqn:pum_pou5}
\end{eqnarray}
The {\em partition of unity method (PUM)}
builds an approximation $u_{ap} = \sum_i \phi_i v_i$ where the
$v_i$ are taken from the local approximation spaces:
\begin{equation}
   \label{eqn:pum_local_spaces}
V_i \subset C^k(\Omega \cap \Omega_i)
     \subset H^1(\Omega \cap \Omega_i), ~~~ \forall i, ~~~ (k \ge 0).
\end{equation}
The following simple lemma makes possible several useful results.
\begin{lemma} \label{lemma:pum}
Let $w,w_i \in H^1(\Omega)$ with
$\text{supp}~ w_i \subseteq \overline{\Omega \cap \Omega_i}$.
Then
\begin{eqnarray*}
\sum_i \|w\|_{H^k(\Omega_i)}^2 
   &\le& M \|w\|_{H^k(\Omega)}^2, ~~~~ k=0,1
\\
\| \sum_i w_i \|_{H^k(\Omega)}^2 
   &\le& M \sum_i \|w_i\|_{H^k(\Omega \cap \Omega_i)}^2, ~~~~ k=0,1
\end{eqnarray*}
\end{lemma}
\begin{proof}
The proof follows from~(\ref{eqn:pum_overlap})
and~(\ref{eqn:pum_pou1})--(\ref{eqn:pum_pou5});
see~\cite{BaMe97}.\qed
\end{proof}
The basic approximation properties of PUM are as follows.
\begin{theorem} [Babu\v{s}ka and Melenk~\cite{BaMe97}]
   \label{thm:pum}
If the local spaces $V_i$ have the
following approximation properties:
\begin{eqnarray*}
\|u-v_i\|_{L^2(\Omega \cap \Omega_i)} &\le& \epsilon_0(i),
   ~~~ \forall i, \\
\|\nabla(u-v_i)\|_{L^2(\Omega \cap \Omega_i)} &\le& \epsilon_1(i), 
   ~~~ \forall i,
\end{eqnarray*}
then the following {\em a priori} global error estimates hold:
\begin{eqnarray*}
\|u-u_{ap}\|_{L^2(\Omega)}
   &\le& \sqrt{M} C_{\infty} \left( \sum_i \epsilon_0^2(i) \right)^{1/2}, \\
\|\nabla(u-u_{ap})\|_{L^2(\Omega)}
   &\le& \sqrt{2M} \left( 
     \sum_i \left( \frac{C_G}{\text{diam}(\Omega_i)} \right)^2
         \epsilon_1^2(i)
     + C_{\infty}^2 \epsilon_0^2(i) \right)^{1/2}.
\end{eqnarray*}
\end{theorem}
\begin{proof}
This follows from Lemma~\ref{lemma:pum} by taking
$u-u_{ap} = \sum_i \phi_i(u-v_i)$ and $w_i = \phi_i(u-v_i)$.\qed
\end{proof}

We now give a global $H^1$-error estimate of the PPUM
adaptive algorithm proposed in~\cite{BaHo98a}.
We can view PPUM as building a PUM approximation
$u_{pp} = \sum_i \phi_i v_i$ where the $v_i$ are taken from the
local spaces:
\begin{equation}
   \label{eqn:ppum_local_spaces}
V_i = \calg{X}_i V_i^g \subset C^k(\Omega \cap \Omega_i)
     \subset H^1(\Omega \cap \Omega_i), ~~~ \forall i, ~~~ (k \ge 0),
\end{equation}
where $\calg{X}_i$ is the characteristic function for $\Omega_i$, and where
\begin{equation}
   \label{eqn:ppum_local_spaces2}
V_i^g \subset C^k(\Omega) \subset H^1(\Omega), ~~~ \forall i, ~~~ (k \ge 0).
\end{equation}
In PPUM, the global spaces $V_i^g$
in~(\ref{eqn:ppum_local_spaces})--(\ref{eqn:ppum_local_spaces2})
are built from locally enriching an initial coarse global space $V_0$
by locally adapting the finite element mesh on which $V_0$ is built.
(This is in contrast to classical overlapping schwarz domain decomposition
methods where local spaces are often built through enrichment of
$V_0$ by locally adapting the mesh on which $V_0$ is built,
and then removing the portions of the mesh exterior to the
adapted region.)
The PUM space $V$ is then
\begin{eqnarray*}
V &=& \left\{~ v ~|~ v = \sum_i \phi_i v_i, ~~~ v_i \in V_i ~\right\} \\
  &=& \left\{~ v ~|~ v = \sum_i \phi_i \calg{X}_i v_i^g = \sum_i \phi_i v_i^g, 
     ~~~ v_i^g \in V_i^g ~\right\} \subset H^1(\Omega).
\nonumber
\end{eqnarray*}

Consider now the following linear elliptic problem in the plane:
\begin{equation}
   \label{eqn:linmod}
\begin{array}{rcl}
- \nabla \cdot (a \nabla u) &=& f ~\text{in}~ \Omega, \\
                          u &=& 0 ~\text{on}~ \partial \Omega,
\end{array}
\end{equation}
where $a_{ij} \in W^{1,\infty}(\Omega)$, $f \in L^2(\Omega)$,
$a_{ij} \xi_i \xi_j \ge a_0 > 0$, $\forall \xi_i \ne 0$,
where $\Omega \subset \bbbb{R}^2$ is a convex polygon.
(The results below also hold more generally for classes
of two- and three-dimensional nonlinear problems.)
A weak formulation is:
$$
\text{Find}~u \in H^1_0(\Omega) ~\text{such~that}~
     \langle F(u),v \rangle = 0, ~~~\forall v \in H^1_0(\Omega),
$$
where
$$
\langle F(u),v \rangle = \int_{\Omega} a \nabla u \cdot \nabla v ~dx
                       - \int_{\Omega} f v ~dx.
$$
The PUM is usually used to solve a PDE as a Galerkin
method in the globally coupled PUM space (cf.~\cite{GrSc00}):
$$
\text{Find}~ u_{ap} \in V \subset H^1_0(\Omega) ~\text{s.t.}~
    \langle F(u_{ap}),v \rangle = 0, ~\forall v \in V \subset H^1_0(\Omega).
$$
In contrast, PPUM proposed in~\cite{BaHo98a}
builds an approximation $u_{pp}$ from
decoupled local Galerkin solutions:
\begin{equation}
   \label{eqn:ppum_upp}
u_{pp} = \sum_i \phi_i u_i
       = \sum_i \phi_i u_i^g,
\end{equation}
where each $u_i^g$ satisfies:
\begin{equation}
   \label{eqn:ppum_uig}
\text{Find}~ u_i^g \in V_i^g ~\text{such~that}~
       \langle F(u_i^g),v_i^g \rangle = 0, ~~~\forall v_i^g \in V_i^g.
\end{equation}
We have the following global error estimate for the approximation $u_{pp}$ 
in~(\ref{eqn:ppum_upp}) built from~(\ref{eqn:ppum_uig})
using the local PPUM parallel algorithm.
\begin{theorem} \label{thm:bh}
Assume the solution to~(\ref{eqn:linmod}) satisfies
$u \in H^{1+\alpha}(\Omega)$, $\alpha > 0$, and assume that
quasi-uniform meshes of sizes $h$ and $H > h$ are used for $\Omega_i^0$
and $\Omega \backslash \Omega_i^0$ respectively.
If $\text{diam}(\Omega_i) \ge 1/Q > 0 ~~\forall i$, then
the global solution $u_{pp}$ in~(\ref{eqn:ppum_upp}) produced by the PPUM
Algorithm~\ref{alg:bh} satisfies the following global error bounds:
\begin{eqnarray*}
\|u-u_{pp}\|_{L^2(\Omega)}
   &\le& \sqrt{PM} C_{\infty} 
         \left( C_1 h^{\alpha} + C_2 H^{1+\alpha} \right), \\
\|\nabla(u-u_{pp})\|_{L^2(\Omega)}
   &\le& \sqrt{2PM (Q^2 C_G^2 + C_{\infty}^2)}
         \left( C_1 h^{\alpha} + C_2 H^{1+\alpha} \right),
\end{eqnarray*}
where $P = $ number of local spaces $V_i$.
Further, if $H \le h^{\alpha/(1+\alpha)}$ then:
\begin{eqnarray*}
\|u-u_{pp}\|_{L^2(\Omega)}
   &\le& \sqrt{PM} C_{\infty} \max\{C_1,C_2\} h^{\alpha}, \\
\|\nabla(u-u_{pp})\|_{L^2(\Omega)}
    &\le& \sqrt{2PM (Q^2 C_G^2 + C_{\infty}^2)} \max\{C_1,C_2\} h^{\alpha},
\end{eqnarray*}
so that the solution produced by Algorithm~\ref{alg:bh}
is of optimal order in the $H^1$-norm.
\end{theorem}
\begin{proof}
Viewing PPUM as a PUM gives access to the PUM {\em a priori}
estimates in Theorem~\ref{thm:pum};
these require local estimates of the form:
\begin{eqnarray*}
\|u-u_i\|_{L^2(\Omega \cap \Omega_i)} 
  =   \|u-u_i^g\|_{L^2(\Omega \cap \Omega_i)} &\le& \epsilon_0(i), \\
\|\nabla(u-u_i)\|_{L^2(\Omega \cap \Omega_i)} 
  =   \|\nabla(u-u_i^g)\|_{L^2(\Omega \cap \Omega_i)} &\le& \epsilon_1(i).
\end{eqnarray*}
Such local {\em a priori} estimates are available
for problems of the form~(\ref{eqn:linmod})~\cite{NiSc74,XuZh97}.
They can be shown to take the following form:
$$
\| u - u_i^g \|_{H^1(\Omega_i \cap \Omega)} \le C \left(
    \inf_{v_i^0 \in V_i^0} \|u - v_i^0\|_{H^1(\Omega_i^0 \cap \Omega)}
  + \| u - u_i^g \|_{L^2(\Omega)} \right)
$$
where
$$
   V_i^0 \subset C^k(\Omega_i^0 \cap \Omega) \subset H^1(\Omega_i \cap \Omega),
$$
and where
$$
   \Omega_i \subset \subset \Omega_i^0,
   \ \ \ \ \ \Omega_{ij} = \Omega_i^0 \bigcap \Omega_i^0,
   \ \ \ \ \ |\Omega_{ij} | \approx | \Omega_i | \approx | \Omega_j |.
$$
Since we assume $u \in H^{1+\alpha}(\Omega)$, $\alpha > 0$,
and since quasi-uniform meshes of sizes $h$ and $H > h$ are used
for $\Omega_i^0$ and $\Omega \backslash \Omega_i^0$ respectively, we have:
\begin{eqnarray*}
\| u - u_i^g \|_{H^1(\Omega_i \cap \Omega)}
&=& \left(
\| u - u_i^g \|_{L^2(\Omega_i \cap \Omega)}^2
+ \| \nabla( u - u_i^g) \|_{L^2(\Omega_i \cap \Omega)}^2
\right)^{1/2}
\nonumber \\
&\le& C_1 h^{\alpha} + C_2 H^{1 + \alpha}.
\end{eqnarray*}
I.e., in this setting we can use
$\epsilon_0(i) = \epsilon_1(i) = C_1 h^{\alpha} + C_2 H^{1 + \alpha}$.
The {\em a priori} PUM estimates in Theorem~\ref{thm:pum} then become:
\begin{eqnarray*}
\|u-u_{pp}\|_{L^2(\Omega)}
   &\le& \sqrt{M} C_{\infty} \left(
           \sum_i (C_1 h^{\alpha} + C_2 H^{1+\alpha})^2 \right)^{1/2},
\\
\|\nabla(u-u_{pp})\|_{L^2(\Omega)}
   &\le& \sqrt{2M} 
\nonumber
\end{eqnarray*}
$$
\cdot
\left( \left[
     \sum_i \left( \frac{C_G}{\text{diam}(\Omega_i)} \right)^2
     + C_{\infty}^2 
      \right]
         (C_1 h^{\alpha} + C_2 H^{1+\alpha})^2
         \right)^{1/2}.
$$
If $P = $ number of local spaces $V_i$,
and if $\text{diam}(\Omega_i) \ge 1/Q > 0 ~~\forall i$, this is simply:
\begin{eqnarray*}
\|u-u_{pp}\|_{L^2(\Omega)}
   &\le& \sqrt{PM} C_{\infty} 
         \left( C_1 h^{\alpha} + C_2 H^{1+\alpha} \right), \\
\|\nabla(u-u_{pp})\|_{L^2(\Omega)}
   &\le& \sqrt{2PM (Q^2 C_G^2 + C_{\infty}^2)}
         \left( C_1 h^{\alpha} + C_2 H^{1+\alpha} \right).
\end{eqnarray*}
If $H \le h^{\alpha/(1+\alpha)}$ then $u_{pp}$ from PPUM is
asymptotically as good as a global Galerkin solution when
the error is measured in the $H^1$-norm.\qed
\end{proof}
Estimates similar to Theorem~\ref{thm:bh} appear in~\cite{XuZh97} for
a variety of related parallel algorithms.
Note that improving the estimates in the $L^2$-norm is not possible;
the required local estimates simply do not hold.
Improving the solution quality in the $L^2$-norm would require
more global information.

\subsection[Availability of MC and the supporting tools]
           {Availability of MC and the supporting tools MALOC and SG}

\MC\ is built on top of a low-level portability library called
\MALOC\ (Minimal Abstraction Layer for Object-oriented C).
Most of the images appearing in this paper were produced using a software
tool called \SG\ (Socket Graphics), which is also built on top of \MALOC.
\MALOC, \MC, and \SG\ were developed by the author over several years,
with generous contributions from a number of colleagues.
\MALOC, \MC, and \SG\ are freely redistributable under the
GNU General Public License (GPL), and the source code for all
three packages is freely available at the following website:
\begin{center}
{\tt http://www.scicomp.ucsd.edu/\~{}mholst/}
\end{center}
\MALOC, \MC, and \SG,
as well as a fully functional MATLAB version of \MC\ called \MCLAB,
are part of a larger project called \FETK\ (The Finite Element Toolkit).
Information about \FETK\ can be found at:
\begin{center}
{\tt http://www.fetk.org}
\end{center}

\section[Example: The constraints in the Einstein equations]
        {Example: The Hamiltonian and momentum constraints
         in the Einstein equations}
  \label{sec:constraints}

The evolution of the gravitational field was conjectured by Einstein
to be governed by twelve coupled first-order hyperbolic equations for the
metric of space-time and its time derivative, where the evolution is
constrained for all time by a coupled four-component elliptic system.
This four-component elliptic system consists of a nonlinear scalar
{\em Hamiltonian constraint}, and a linear 3-vector {\em momentum}
constraint.
The evolution and constraint equations, similar in some respects to
Maxwell's equations, are collectively referred to as
the {\em Einstein equations}.
Solving the constraint equations numerically, separately or together
with the evolution equations, is currently of great interest
to the physics community (cf.~\cite{HoBe01,HoBe99,BeHo99} for more detailed
discussions of this application).

The Hamiltonian and momentum constraints in the Einstein equations,
taken separately or together as a coupled system, have the form
(\ref{eqn:str1})--(\ref{eqn:str3}).
Allowing for both Dirichlet and Robin boundary conditions as are
typically used in black hole and neutron star models
(cf.~\cite{HoBe01,HoBe99,BeHo99}), the strong form can be written as:
\begin{eqnarray}
\hat{\Delta}\phi & = & \frac{1}{8} \hat{R} \phi
     + \frac{1}{12} ({\rm tr} K)^2 \phi^5
\label{eqn:ham_str1} \\
& &
     - \frac{1}{8} (\Ahatstar_{ab} + (\hat{L}W)_{ab})^2 \phi^{-7}
     - 2 \pi \hat{\rho} \phi^{-3} ~~\text{in}~\calg{M},
\nonumber \\
\hat{n}_a \hat{D}^a \phi + c \phi
     & = & z ~\text{on}~\partial_1 \calg{M}, \label{eqn:ham_str2} \\
\phi & = & f ~\text{on}~\partial_0 \calg{M}, \label{eqn:ham_str3} \\
  \hat{D}_b(\hat{L}W)^{ab} & = & \frac{2}{3} \phi^6 \hat{D}^a {\rm tr}K 
                             + 8 \pi \hat{j}^a
      ~~\text{in}~\calg{M}, \label{eqn:mom_str1} \\
(\hat{L}W)^{ab} \hat{n}_b + C^a_{~b} W^b
    & = & Z^a ~\text{on}~\partial_1 \calg{M}, \label{eqn:mom_str2} \\
  W^a & = & F^a ~~\text{on}~\partial_0 \calg{M},  \label{eqn:mom_str3}
\end{eqnarray}
where the following standard notation has been employed:
\begin{eqnarray*}
\hat{\Delta}\phi & = & \hat{D}_a \hat{D}^a \phi, \\
(\hat{L}W)^{ab} & = & \hat{D}^a W^b + \hat{D}^b W^a 
                - \frac{2}{3} \hat{\gamma}^{ab} \hat{D}_c W^c, \\
{\rm tr} K & = & \gamma^{ab}K_{ab}, \\
(C_{ab})^2 & = & C^{ab}C_{ab}.
\end{eqnarray*}
The symbols in the equations ($\hat{R}$, $K$, $\Ahatstar_{ab}$,
 $\hat{\rho}$, $\hat{j}^a$, $z$, $Z^a$, $f$, $F^a$, $c$, and $C^a_b$)
represent various physical parameters, and are described in detail
in~\cite{HoBe01,HoBe99,BeHo99} and the referenences therein.

Equations (\ref{eqn:ham_str1})--(\ref{eqn:mom_str3})
are known to be well-posed only for restricted problem data and
manifold topologies~\cite{MuYo73,MuYo74a,MuYo74b}.
Below we will present two well-posedness results from~\cite{HoBe01}
which hold under certain assumptions.
Note that if multiple solutions in the form of folds or
bifurcations are present in solutions of
(\ref{eqn:ham_str1})--(\ref{eqn:mom_str3}) then 
path-following numerical methods will be required for numerical
solution~\cite{Kell87,Kell92}.

\subsection{Weak formulation, linearization, and well-posedness}
   \label{sec:weakForm_gr}

Both the Hamiltonian constraint (\ref{eqn:ham_str1}) and the
momentum constraint (\ref{eqn:mom_str1}), taken separately or as a system,
fall into the class of
second-order divergence-form elliptic systems of tensor
equations in~(\ref{eqn:str1})--(\ref{eqn:str3}).
Derivation of the weak formulation
produces a weak system of the form~(\ref{eqn:weak})--(\ref{eqn:weakForm}),
with some interesting twists along the way described in~\cite{HoBe01}.
Following the notation in~\cite{HoBe01}, we employ a background
(or {\em conformal})
metric $\hat{\gamma}_{ab}$ to define the volume element
$dx = \sqrt{\text{det}~\hat{\gamma}_{ab}}~dx^1 dx^2 dx^3$
and the corresponding boundary volume element $ds$,
and for use as the manifold connection for covariant differentiation.
The notation for covariant differentiation using the conformal connection
will be denoted $\hat{D}_a$ to be consistent with the relativity literature,
and the various quantities from Section~\ref{sec:elliptic_general}
will now be hatted to denote use of this conformal metric.
For example, the unit normal to $\partial \calg{M}$ will now
be denoted $\hat{n}^a$.

Ordering the Hamiltonian constraint first in the 
system~(\ref{eqn:str1}), and defining the product metric
$\calg{G}_{ij}$ and the vectors $u^i$ and $v^j$ appearing
in~(\ref{eqn:product_metric}) and~(\ref{eqn:weakForm}) as:
$$
\calg{G}_{ij} = \left[ \begin{array}{ccc}
         1 & 0      \\
         0 & g_{ab} \\
         \end{array} \right],
~~~~~
u^i = \left[ \begin{array}{c}
         \phi \\
         W^a  \\
         \end{array} \right],
~~~~~
v^j = \left[ \begin{array}{c}
         \psi \\
         V^b  \\
         \end{array} \right],
$$
it is shown in~\cite{HoBe01} that the coupled Hamiltonian and momentum
constraints have a coupled weak formulation in the form of~(\ref{eqn:weak}),
where the form definition is as follows:
\begin{equation}
   \label{eqn:weakHamMom}
\langle F(u),v \rangle
= \langle F([\phi,W^a]),[\psi,V^a] \rangle
= \langle F_{\text{H}}(\phi),\psi \rangle
+ \langle F_{\text{M}}(W^a),V^a \rangle.
\end{equation}
The individual Hamiltonian form is shown in~\cite{HoBe01} to be:
\begin{equation}
   \label{eqn:ham_weak}
\langle F_{\text{H}}(\phi),\psi \rangle =
  \int_{\calg{M}} \hat{D}_a \phi \hat{D}^a \psi ~dx
+ \int_{\calg{M}} P'(\phi) \psi ~dx
\end{equation}
$$
+ \int_{\partial_1 \calg{M}} (c \phi - z) \psi ~ds,
$$
where
\begin{equation}
  \label{eqn:dp_def}
P'(\phi) = \frac{1}{8} \hat{R} \phi
         + \frac{1}{12} ({\rm tr} K)^2 \phi^5
         - \frac{1}{8} (\Ahatstar_{ab} + (\hat{L}W)_{ab})^2 \phi^{-7}
         - 2 \pi \hat{\rho} \phi^{-3},
\end{equation}
and the momentum form is shown in~\cite{HoBe01} to be:
\begin{equation}
   \label{eqn:mom_weak}
\langle F_{\text{M}}(W^a),V^a \rangle =
  \int_{\calg{M}}
  \left( 2 \mu (\hat{E}W)^{ab} (\hat{E}V)_{ab}
  + \lambda \hat{D}_a W^a \hat{D}_b V^b \right)
 ~dx
\end{equation}
$$
+ \int_{\calg{M}} 
    \left(
      \frac{2}{3} \phi^6 \hat{D}^a {\rm tr}K
    + 8 \pi \hat{j}^a
    \right) V_a ~dx
  \int_{\partial_1 \calg{M}} (C^a_{~b} W^b - Z^a) V_a ~ds,
$$
where $\mu=1$, $\lambda=-2/3$,
and where the {\em deformation tensor} $(\hat{E}V)^{ab}$ is the
symmetrized gradient:
\begin{equation}
   \label{eqn:deformation_tensor}
(\hat{E}V)^{ab} = \frac{1}{2} \left( \hat{D}^b V^a + \hat{D}^a V^b \right).
\end{equation}

The Gateaux-derivative of the nonlinear weak form
$\langle F(\cdot),\cdot \rangle$ in equation~(\ref{eqn:weakHamMom}) above
is needed for use in Newton-like iterative solution methods such as
Algorithm~\ref{alg:newton}.
Defining an arbitrary variation direction $w=[\xi,X^a]$, it is shown
in~\cite{HoBe01} that the Gateaux-derivative takes the following form
(for fixed $[\phi,W^a]$), linear separately in each of the variables
$[\xi,X^a]$ and $[\psi,V^a]$:
\begin{equation}
   \label{eqn:linearized_constraints}
\langle DF([\phi,W^a])[\xi,X^a],[\psi,V^a] \rangle
   = 
\end{equation}
$$
  \int_{\partial_1 \calg{M}}
           \left( c \xi \psi + C^a_{~b} X^b V_a
           \right) ~ds
$$
$$
   + \int_{\calg{M}}
       \left( \hat{D}_a \xi \hat{D}^a \psi
        + 2 \mu (\hat{E}X)^{ab} (\hat{E}V)_{ab}
       + \lambda \hat{D}_a X^a \hat{D}_b V^b
           \right)~dx
$$
$$
   + \int_{\calg{M}}
             \left(
             \frac{1}{8} \hat{R}
           + \frac{5}{12} ({\rm tr} K)^2 \phi^4
           + \frac{7}{8} (\Ahatstar_{ab} + (\hat{L}W)_{ab})^2 \phi^{-8}
           + 6 \pi \hat{\rho} \phi^{-4}
          \right) \xi \psi ~dx
$$
$$
   - \int_{\calg{M}}
     \left(
     \frac{1}{4} (\Ahatstar_{ab} + (\hat{L}W)_{ab}) \phi^{-7} 
     \right) 
       (\hat{L}X)^{ab} \psi ~dx
     + \int_{\calg{M}}
       \left(
           4 \phi^5 \hat{D}^a {\rm tr}K
         \right) V_a \xi ~dx.
$$
Now that the nonlinear weak 
form $\langle F(\cdot),\cdot \rangle$ and the associated bilinear
linearization form $\langle DF(\cdot)\cdot,\cdot \rangle$
are defined and can be evaluated
using numerical quadrature, the assembly of the nonlinear residual as well as
linearizations about any point can be performed precisely as outlined
above for a generic nonlinear finite element method.
Again, once we have the weak formulation and a linearization, the
discretization in \MC\ is automatic and generic.
Although the forms $\langle F(\cdot),\cdot \rangle$
and $\langle DF(\cdot)\cdot,\cdot \rangle$ above
appear somewhat complicated in the case of the constraint equations, the
discretization in \MC\ involves simply evaluating the integrands for use in
quadrature formulae.

We now state two new existence and uniqueness results for the
Hamiltonian and momentum constraints which were established recently
in~\cite{HoBe01}.
A number of assumptions on the problem data are required.
\begin{assumption}
   \label{ass:mom_data}
We assume that $\calg{M}$ is a connected compact Riemannian $3$-manifold with
Lipschitz-continuous boundary $\partial \calg{M}$.
We also assume that the data has the following properties:
$$
\hat{\gamma}_{ab} \in W^{1,\infty}(\calg{M}),
~~
K^{ab} \in W^{1,6/5}(\calg{M}),
~~
\phi \in L^{\infty}(\calg{M}),
~~
\overline{W}^a \in H^1(\calg{M}),
$$
$$
\hat{j}^a \in H^{-1}(\calg{M}),
~~
C^a_{~b} \in L^2(\partial_1 \calg{M}),
~~
Z^a \in L^{4/3}(\partial_1 \calg{M}),
~~
F^a \in H^{1/2}(\partial_0 \calg{M}),
$$
where $\overline{W}^a|_{\partial_0 \calg{M}} = F^a$ in the trace sense,
and where for some constant $\sigma > 0$,
$$
\int_{\partial_1 \calg{M}}
   C^a_{~b} V^b V_a ~dx \ge \sigma \| V^a \|_{L^2(\partial_1 \calg{M})}^2,
       ~~\forall~ V^a \in L^4(\partial_1 \calg{M}).
$$
\end{assumption}
\begin{assumption}
   \label{ass:ham_data}
We assume that $\calg{M}$ is a connected compact Riemannian $3$-manifold with
Lipschitz-continuous boundary $\partial \calg{M}$,
where $\text{meas}(\partial_0 \calg{M}) > 0$.
We also assume that the data has the following properties:
$$
\hat{\gamma}_{ab} \in W^{1,\infty}(\calg{M}),
~~
K^{ab} \in L^{\infty}(\calg{M}),
~~
W^a \in W^{1,\infty}(\calg{M}),
~~
\overline{\phi} \in H^1(\calg{M}) \cap L^{\infty}(\calg{M}),
$$
$$
\hat{R}, \Ahatstar_{ab}, \rho \in L^{\infty}(\calg{M}),
~~
c,z \in L^{\infty}(\partial_1 \calg{M}),
~~
f \in H^{1/2}(\partial_0 \calg{M}) \cap L^{\infty}(\partial_0 \calg{M}),
$$
where $\bar{\phi}|_{\partial_0 \calg{M}} = f$ in the trace sense, and
$$
0 < \inf_{\partial_0 \calg{M}} f
  \le f
  \le \sup_{\partial_0 \calg{M}} f
  < \infty,
~~~~a.e.~\text{in}~ \partial_0 \calg{M},
$$
$$
\rho \ge 0,
~~a.e.~\text{in}~ \calg{M},
~~c \ge 0,
~~ z = 0,
 ~~a.e.~\text{on}~ \partial_1 \calg{M}. 
$$
\end{assumption}
\begin{theorem}
   \label{thm:mom_continuous}
Let Assumption~\ref{ass:mom_data} hold.
Then there exists a unique solution
$W^a \in \overline{W}^a + H^1_{0,D}(\calg{M})$
to the momentum constraint equation
(\ref{eqn:mom_str1})--(\ref{eqn:mom_str3})
which depends continuously on the problem data.
Moreover, $U^a = W^a - \overline{W}^a \in H^1_{0,D}(\calg{M})$
satisfies the following {\em a priori} bound:
\begin{equation}
    \label{eqn:apriori_bound_coercive}
\| U^a \|_{H^1(\calg{M})}
   \le  \| U^a \|_{L^2(\calg{M})}
       + \frac{L}{\alpha},
\end{equation}
where $\alpha$ is the strong ellipticity constant
and $L$ is a bound on the linear functional arising in the weak
form.
If $\text{meas}(\partial_0 \calg{M}) > 0$, then the following
bound also holds:
\begin{equation}
    \label{eqn:apriori_bound_elliptic}
\| U^a \|_{H^1(\calg{M})}
   \le \frac{L}{m},
\end{equation}
where $m$ is the coercivity constant.
\end{theorem}
\begin{proof}
The proof given in~\cite{HoBe01} is based on the use
of a Riesz-Schauder alternative argument (uniqueness implies existence),
which is accessible after establishing that the momentum weak form
operator has a number of properties, including strong ellipticity and
satisfaction of a G{\aa}rding inequality.\qed
\end{proof}
\begin{theorem}
   \label{thm:ham_continuous}
Let Assumption~\ref{ass:ham_data} hold.
Then there exists a unique solution
$\phi \in \bar{\phi} + H^1_{0,D}(\calg{M})$
to the Hamiltonian constraint equation
(\ref{eqn:ham_str1})--(\ref{eqn:ham_str3}).
The solution $\phi$ satisfies {\em a priori} $L^\infty$-bounds
and is strictly positive {\em a.e.} in $\calg{M}$.
\end{theorem}
\begin{proof}
The proof given in~\cite{HoBe01} is based on variational analysis
and fixed-point arguments, after using a weak maximum principle to remove
the poles at the origin in the nonlinearity.
\qed
\end{proof}

The two results above indicate that the momentum and Hamiltonian
constraints on connected compact Riemannian manifolds with
Lipschitz boundaries
have well-posed weak formulations in the unweighted Sobolev
spaces $H^1_{0,D}(\calg{M})$.
A small amount of additional regularity,
namely the intersection of
Assumptions~\ref{ass:mom_data} and~\ref{ass:ham_data},
is required to give simultaneously well-posed weak formulations.
Under smoothness assumptions on the boundary and coefficients,
this minimal additional regularity can be shown for both $\phi$ and $W^a$
using elliptic regularity arguments (cf.~\cite{MaHu94} for a discussion
of the linear elasticity case which can be adapted here for the momentum
constraint).
Unfortunately, elliptic systems such as the momentum constraint do not
satisfy maximum principles analogous to the weak maximum principle
derived for the (scalar) Hamiltonian constraint in~\cite{HoBe01},
and as a result it is more
difficult to establish $L^{\infty}$-bounds on $W^a$.
Note that simultaneous well-posedness of the Hamiltonian and momentum
constraints individually does not imply well-posedness of the coupled system.
Limited results for the coupled system exist for some simplified situations;
cf.~\cite{IsMo96,Isen95,ChYo80,CIM92,MuYo73}.
Some new results for the coupled system,
based on Theorems~\ref{thm:mom_continuous} and~\ref{thm:ham_continuous},
appear in~\cite{HoBe01}.

\subsection[Quasi-optimal {\em a priori} Galerkin error estimates]
           {Quasi-optimal {\em a priori} error estimates
            for Galerkin approximations}
  \label{sec:approx}

In this section we consider the theory for Galerkin approximations of
the Hamiltonian and momentum constraints.
Following the approaches in~\cite{Brez85,Scha74,JeKe91} for related problems,
we establish two abstract results, the first of which applies
to linear variational problems satisfying a G{\aa}rding inequality,
whereas the second result applies to monotonically
nonlinear variational problems.
When applied to the Hamiltonian and momentum constraints, each result will
take the form:
\begin{equation}
   \label{eqn:holst_approx}
\|u-u_h\|_{H^1(\calg{M})} \le C \inf_{v \in V_h} \|u-v\|_{H^1(\calg{M})},
\end{equation}
where $u_h$ is a Galerkin approximation such as provided by a finite element 
discretization, and where $V_h \subset H^1_{0,D}(\calg{M})$ is the subspace
of continuous piecewise polynomials defined over simplices.
These results are {\em quasi-optimal} in the sense that they imply that
a Galerkin solution of either the Hamiltonian or momentum constraint
is within a constant of being the best approximation in the particular
subspace in which the Galerkin solution lives.
After giving the two abstract results along with their simple short proofs,
we indicate how they can be applied to the momentum and Hamiltonian
constraints in the context of Galerkin finite element methods.

While the term on the left in~(\ref{eqn:holst_approx})
is in general difficult to analyze,
the term on the right represents the fundamental question addressed by
classical approximation theory in normed spaces, of which much is known.
To bound the term on the right from above,
one picks a function in $V_h$ which is
particularly easy to work with, namely a nodal or generalized interpolant
of $u$, and then one employs standard techniques
in interpolation theory.
Therefore, it is clear that the importance of approximation results
such as~(\ref{eqn:holst_approx}) are that they completely separate
the details of the momentum and Hamiltonian constraints from the
approximation theory, making available all known results
on finite element interpolation of functions in Sobolev spaces
(cf.~\cite{Ciar78}).
There are some additional difficulties in using the standard finite element
interpolation theory associated with the fact that we are working with a
domain with the structure of a Riemannian 3-manifold
rather than an open set in $\bbbb{R}^d$; these are being addressed in
work in progress~\cite{Hols97c}, and will not be discussed in detail here.

\subsubsection{Approximation theory for the momentum constraint}
   \label{sec:mom_approx}

We now give a quasi-optimal {\em a priori} error estimate which characterizes
the quality of a Galerkin approximation to the solution of the
momentum constraint.
Quasi-optimal estimates are quite standard in the finite element approximation
theory literature for V-elliptic bilinear forms, but unfortunately it is
shown in~\cite{HoBe01} that the momentum constraint weak form is
only V-coercive (satisfying a G{\aa}rding inequality).
However, following Schatz~\cite{Scha74} we show this is sufficient to
establish similar quasi-optimal results for the momentum constraint
(cf.~\cite{ScWa00,XuZh97} for related results).

In order to derive such a result following the approach in~\cite{Scha74},
we begin with a Gelfand triple of Hilbert spaces
$V \subset H\equiv H^* \subset V^*$ with continuous embedding, meaning
that the pivot space $H$ and its dual space $H^*$ are identified through
the Riesz representation theorem,
and that the embedding $V \subset H$ is continuous.
A consequence of this is:
\begin{equation}
  \label{eqn:mom_as1}
\|u\|_H \le C \|u\|_V, ~~\forall u \in V,
\end{equation}
where we will assume that the embedding constant $C=1$
(the norm $\|\cdot\|_V$ can be redefined as necessary).
In our setting of the momentum constraint,
we have $H=L^2(\calg{M})$ and $V=H^1_{0,D}(\calg{M})$
generating the triple; we will stay with the abstract notation
involving $H$ and $V$ for clarity.
We are given the following variational problem:
\begin{equation}
  \label{eqn:mom_gal_cont}
\text{Find}~u \in V ~\text{s.t.}~
      A(u,v) = F(v), ~~\forall v \in V,
\end{equation}
where the bilinear form
$A(u,v) : V \times V \mapsto \bbbb{R}$ is bounded
\begin{equation}
  \label{eqn:mom_as2}
A(u,v) \le M \| u \|_V \| v \|_V,
      ~~\forall u,v \in V,
\end{equation}
and V-coercive (satisfying a G{\aa}rding inequality):
\begin{equation}
  \label{eqn:mom_as3}
m \|u\|_V^2 \le K \|u\|_H^2 + A(u,u),
      ~~\forall u \in V, ~~~\text{where}~~ m > 0,
\end{equation}
and where the linear functional $F(v) : V \mapsto \bbbb{R}$ is bounded
and thus lies in the dual space $V^*$:
$$
F(v) \le L \| v \|_V,
      ~~\forall v \in V.
$$
It is shown in~\cite{HoBe01} that the weak formulation of the
momentum constraint~(\ref{eqn:mom_weak})
fits into this framework; to simplify the discussion, we have
assumed that any Dirichlet function $\bar{u}$ has been absorbed
into the linear functional $F(v)$ in the obvious way.
Our discussion can be easily modified to include approximation
of $\bar{u}$ by $\bar{u}_h$.

Now, we are interested in the quality of a Galerkin approximation:
\begin{equation}
  \label{eqn:mom_gal_disc}
\text{Find}~u_h \in V_h \subset V~\text{s.t.}~
      A(u_h,v) = A(u,v) = F(v), ~~\forall v \in V_h \subset V.
\end{equation}
We will assume that there exists a sequence of approximation subspaces
$V_h \subset V$ parameterized by $h$, with $V_{h_1} \subset V_{h_2}$
when $h_2 < h_1$, and that there exists a sequence $\{a_h\}$, with
$\lim_{h\rightarrow 0} a_h = 0$, such that
\begin{equation}
  \label{eqn:mom_as4}
\|u-u_h\|_H \le a_h \|u-u_h\|_V, 
   ~\text{when}~ A(u-u_h,v) = 0, ~\forall v \in V_h \subset V.
\end{equation}
The assumption~(\ref{eqn:mom_as4}) is very natural;
in our setting, it is the assumption
that the error in the approximation converges to zero more quickly in the
$L^2$-norm than in the $H^1$-norm.
This is easily verified in the setting of piecewise polynomial approximation
spaces, under very mild smoothness requirements on the solution~$u$;
cf. Lemmas~2.1 and~2.2 in~\cite{XuZh97}.
Under these assumptions, we have the following {\em a priori} error estimate.
Although the assumptions are slightly different, the result and
the main idea for the simple proof we give below (included for completeness)
go back to Schatz~\cite{Scha74} (see also~\cite{ScWa00,XuZh97}).

\begin{theorem}
   \label{thm:mom_approx}
Let $V \subset H \subset V^*$ be a Gelfand triple of Hilbert spaces
with continuous embedding.
Assume that (\ref{eqn:mom_gal_cont}) is uniquely solvable, and that assumptions
(\ref{eqn:mom_as1}), (\ref{eqn:mom_as2}), (\ref{eqn:mom_as3}),
and (\ref{eqn:mom_as4}) hold.
Then for $h$ sufficiently small, there exists a unique approximation
$u_h$ satisfying~(\ref{eqn:mom_gal_disc}),
for which the following quasi-optimal
{\em a priori} error bounds hold:
\begin{eqnarray}
\|u-u_h\|_V 
   & \le & C \inf_{v \in V_h} \|u-v\|_V,
\label{eqn:schatz} \\
\|u-u_h\|_H 
   & \le & C a_h \inf_{v \in V_h} \|u-v\|_V,
\label{eqn:schatz_L2}
\end{eqnarray}
where $C$ is a constant independent of $h$.
If $K \le 0$ in~(\ref{eqn:mom_as3}), then the above holds for all $h$.
\end{theorem}
\begin{proof}
The following proof follows the idea in~\cite{Scha74}.
We begin with the G{\aa}rding inequality~(\ref{eqn:mom_as3})
and then employ~(\ref{eqn:mom_as2}):
\begin{eqnarray}
m \|u-u_h\|_V^2 - K \|u-u_h\|_H^2
&\le& A(u-u_h,u-u_h)
\nonumber \\
&=& A(u-u_h,u-v)
\nonumber \\
&\le& M \|u-u_h\|_V \|u-v\|_V,
   \label{eqn:schatz_tmp}
\end{eqnarray}
where we have used Galerkin orthogonality: $A(u-u_h,v)=0,
     ~\forall v \in V_h$,
to replace $u_h$ with an arbitrary $v \in V_h$.
Excluding first the case that $\|u-u_h\|_V=0$ we divide through by
$m \|u-u_h\|_V$ and employ~(\ref{eqn:mom_as1}) and~(\ref{eqn:mom_as4}),
giving $\forall v \in V_h$,
\begin{eqnarray}
\left( 1 - \frac{Ka_h}{m} \right) \|u-u_h\|_V
&\le& \|u-u_h\|_V - \frac{K\|u-u_h\|_H^2}{m\|u-u_h\|_V}
\nonumber \\
&\le& \frac{M}{m} \|u - v\|_V, 
  \label{eqn:schatz_core}
\end{eqnarray}
which we note also holds when $\|u-u_h\|=0$.

Assume first that $K > 0$.
Since $\lim_{h \rightarrow 0} a_h = 0$, there exists $\overline{h}$ such that
$a_h < m/K, ~\forall h \le \overline{h}$.
This implies $\forall v \in V_h$,
\begin{equation}
 \label{eqn:schatz_tmp3}
    \left( 1 - \frac{Ka_{\overline{h}}}{m} \right) \|u-u_h\|_V
\le \left( 1 - \frac{Ka_h}{m} \right) \|u-u_h\|_V
    \le \frac{M}{m} \|u - v\|_V.
\end{equation}
Taking $u=0$ in~(\ref{eqn:mom_gal_disc}) together
with $v=0$ in~(\ref{eqn:schatz_tmp3}), with $h  \le \overline{h}$,
implies that the homogeneous problem
$$
\text{Find}~ u_h \in V_h ~\text{s.t.}~
     A(u_h,v) = 0, \ \ \forall~ v \in V_h,
$$
has only the trivial solution, so that by the discrete Fredholm alternative
a solution $u_h$ to~(\ref{eqn:mom_gal_disc}) is unique and
therefore exists.
Equation~(\ref{eqn:schatz_tmp3}) then finally gives~(\ref{eqn:schatz})
whenever $h \le \overline{h}$, with the choice
$$
C = \frac{M}{m \left( 1 - \frac{Ka_{\overline{h}}}{m} \right)}
  = \frac{M}{m - Ka_{\overline{h}}}.
$$

Assume now that $K \le 0$.
Directly from~(\ref{eqn:schatz_core}) we can conclude~(\ref{eqn:schatz})
with $C = M/m$, which is completely independent of $n$;
this then becomes Cea's Lemma for V-elliptic forms~\cite{Ciar78}.
Moreover, the continuous and discrete problems are both uniquely solvable
due to V-ellipticity~(\ref{eqn:mom_as3}), independent of $h$.

In either case of $K > 0$ or $K \le 0$,
the second estimate (\ref{eqn:schatz_L2})
now follows immediately from assumption~(\ref{eqn:mom_as4}).\qed
\end{proof}

In the case of the momentum constraint it was established in~\cite{HoBe01}
that the assumptions required for Theorem~\ref{thm:mom_approx} hold,
with the exception of~(\ref{eqn:mom_as4}).
In the case of Robin boundary conditions, it was shown in~\cite{HoBe01}
that $1 \le \alpha = K \le 4/3$.
This gives
$$
C = \frac{M}{\alpha(1 - a_{\overline{h}})} \le \frac{M}{1 - a_{\overline{h}}}.
$$
Under the mild assumption that the {\em a priori}
bound~(\ref{eqn:apriori_bound_coercive}) or~(\ref{eqn:apriori_bound_elliptic})
can be shown to hold in a slightly stronger Sobolev norm, referred to as
an elliptic regularity estimate:
$$
\| W^a \|_{H^{1+s}(\calg{M})}
   \le \frac{L}{m},  ~~~~ s > 0,
$$
then it can be shown that~(\ref{eqn:mom_as4}) holds in the setting
of piecewise linear finite element spaces, with
$a_h = n^{-\gamma}$ for some $\gamma > 0$, where
$n = \dim( V_h )$.
This makes it clear that the requirement that $h$ be sufficiently small
is not a practical restriction on applying the finite element method to
the momentum constraint.

\subsubsection{Approximation theory for the Hamiltonian constraint}
   \label{sec:ham_approx}

We consider now the nonlinear Hamiltonian constraint, and derive a
quasi-optimal {\em a priori} error estimate for Galerkin approximations
analogous to that derived in the previous section for the momentum constraint.
The approximation theory for Galerkin approximations
to the nonlinear Hamiltonian constraint~(\ref{eqn:ham_weak}) is
somewhat more complex than for the momentum constraint~(\ref{eqn:mom_weak}).
However, it is still possible to establish a result for the Hamiltonian
constraint which shows that a Galerkin approximation is quasi-optimal
under some weak assumptions on the nonlinearity.
A number of such estimates have appeared in the literature; the result we
derive below is similar to estimates in~\cite{Ciar78,Brez85,JeKe91}.

We begin again with a Gelfand triple of Hilbert
spaces $V \subset H\equiv H^* \subset V^*$ with continuous embedding, so that
again~(\ref{eqn:mom_as1}) holds.
We are given the following nonlinear variational problem:
\begin{equation}
  \label{eqn:ham_gal_cont}
\text{Find}~u \in V ~\text{s.t.}~
      A(u,v) + \langle B(u),v \rangle = F(v), ~~\forall v \in V,
\end{equation}
where the bilinear form
$A(u,v) : V \times V \mapsto \bbbb{R}$ is bounded
\begin{equation}
  \label{eqn:ham_as2}
A(u,v) \le M \| u \|_V \| v \|_V,
      ~~\forall u,v \in V,
\end{equation}
and V-elliptic:
\begin{equation}
  \label{eqn:ham_as3}
m \|u\|_V^2 \le A(u,u),
      ~~\forall u \in V, ~~~\text{where}~~ m > 0,
\end{equation}
where the linear functional $F(v) : V \mapsto \bbbb{R}$ is bounded
and thus lies in the dual space $V^*$:
$$
F(v) \le L \| v \|_V,
      ~~\forall v \in V,
$$
and where the nonlinear form
$\langle B(u),v \rangle : V \times V \mapsto \bbbb{R}$
is assumed to be monotonic:
\begin{equation}
  \label{eqn:ham_as3a}
0 \le \langle B(u)-B(v),u-v \rangle,
      ~~\forall u,v \in V,
\end{equation}
where we have used the notation:
\begin{equation}
\label{eqn:weakFormDiff}
  \langle B(u)-B(v),w \rangle = \langle B(u),w \rangle
                              - \langle B(v),w \rangle.
\end{equation}
We are interested in the quality of a Galerkin approximation:
\begin{equation}
  \label{eqn:ham_gal_disc}
\text{Find}~u_h \in V_h ~\text{s.t.}~
      A(u_h,v) + \langle B(u_h),v \rangle = F(v),
      ~~\forall v \in V_h,
\end{equation}
where $V_h \subset V$.
We will assume that $\langle B(u),v \rangle$
is bounded in the following weak sense:
If $u \in V$ satisfies~(\ref{eqn:ham_gal_cont}),
if $u_h \in V_h$ satisfies~(\ref{eqn:ham_gal_disc}),
and if $v \in V_h$, then there exists a constant $K>0$ such that:
\begin{equation}
  \label{eqn:ham_as3b}
\langle B(u)-B(u_h),u-v \rangle \le K \| u-u_h \|_V \| u-v \|_V.
\end{equation}
It is shown in~\cite{HoBe01} that the weak formulation of the
Hamiltonian constraint~(\ref{eqn:ham_weak})
fits precisely into this framework with the possible exception
of~(\ref{eqn:ham_as3b});
we will show below that {\em a priori} bounds such as those established
in~\cite{HoBe01} can be used to establish~(\ref{eqn:ham_as3b}).
We have again assumed that any Dirichlet function $\bar{u}$ has been
absorbed into the various forms in the obvious way to simplify the discussion.
The discussion can be modified to include approximation
of $\bar{u}$ by $\bar{u}_h$.

Again, we are interested in the quality of a
Galerkin approximation $u_h$ satisfying~(\ref{eqn:ham_gal_disc}),
or equivalently:
$$
    A(u-u_h,v) + \langle B(u)-B(u_h),v \rangle = 0,
      ~\forall v \in V_h \subset V.
$$
As before,
we will assume that there exists a sequence of approximation subspaces
$V_h \subset V$ parameterized by $h$, with $V_{h_1} \subset V_{h_2}$
when $h_2 < h_1$, and that there exists a sequence $\{a_h\}$, with
$\lim_{h\rightarrow 0} a_h = 0$, such that
\begin{equation}
  \label{eqn:ham_as4}
\|u-u_h\|_H \le a_h \|u-u_h\|_V, 
\end{equation}
holds whenever $u_h$ satisfies~(\ref{eqn:ham_gal_disc}).
The assumption~(\ref{eqn:ham_as4}) is again very natural;
see the discussion above following~(\ref{eqn:mom_as4}).
Under these assumptions, we have the following {\em a priori} error estimate.
\begin{theorem}
   \label{thm:ham_approx}
Let $V \subset H \subset V^*$ be a Gelfand triple of Hilbert spaces
with continuous embedding.
Assume that (\ref{eqn:ham_gal_cont}) and (\ref{eqn:ham_gal_disc}) are
uniquely solvable, and that assumptions
(\ref{eqn:mom_as1}), (\ref{eqn:ham_as2}), (\ref{eqn:ham_as3}),
(\ref{eqn:ham_as3b}),
and (\ref{eqn:ham_as4}) hold.
Then the approximation $u_h$ satisfying~(\ref{eqn:ham_gal_disc})
obeys the following quasi-optimal {\em a priori} error bounds:
\begin{eqnarray}
\|u-u_h\|_V 
   & \le & C \inf_{v \in V_h} \|u-v\|_V,
\label{eqn:holst} \\
\|u-u_h\|_H 
   & \le & C a_h \inf_{v \in V_h} \|u-v\|_V,
\label{eqn:holst_L2}
\end{eqnarray}
where $C$ is a constant independent of $h$.
\end{theorem}
\begin{proof}
We begin by subtracting~(\ref{eqn:ham_gal_disc}) from~(\ref{eqn:ham_gal_cont}),
and taking $v=w \in V_h \subset V$ in both equations, giving:
\begin{equation}
   \label{eqn:holst_tmp0}
  A(u-u_h,w) + \langle B(u)-B(u_h),w \rangle = 0,
~~~\forall w \in V_h.
\end{equation}
In particular if $v \in V_h$, so that $w=v-u_h \in V_h$, this implies that
\begin{eqnarray}
  A(u-u_h,v-u_h)
&=& \langle B(u_h)-B(u), v-u_h \rangle
\nonumber \\
&=& \langle B(u_h)-B(u), v-u \rangle
  - \langle B(u_h)-B(u), u_h-u \rangle
\nonumber \\
&\le& \langle B(u_h)-B(u),v-u \rangle,
   \label{eqn:holst_tmp1}
\end{eqnarray}
where we have employed monotonicity~(\ref{eqn:ham_as3a}).
Beginning now with~(\ref{eqn:ham_as3}) we have for arbitrary $v \in V_h$ that
\begin{eqnarray}
m \|u-u_h\|_V^2 
&\le& A(u-u_h,u-u_h)
\nonumber \\
&=& A(u-u_h,u-v) + A(u-u_h,v-u_h)
\nonumber \\
&\le& A(u-u_h,u-v) + \langle B(u_h)-B(u),v-u \rangle
\nonumber \\
&\le& M \|u-u_h\|_V \|u-v\|_V 
\label{eqn:holst_tmp2} \\
& & + K \|u-u_h\|_V \|u-v\|_V.
\nonumber
\end{eqnarray}
where we have used~(\ref{eqn:holst_tmp1}), (\ref{eqn:ham_as2}),
and~(\ref{eqn:ham_as3b}).
Excluding first the case that $\|u-u_h\|_V=0$ we divide through by
$m \|u-u_h\|_V$, giving
\begin{equation}
  \label{eqn:holst_core}
\|u-u_h\|_V \le \left( \frac{M+K}{m} \right)
    \|u - v\|_V, ~~~ \forall v \in V_h,
\end{equation}
which we note also holds when $\|u-u_h\|=0$.
This gives~(\ref{eqn:holst}) with $C=(M+K)/m$.
The second estimate (\ref{eqn:holst_L2})
now follows immediately from assumption~(\ref{eqn:ham_as4}).\qed
\end{proof}

In the case of the Hamiltonian constraint the nonlinear weak form
$\langle B(u),v \rangle$ has
the form
$$
\langle B(u),v \rangle = \int_{\calg{M}} P'(u) v ~dx,
$$
where $P'(u)$ is defined in~(\ref{eqn:dp_def}).
If both $u$ and $u_h$ satisfy {\em a priori} bounds as
established in~\cite{HoBe01},
then the continuity of $P''(x)$ on $(0,\infty)$ implies that
there exists $w \in L^{\infty}(\calg{M})$ satisfying similar
bounds such that
$$
   P'(u)-P'(u_h) = P''(w)(u - u_h),
~~~a.e.~~~\text{in}~ \calg{M}.
$$
Consider now
$$
\langle B(u)-B(u_h),u-v \rangle
~~~~~~~~~~~~~~~~~~~~~~~~~~~~~~~~~~~~~~~~~~~~~~~~~~~~~~~~~~~
$$
\vspace*{-0.8cm}
\begin{eqnarray*}
     &=& \int_{\calg{M}} (P'(u)-P'(u_h)) (u-v) ~dx
\nonumber \\
     &=& \int_{\calg{M}} P''(w) (u-u_h)(u-v) ~dx
\nonumber \\
&\le& \| P''(w) \|_{L^{\infty}(\alpha \le w \le \beta)}
       \| u-u_h \|_{L^2(\calg{M})}
       \| u-v   \|_{L^2(\calg{M})}
\nonumber \\
&\le& K
       \| u-u_h \|_{H^1(\calg{M})}
       \| u-v   \|_{H^1(\calg{M})}.
\end{eqnarray*}
Therefore, (\ref{eqn:ham_as3b}) holds with
$K=\| P''(w) \|_{L^{\infty}(\alpha \le w \le \beta)}$,
which can be computed explicitly from the results in~\cite{HoBe01}.
Although a Galerkin approximation $u_h$ constructed from finite element
bases will not in general satisfy a discrete maximum principle which would
lead to {\em a priori} bounds as in~\cite{HoBe01}, it
is possible to establish $L^{\infty}$-bounds for a Galerkin finite 
element solution to the Hamiltonian constraint under some assumptions
on the size and shape of the elements in the mesh
(cf. Theorem 3.2 in~\cite{KeJe90}).

Therefore, we see that
in the case of the Hamiltonian constraint we have established
that the assumptions required for Theorem~\ref{thm:ham_approx} hold,
with the exception of~(\ref{eqn:ham_as4}).
Under the mild additional regularity assumption:
$$
\| \phi \|_{H^{1+s}(\calg{M})}
   \le C < \infty,  ~~~~ s > 0,
$$
where $C$ depends on the data, then it can be shown
that~(\ref{eqn:ham_as4}) holds in the setting
of piecewise linear finite element spaces, with
$a_h = n^{-\gamma}$ for some $\gamma > 0$, where $n = \dim( V_h )$.

Other approaches also lead to well-posed weak formulations of the
Hamiltonian constraint with associated approximation theory.
In particular, an obstacle problem formulation is possible as a technique
for handling the pole at the origin in the Hamiltonian constraint, leading
to a nonlinear variational inequality.
This approach requires fewer assumptions on the data in the Hamiltonian
constraint than we have assumed here.
Although several difficulties arise in {\em a priori} error analysis,
a number of results for linear and nonlinear variational
inequalities are known, and could be applied in this case.
Results similar to Theorem~\ref{thm:ham_approx} are obtainable under
the same minimal assumption $\phi \in H^1(\calg{M})$ required to give a
well-posed weak formulation (cf.~\cite{Ciar78,Brez85,BHR77,Falk74}).
Introducing a cut-off function in place of the
two terms with poles in the Hamiltonian constraint leads to a well-posed
weak formulation, although the error analysis is not clear.
Approaches based on weighted Sobolev spaces also lead to well-posed
weak formulations, but incorporation of weights into the finite element
subspaces is technically complicated.

\subsection{Numerical solution using MC}

To use \MC\ to calculate the initial bending of space and time around two
massive black holes separated by a fixed distance by solving the above
constraint equations, we place two spherical objects in space, 
the first object having unit radius (after appropriate normalization),
the second object having radius 2, separated by a distance of 20.
Infinite space is truncated with an enclosing sphere of radius 100.
(This outer boundary may be moved further from the objects to improve the
accuracy of boundary condition approximations.)
Resonable choices for the remaining functions and parameters appearing
in the equations are used below to completely specify the problem for use
as an illustrative numerical example.
(More careful examination of the various functions and parameters 
appear in~\cite{HoBe01}, and a number of detailed
experiments with more physically meaningful data appear
in~\cite{HoBe99,BeHo99}.)

We then generate an initial (coarse) mesh of tetrahedra inside the enclosing
sphere, exterior to the two spherical objects within the enclosing sphere.
The mesh is generated by adaptively bisecting an initial mesh consisting
of an icosahedron volume filled with tetrahedra.
The bisection procedure simply bisects any tetrahedron which touches the
surface of one of the small spherical objects.
When a reasonable approximation to the surface of the spheres is obtained,
the tetrahedra completely inside the small spherical objects are removed,
and the points forming the surfaces of the small spherical objects are
projected to the spherical surfaces exactly.
This projection involves solving a linear elasticity problem,
together with the use of a shape-optimization-based smoothing procedure.
The smoothing procedure locally optimizes the shape measure function in
equation~(\ref{eqn:shape}) for a given $d$-simplex $s$, in an iterative
fashion.
A much improved binary black hole mesh generator has been developed by
D.~Bernstein; the new mesh generator is described in~\cite{HoBe99,BeHo99}
along with a number of more detailed examples using MC.

The initial coarse mesh in Figures~\ref{bh_1a}--\ref{bh_1c}, generated using
the procedure described above, has approximately 31,000 tetrahedral elements
and 6,000 vertices.
To solve the problem on a 4-processor computing cluster using
PPUM (see Section~\ref{sec:ppum}), we begin by partitioning
the domain into four subdomains
(shown in Figures~\ref{bh_speca}--\ref{bh_specb})
with approximately equal error
using the recursive spectral bisection algorithm described in~\cite{BaHo98a}.
The four subdomain problems are then solved independently by \MC, starting
from the complete coarse mesh and coarse mesh solution.
The mesh is adaptively refined in each subdomain until
a mesh with roughly 50000 vertices is obtained
(yielding subdomains with about 250000 simplices each).

The resulting refined subdomain meshes are shown in
Figures~\ref{bh_ref1}--\ref{bh_ref2}.
The refinement performed by \MC\ is confined
primarily to the given region as driven by the weighted residual
error indicator from Section~\ref{sec:indicators},
with some refinement into adjacent regions
due to the closure algorithm which maintains conformity and shape regularity. 
The four problems are solved completely independently by the
sequential adaptive software package \MC.
One component of the solution (the conformal factor $\phi$) of the
elliptic system is depicted in Figure~\ref{bh_sol1}
(the subdomain zero solution)
and in Figure~\ref{bh_sol2} (the subdomain two solution).

While this example illustrates some of the capabilities of \MC,
a number of more detailed examples involving the contraints,
using more physically meaningful data, appear in~\cite{HoBe99,BeHo99}.

\begin{figure}[htbp]
\centerline{
\mbox{\myfigpng{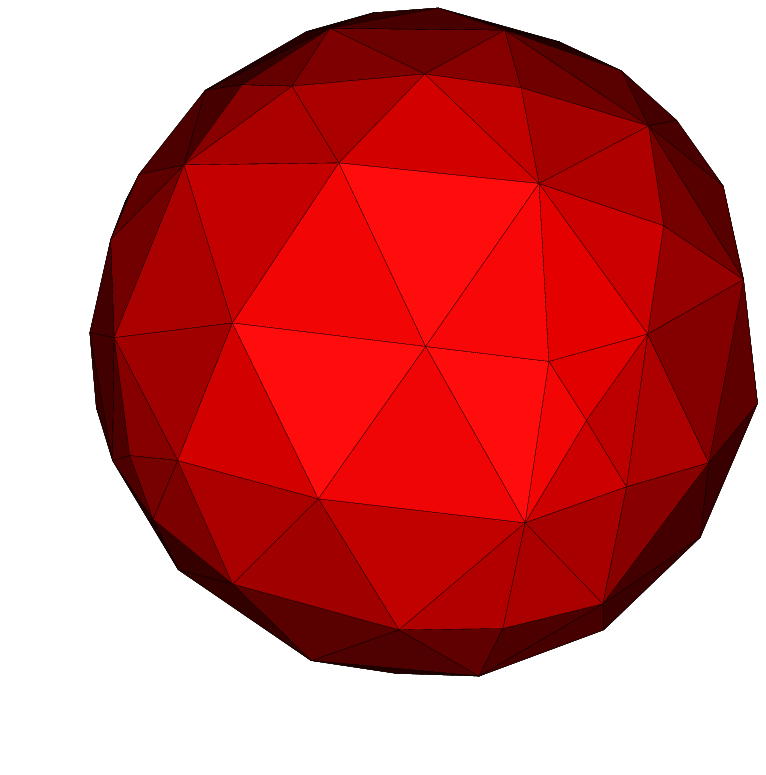}{2.6in}}
\mbox{\myfigpng{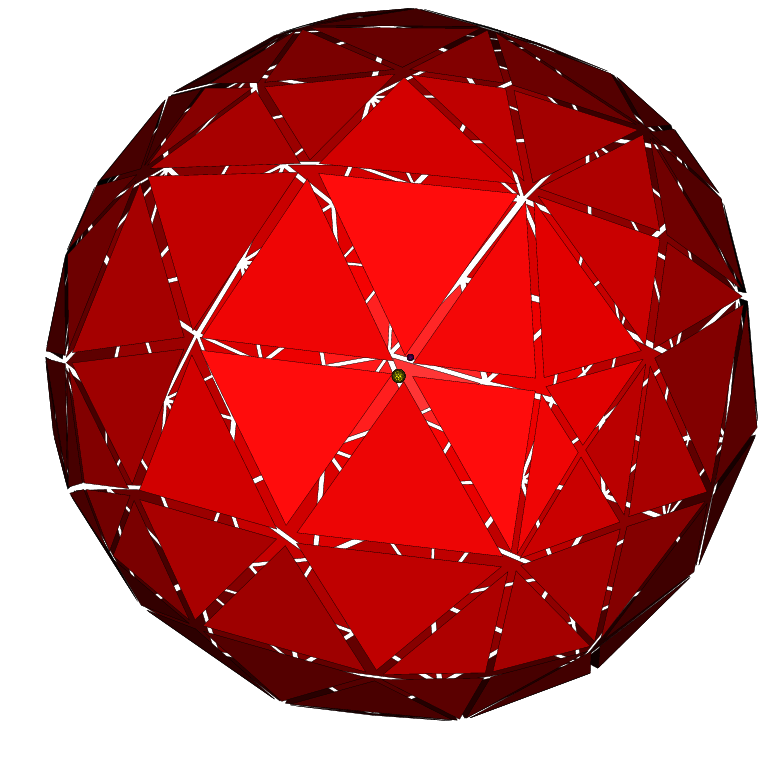}{2.6in}}
}
\caption{The coarse binary black hole mesh
         (approximately 6,000 vertices and 31,000 simplices).}
\label{bh_1a}
\end{figure}
\begin{figure}[htbp]
\centerline{\hbox{\myfigpng{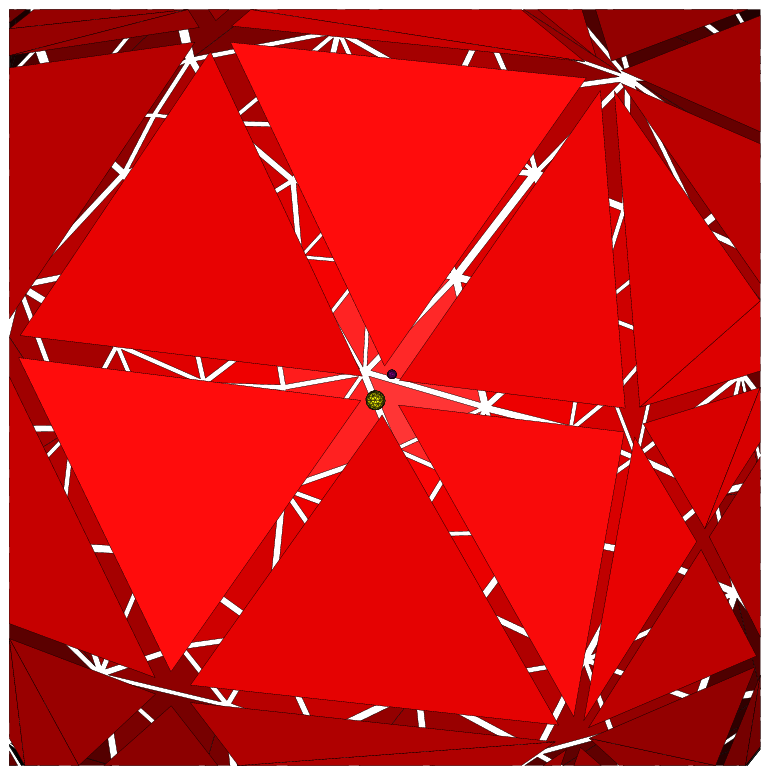}{3.9in}}}
\caption{Exploded view of the coarse binary black hole mesh
         showing the two interior hole boundaries.}
\label{bh_1b}
\end{figure}

\begin{figure}[htbp]
\centerline{
\mbox{\myfigpng{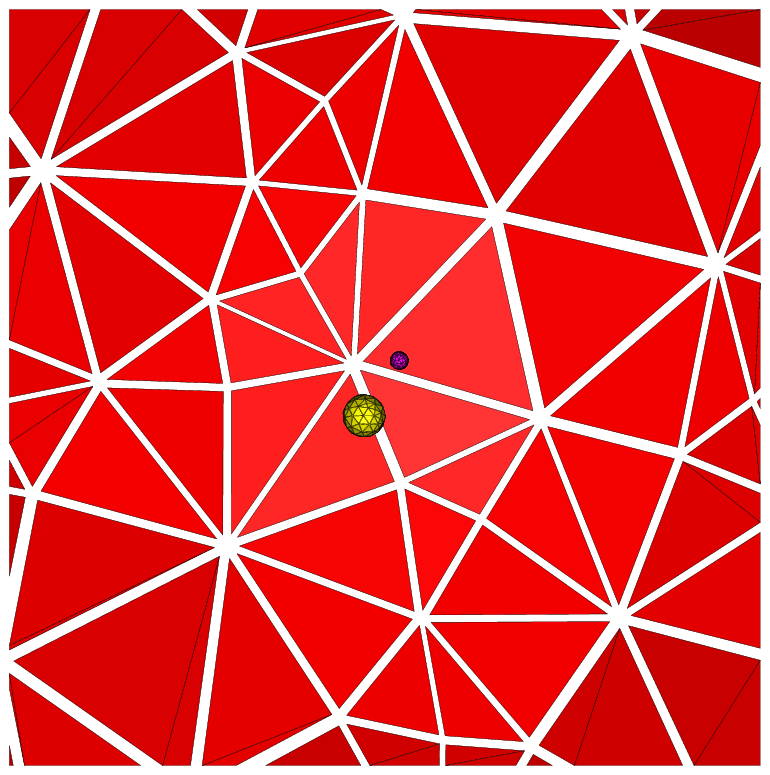}{2.4in}}
\mbox{\myfigpng{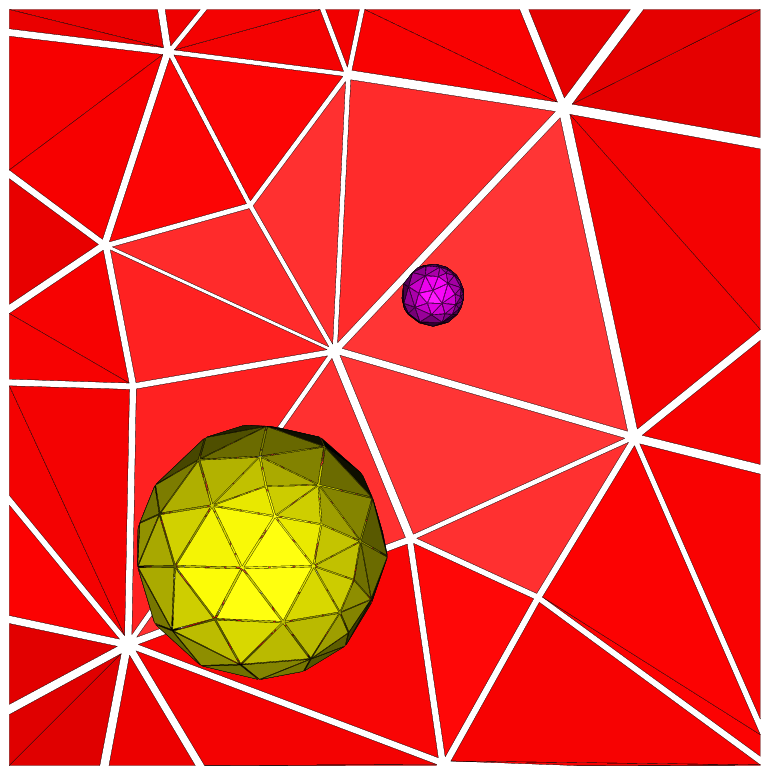}{2.4in}}
}
\caption{Closeup of the interior of the coarse binary black hole mesh.
         The interior holes surfaces of black hole coarse mesh;
         the larger hole surface is colored yellow,
         the smaller hole surface is colored purple,
         and the exterior boundary is colored red.}
\label{bh_1c}
\end{figure}

\begin{figure}[htbp]
\centerline{
\mbox{\myfigpng{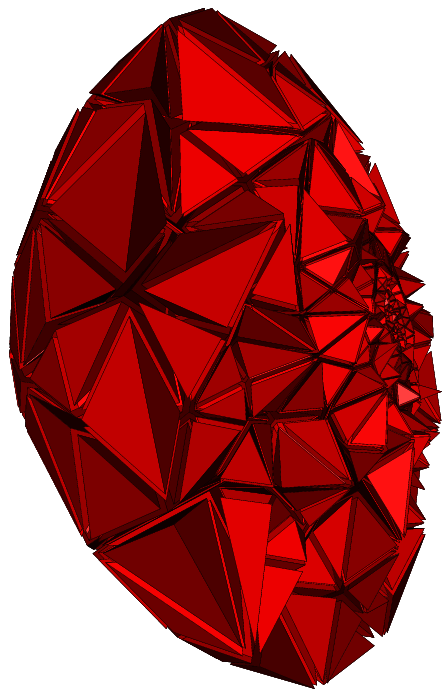}{1.9in}}
\mbox{\myfigpng{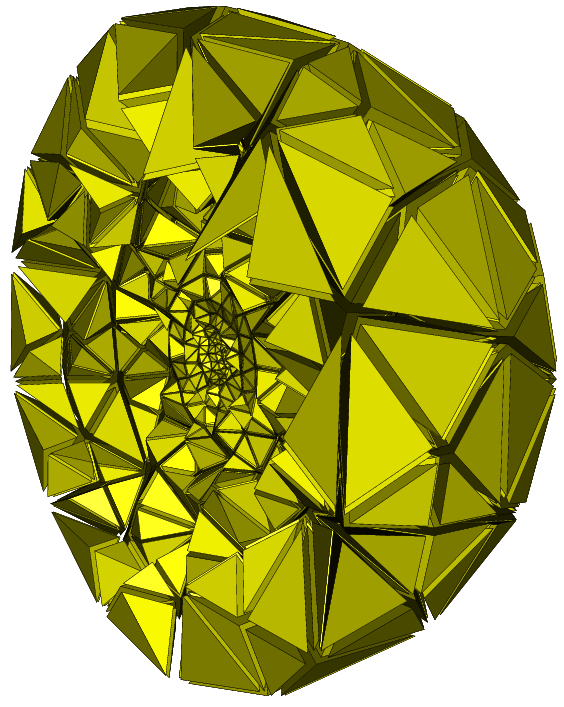}{1.9in}}
}
\caption{Subdomains 2 (red) and 4 (yellow)
         from spectral bisection of the coarse binary black hole mesh;
         these subdomains enclose two smaller subdomains that contain
         the inner holes.}
\label{bh_speca}
\end{figure}

\begin{figure}[htbp]
\centerline{
\mbox{\myfigpng{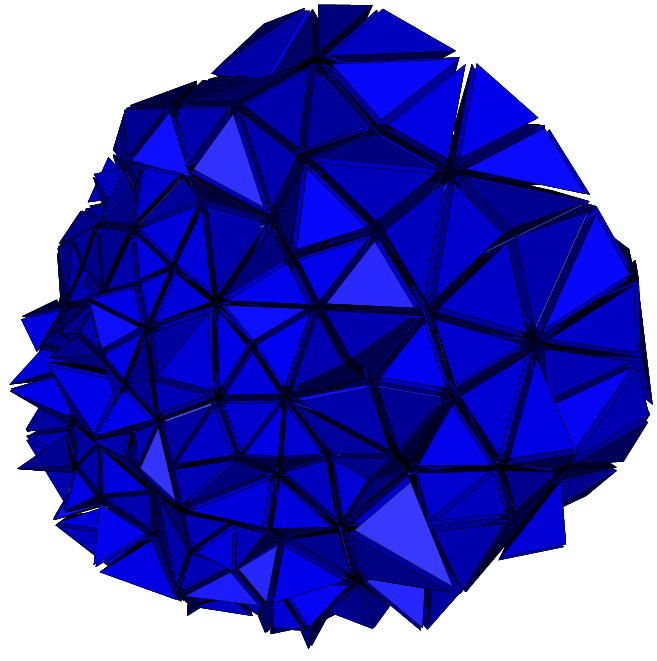}{1.8in}}
\mbox{\myfigpng{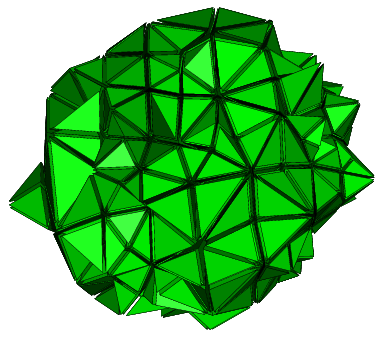}{1.8in}}
}
\caption{Subdomains 3 (blue) and 1 (green)
         from spectral bisection of the coarse binary black hole mesh;
         these subdomains each contain one of the inner holes.}
\label{bh_specb}
\end{figure}

\begin{figure}[htbp]
\centerline{
\mbox{\myfigpng{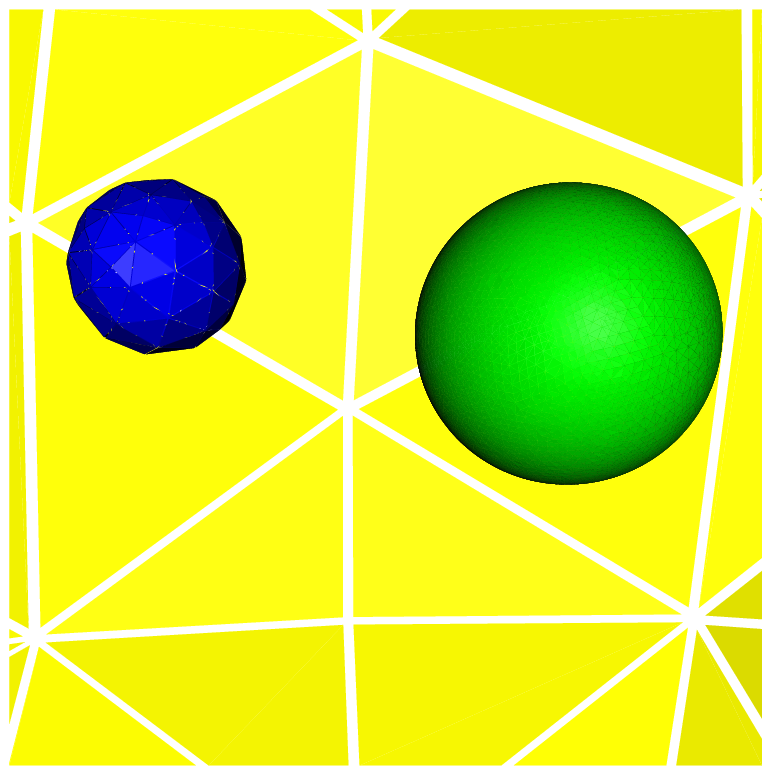}{3.0in}}
\mbox{\myfigpng{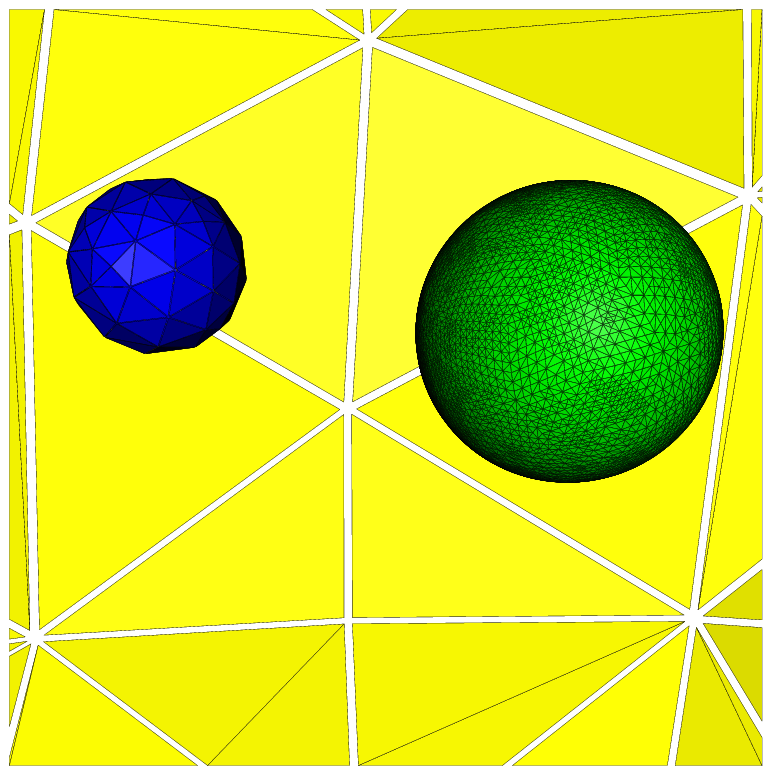}{3.0in}}
}
\caption{Closeup of the subdomain 1 refined mesh around the
         surface of the smaller hole.
         (Approximately 51,000 vertices and 266,000 simplices; only faces of
         tetrahedra on the boundary surfaces are shown).}
\label{bh_ref1}
\end{figure}

\begin{figure}[htbp]
\centerline{
\mbox{\myfigpng{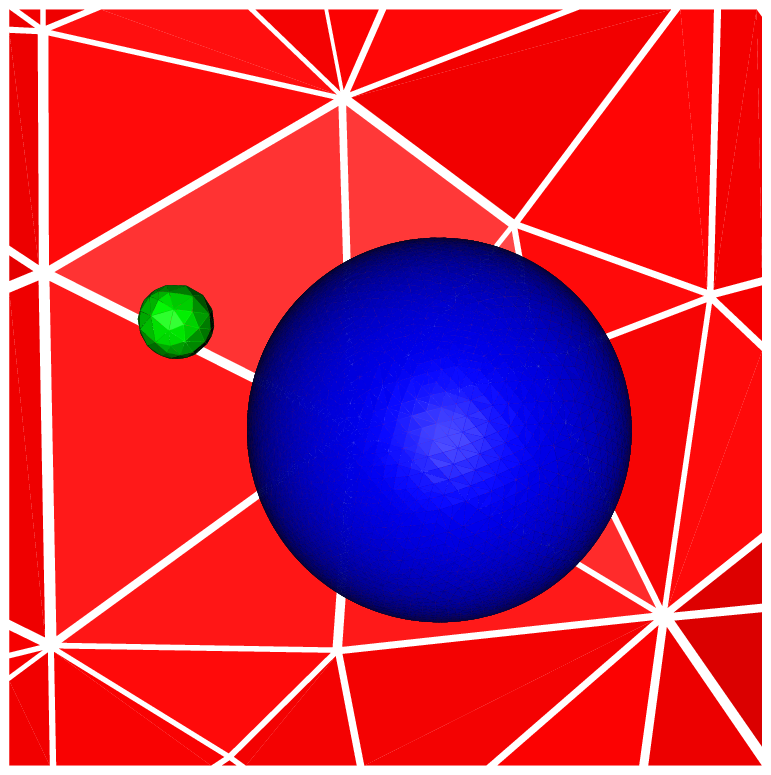}{3.0in}}
\mbox{\myfigpng{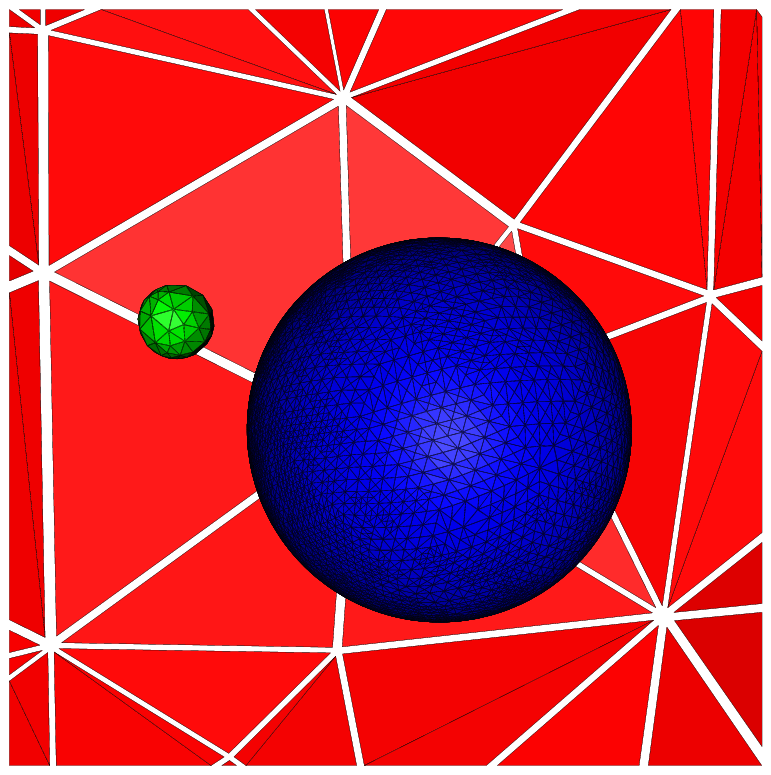}{3.0in}}
}
\caption{Closeup of the subdomain 3 refined mesh around the
         surface of the larger hole.
         (Approximately 45,000 vertices and 228,000 simplices; only faces of
         tetrahedra on the boundary surfaces are show).}
\label{bh_ref2}
\end{figure}

\begin{figure}[htbp]
\centerline{\hbox{\myfigpng{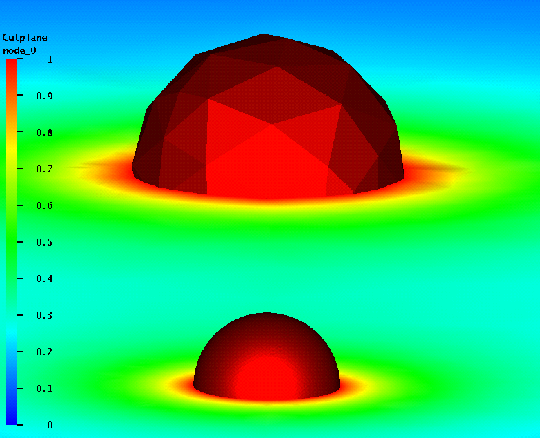}{3.5in}}}
\caption{The conformal factor $\phi$ from the adapted subdomain 1 solve.}
\label{bh_sol1}
\end{figure}

\begin{figure}[htbp]
\centerline{\hbox{\myfigpng{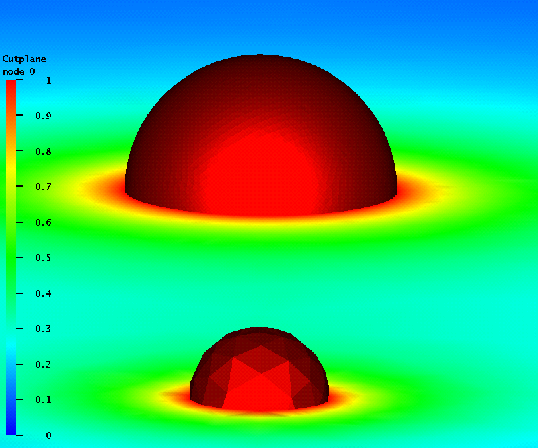}{3.5in}}}
\caption{The conformal factor $\phi$ from the adapted subdomain 3 solve.}
\label{bh_sol2}
\end{figure}

\section{Summary}
  \label{sec:summary}

In this paper we considered the design of adaptive multilevel finite
element methods for certain elliptic systems arising in geometric analysis
and general relativity.  
We began with a brief introduction to nonlinear elliptic tensor systems
on manifolds, and then discussed adaptive finite element methods for
this class of problems.
We derived two {\em a posteriori} error indicators,
one of which was local residual-based, and one of which was
based on a global linearized adjoint or {\em dual} problem.

The implementation of these methods and indicators in the ANSI C finite
element software package \MC\ was discussed, including detailed descriptions
of some of the more interesting algorithms and data structures it employs.
\MC\ was designed by the author specifically for solving general second-order
nonlinear elliptic systems of tensor equations on Riemannian manifolds
with boundary, including domains requiring multiple coordinate systems.
The key feature of \MC\ which makes it particularly useful for
highly complex tensor systems of PDEs arising in geometric analysis and
general relativity is its abstraction; in addition to the support for
multi-chart manifolds, the \MC\ user supplies only two ANSI C functions
representing the weak form of the tensor system
$\langle F(u),v \rangle$ along with its linearization
form $\langle DF(u)w,v \rangle$.
Moreover, the forms themselves may be implemented almost exactly
as they are written on paper, due to the fact that the quadrature-based
assembly allows for tensor expressions to be treated discretely
as point tensors rather than tensor fields.
If residual-based or duality-based {\em a posteriori} error estimation
is to be used, then the user must provide a third function $F(u)$, which is
essentially the strong form of the differential equation as needed for the
residual and duality indicators given in Section~\ref{sec:indicators}.
We also described an unusual approach taken in \MC\ for using
parallel computers in an adaptive setting, based on joint work with
R. Bank~\cite{BaHo98a}.
We then derived global $L^2$- and $H^1$-error estimates for the solutions
produced by the parallel algorithm, by interpreting the algorithm as a
special partition of unity method~\cite{BaMe97} and by using the recent
local estimates of Xu and Zhou~\cite{XuZh97}.

As an illustrative example, we took a brief look at the
Hamiltonian and momentum constraints in the Einstein equations.
We first summarized a number of operator properties and solvability
results recently established in~\cite{HoBe01},
and then derived two {\em a priori} error estimates for Galerkin
approximations, completing the theoretical framework for effective
use of adaptive multilevel finite element methods.
We finished by presenting an illustrative example using the \MC\ software.
More detailed examples may be found in~\cite{HoBe99,BeHo99}.

\section*{Acknowledgements}
The author thanks K.~Thorne and H.~Keller for many fruitful discussions
over several years at Caltech, and also thanks D.~Bernstein for numerous
discussions of the relativity applications which continually renew
my fascination with physics.
The author would also like to thank D.~Arnold, R.~Bank, and D.~Estep
for their helpful advice which contributed to this work.
This work was supported in part by a UCSD Hellman Fellowship
and in part by NSF~CAREER~Award~9875856.

\bibliographystyle{abbrv}
\bibliography{../bib/books,../bib/papers,../bib/mjh,../bib/library,../bib/ref-gn,../bib/coupling,../bib/pnp}

\begin{thebibliography}{10}

\bibitem{Adam78}
R.~A. Adams.
\newblock {\em Sobolev Spaces}.
\newblock Academic Press, San Diego, CA, 1978.

\bibitem{Akso01}
B.~Aksoylu.
\newblock {\em Adaptive Multilevel Numerical Methods}.
\newblock PhD thesis, Department of Mathematics, UC San Diego, 2001.

\bibitem{AMP97}
D.N. Arnold, A.~Mukherjee, and L.~Pouly.
\newblock Locally adapted tetrahedral meshes using bisection.
\newblock {\em SIAM J.\ Sci.\ Statist.\ Comput.}, 22(2):431--448, 1997.

\bibitem{Aubi82}
T.~Aubin.
\newblock {\em Nonlinear Analysis on Manifolds. {Monge-Amp\'ere} Equations}.
\newblock Springer-Verlag, New York, NY, 1982.

\bibitem{BaMe97}
I.~Babu\v{s}ka and J.~M. Melenk.
\newblock The partition of unity finite element method.
\newblock {\em Internat. J. Numer. Methods Engrg.}, 40:727--758, 1997.

\bibitem{BaRh78b}
I.~Babu\v{s}ka and W.C. Rheinboldt.
\newblock Error estimates for adaptive finite element computations.
\newblock {\em SIAM J.\ Numer.\ Anal.}, 15:736--754, 1978.

\bibitem{BaRh78a}
I.~Babu\v{s}ka and W.C. Rheinboldt.
\newblock A posteriori error estimates for the finite element method.
\newblock {\em Int. J. Numer. Meth. Engrg.}, 12:1597--1615, 1978.

\bibitem{BHW99}
N.~Baker, M.~Holst, and F.~Wang.
\newblock Adaptive multilevel finite element solution of the
  {Poisson-Boltzmann} equation {II}: refinement at solvent accessible surfaces
  in biomolecular systems.
\newblock {\em J.\ Comput.\ Chem.}, 21:1343--1352, 2000.

\bibitem{BaHo98a}
R.~Bank and M.~Holst.
\newblock A new paradigm for parallel adaptive mesh refinement.
\newblock {\em SIAM J.\ Sci.\ Comput.}, 22(4):1411--1443, 2000.

\bibitem{PLTMG}
R.~E. Bank.
\newblock {\em {PLTMG}: A Software Package for Solving Elliptic Partial
  Differential Equations, Users' Guide 8.0}.
\newblock Software, Environments and Tools, Vol.~5. SIAM, Philadelphia, PA,
  1998.

\bibitem{BaDu81}
R.~E. Bank and T.~F. Dupont.
\newblock An optimal order process for solving finite element equations.
\newblock {\em Math. Comp.}, 36(153):35--51, 1981.

\bibitem{BDY88}
R.~E. Bank, T.~F. Dupont, and H.~Yserentant.
\newblock The hierarchical basis multigrid method.
\newblock {\em Numer.\ Math.}, 52:427--458, 1988.

\bibitem{BHMP98}
R.~E. Bank, M.~Holst, B.~Mantel, J.~Periaux, and C.~H. Zhou.
\newblock {CFD PPLTMG}: Using a posteriori error estimates and domain
  decomposition.
\newblock In {\em ECCOMAS 98}, New York, NY, 1998. John Wiley \& Sons.

\bibitem{BaMi89}
R.~E. Bank and H.~D. Mittelmann.
\newblock Stepsize selection in continuation procedures and damped {Newton's}
  method.
\newblock {\em J. Computational and Applied Mathematics}, 26:67--77, 1989.

\bibitem{BaRo80}
R.~E. Bank and D.~J. Rose.
\newblock Parameter selection for {Newton}-like methods applicable to nonlinear
  partial differential equations.
\newblock {\em SIAM J.\ Numer.\ Anal.}, 17(6):806--822, 1980.

\bibitem{BaRo81}
R.~E. Bank and D.~J. Rose.
\newblock {Global Approximate Newton Methods}.
\newblock {\em Numer.\ Math.}, 37:279--295, 1981.

\bibitem{BaRo82}
R.~E. Bank and D.~J. Rose.
\newblock Analysis of a multilevel iterative method for nonlinear finite
  element equations.
\newblock {\em Math. Comp.}, 39(160):453--465, 1982.

\bibitem{BaSm93}
R.~E. Bank and R.~K. Smith.
\newblock A posteriori error estimates based on hierarchical bases.
\newblock {\em SIAM J.\ Numer.\ Anal.}, 30(4):921--935, 1993.

\bibitem{BaSm97}
R.~E. Bank and R.~K. Smith.
\newblock Mesh smoothing using {\em a posteriori} error estimates.
\newblock {\em SIAM J.\ Numer.\ Anal.}, 34:979--997, 1997.

\bibitem{BaWe85}
R.~E. Bank and A.~Weiser.
\newblock Some a posteriori error estimators for elliptic partial differential
  equations.
\newblock {\em Math.\ Comp.}, 44(170):283--301, 1985.

\bibitem{Bans91b}
E.~B{\"a}nsch.
\newblock An adaptive finite-element strategy for the three-dimsional
  time-dependent {Navier-Stokes} equations.
\newblock {\em Journal of Computional and Applied Mathematics}, 36:3--28, 1991.

\bibitem{Bans91a}
E.~B{\"a}nsch.
\newblock Local mesh refinement in 2 and 3 dimensions.
\newblock {\em Impact of Computing in Science and Engineering}, 3:181--191,
  1991.

\bibitem{BaSi95}
E.~B{\"a}nsch and K.~G. Siebert.
\newblock {\em A posteriori} error estimation for nonlinear problems by duality
  techniques.
\newblock Technical report, Institut f{\"u}r Angewandte Mathematik,
  Hermann-Herder-Strase 10, 79104 Freiburg, Germany, 1995.

\bibitem{BBJL98}
P.~Bastian, K.~Birken, K.~Johannsen, S.~Lang, N.~Neuss, H.~Rentz-Reichert, and
  C.~Wieners.
\newblock {\em {UG} -- A Flexible Software Toolbox for Solving Partial
  Differential Equations}, 1998.

\bibitem{BER95}
R.~Beck, B.~Erdmann, and R.~Roitzsch.
\newblock {KASKADE 3.0}: An ojbect-oriented adaptive finite element code.
\newblock Technical Report TR95--4, Konrad-Zuse-Zentrum for
  Informationstechnik, Berlin, 1995.

\bibitem{BeHo99}
D.~Bernstein and M.~Holst.
\newblock {\em Adaptive Finite Element Solution of the Constraint Equations in
  General Relativity II. Examples}.
\newblock In preparation.

\bibitem{Bey96}
J.~Bey.
\newblock Adaptive grid manager: {AGM3D} manual.
\newblock Technical Report~50, SFB 382, Math. Inst. Univ. Tubingen, 1996.

\bibitem{BrSc94}
S.~C. Brenner and L.~R. Scott.
\newblock {\em The Mathematical Theory of Finite Element Methods}.
\newblock Springer-Verlag, New York, NY, 1994.

\bibitem{Brez85}
F.~Brezzi.
\newblock Mathematical theory of finite elements.
\newblock In A.~K. Noor and W.~D. Pilkey, editors, {\em State-of-the-art
  Surveys on Finite Element Technology}, pages 1--25, New York, NY, 1985. The
  American Society of Mechanical Engineers.

\bibitem{BHR77}
F.~Brezzi, W.~W. Hager, and P.~A. Raviart.
\newblock Error estimates for the finite element solution of variational
  inequalities.
\newblock {\em Numer.\ Math.}, 28:431--443, 1977.

\bibitem{CSZ94}
T.~F. Chan, B.~Smith, and J.~Zou.
\newblock Overlapping {Schwarz} methods on unstructured meshes using
  non-matching coarse grids.
\newblock Technical Report CAM 94-8, Department of Mathematics, UCLA, 1994.

\bibitem{CGZ97}
T.F. Chan, S.~Go, and L.~Zikatanov.
\newblock Lecture notes on multilevel methods for elliptic problems on
  unstructured meshes.
\newblock Technical report, Dept. of Mathematics, UCLA, 1997.

\bibitem{CIM92}
Y.~Choquet-Bruhat, J.~Isenberg, and V.~Moncrief.
\newblock Solutions of constraints for {Einstein} equations.
\newblock {\em C.R. Acad. Sci. Paris}, 315:349--355, 1992.

\bibitem{ChYo80}
Y.~Choquet-Bruhat and J.~W. {York, Jr.}
\newblock The {Cauchy} problem.
\newblock In A.~Held, editor, {\em General Relativity and Gravitation}, New
  York, 1980. Plenum Press.

\bibitem{Ciar78}
P.~G. Ciarlet.
\newblock {\em The Finite Element Method for Elliptic Problems}.
\newblock North-Holland, New York, NY, 1978.

\bibitem{Clem75}
Ph. Cl\'{e}ment.
\newblock Approximation by finite element functions using local regularization.
\newblock {\em R.A.I.R.O.}, 2:77--84, 1975.

\bibitem{DaHm89}
W.~Dahmen.
\newblock Smooth piecewise quadratic surfaces.
\newblock {\em Mathematical Methods in Computer-Aided Geometric Design}, pages
  181--193, 1989.

\bibitem{DaMi84}
W.~Dahmen and C.~Micchelli.
\newblock Subdivision algorithms for the generation of box spline surfaces.
\newblock {\em Computer-Aided Geometric Design}, 18(2):115--129, 1984.

\bibitem{Davi63}
P.~J. Davis.
\newblock {\em Interpolation and Approximation}.
\newblock Dover Publications, Inc., New York, NY, 1963.

\bibitem{DES82}
R.~S. Dembo, S.~C. Eisenstat, and T.~Steihaug.
\newblock {Inexact Newton Methods}.
\newblock {\em SIAM J.\ Numer.\ Anal.}, 19(2):400--408, 1982.

\bibitem{DeLo93}
R.~A. DeVore and G.~G. Lorentz.
\newblock {\em Constructive Approximation}.
\newblock Springer-Verlag, New York, NY, 1993.

\bibitem{EiWa92}
S.~C. Eisenstat and H.~F. Walker.
\newblock {Globally Convergent Inexact Newton Methods}.
\newblock Technical report, Dept. of Mathematics and Statistics, Utah State
  University, 1992.

\bibitem{EHM2001}
D.~Estep, M.~Holst, and D.~Mikulencak.
\newblock Accounting for stability: {\em a posteriori} error estimates for
  finite element methods based on residuals and variational analysis.
\newblock {\em Communications in Numerical Methods in Engineering},
  18(1):15--30, 2002.

\bibitem{Falk74}
R.~Falk.
\newblock Error estimates for the approximation of a class of variational
  inequalities.
\newblock {\em Math. Comput.}, 28:963--971, 1974.

\bibitem{FuKu80}
S.~Fucik and A.~Kufner.
\newblock {\em Nonlinear Differential Equations}.
\newblock Elsevier Scientific Publishing Company, New York, NY, 1980.

\bibitem{GrSc00}
M.~Griebel and M.~A. Schweitzer.
\newblock A particle-partition of unity method for the solution of elliptic,
  parabolic, and hyperbolic {PDE}s.
\newblock {\em SIAM J.\ Sci.\ Statist.\ Comput.}, 22(3):853--890, 2000.

\bibitem{Grim96}
C.M. Grimm.
\newblock {\em Modeling Surfaces of Arbitrary Topology using Manifolds}.
\newblock PhD thesis, Department of Computer Science, Brown University, May
  1996.

\bibitem{GrHu95}
C.M. Grimm and J.F. Hughes.
\newblock Modeling surfaces of arbitrary topology using manifolds.
\newblock In {\em Graphics (Proceedings of SIGGRAPH '95)}, volume 29(4) of {\em
  SIGGRAPH}, pages 359--369. ACM, July, 1995.

\bibitem{Hack85}
W.~Hackbusch.
\newblock {\em Multi-grid Methods and Applications}.
\newblock Springer-Verlag, Berlin, Germany, 1985.

\bibitem{Hack94}
W.~Hackbusch.
\newblock {\em Iterative Solution of Large Sparse Systems of Equations}.
\newblock Springer-Verlag, Berlin, Germany, 1994.

\bibitem{HaEl73}
S.~W. Hawking and G.~F.~R. Ellis.
\newblock {\em The Large Scale Structure of Space-Time}.
\newblock Cambridge University Press, Cambridge, MA, 1973.

\bibitem{Hebe91}
E.~Hebey.
\newblock {\em {Sobolev} Spaces on {Riemannian} Manifolds}.
\newblock Springer-Verlag, Berlin, Germany, 1991.

\bibitem{Hols97c}
M.~Holst.
\newblock {\em Finite element approximation theory on {Riemannian} manifolds}.
\newblock In preparation.

\bibitem{HBW99}
M.~Holst, N.~Baker, and F.~Wang.
\newblock Adaptive multilevel finite element solution of the
  {Poisson-Boltzmann} equation {I}: algorithms and examples.
\newblock {\em J.\ Comput.\ Chem.}, 21:1319--1342, 2000.

\bibitem{HoBe99}
M.~Holst and D.~Bernstein.
\newblock {\em Adaptive Finite Element Solution of the Constraint Equations in
  General Relativity I. Algorithms}.
\newblock In preparation.

\bibitem{HoBe01}
M.~Holst and D.~Bernstein.
\newblock {\em Weak solutions to the {Einstein} constraint equations on
  manifolds with boundary}.
\newblock In preparation.

\bibitem{HoTi96b}
M.~Holst and E.~Titi.
\newblock Determining projections and functionals for weak solutions of the
  {Navier-Stokes} equations.
\newblock In Yair Censor and Simeon Reich, editors, {\em Recent Developments in
  Optimization Theory and Nonlinear Analysis}, volume 204 of {\em Contemporary
  Mathematics}, Providence, Rhode Island, 1997. American Mathematical Society.

\bibitem{HoVa97a}
M.~Holst and S.~Vandewalle.
\newblock Schwarz methods: to symmetrize or not to symmetrize.
\newblock {\em SIAM J.\ Numer.\ Anal.}, 34(2):699--722, 1997.

\bibitem{Isen95}
J.~Isenberg.
\newblock Constant mean curvature solutions of the {Einstein} constraint
  equations on closed manifolds.
\newblock {\em Classical and Quantum Gravity}, 12:2249--2274, 1995.

\bibitem{IsMo96}
J.~Isenberg and V.~Moncrief.
\newblock A set of nonconstant mean curvature solutions of the {Einstein}
  constraint equations on closed manifolds.
\newblock {\em Classical and Quantum Gravity}, 13:1819--1847, 1996.

\bibitem{JeKe91}
J.~W. Jerome and T.~Kerkhoven.
\newblock A finite element approximation theory for the drift diffusion
  semiconductor model.
\newblock {\em SIAM J.\ Numer.\ Anal.}, 28(2):403--422, 1991.

\bibitem{Kell87}
H.~B. Keller.
\newblock {\em Numerical Methods in Bifurcation Problems}.
\newblock Tata Institute of Fundamental Research, Bombay, India, 1987.

\bibitem{Kell92}
H.~B. Keller.
\newblock {\em Numerical Methods for Two-Point Boundary-Value Problems}.
\newblock Dover Publications, New York, NY, 1992.

\bibitem{KeJe90}
T.~Kerkhoven and J.~W. Jerome.
\newblock {$L_{\infty}$} stability of finite element approximations of elliptic
  gradient equations.
\newblock {\em Numer.\ Math.}, 57:561--575, 1990.

\bibitem{Lee97}
J.~M. Lee.
\newblock {\em Riemannian Manifolds}.
\newblock Springer-Verlag, New York, NY, 1997.

\bibitem{LiJo94}
A.~Liu and B.~Joe.
\newblock Relationship between tetrahedron shape measures.
\newblock {\em BIT}, 34:268--287, 1994.

\bibitem{LiJo95}
A.~Liu and B.~Joe.
\newblock Quality local refinement of tetrahedral meshes based on bisection.
\newblock {\em SIAM J.\ Sci.\ Statist.\ Comput.}, 16(6):1269--1291, 1995.

\bibitem{LiRh94}
J.L. Liu and W.C. Rheinboldt.
\newblock A posteriori finite element error estimators for indefinite elliptic
  boundary value problems.
\newblock {\em Numer. Funct. Anal. and Optimiz.}, 15(3):335--356, 1994.

\bibitem{LiRh96}
J.L. Liu and W.C. Rheinboldt.
\newblock A posteriori finite element error estimators for parametrized
  nonlinear boundary value problems.
\newblock {\em Numer. Funct. Anal. and Optimiz.}, 17(5):605--637, 1996.

\bibitem{MaHu94}
J.~E. Marsden and T.~J.~R. Hughes.
\newblock {\em Mathematical Foundations of Elasticity}.
\newblock Dover Publications, New York, NY, 1994.

\bibitem{Maub95}
J.M. Maubach.
\newblock Local bisection refinement for {N}-simplicial grids generated by
  relection.
\newblock {\em SIAM J.\ Sci.\ Statist.\ Comput.}, 16(1):210--277, 1995.

\bibitem{Muck93}
E.P. Mucke.
\newblock {\em Shapes and Implementations in Three-Dimensional Geometry}.
\newblock PhD thesis, Dept. of Computer Science, University of Illinois at
  Urbana-Champaign, 1993.

\bibitem{Mukh96}
A.~Mukherjee.
\newblock {\em An Adaptive Finite Element Code for Elliptic Boundary Value
  Problems in Three Dimensions with Applications in Numerical Relativity}.
\newblock PhD thesis, Dept. of Mathematics, The Pennsylvania State University,
  1996.

\bibitem{NiSc74}
J.~A. Nitsche and A.~H. Schatz.
\newblock Interior estimates for {Ritz-Galerkin} methods.
\newblock {\em Math. Comp.}, 28:937--958, 1974.

\bibitem{MuYo74a}
N.~{O'Murchadha} and J.W. {York, Jr.}
\newblock Initial-value problem of general relativity. {I.} {General}
  formulation and physical interpretation.
\newblock {\em Phys.\ Rev.\ D}, 10(2):428--436, July, 1974.

\bibitem{MuYo74b}
N.~{O'Murchadha} and J.W. {York, Jr.}
\newblock Initial-value problem of general relativity. {II.} {Stability} of
  solutions of the initial-value equations.
\newblock {\em Phys.\ Rev.\ D}, 10(2):437--446, July, 1974.

\bibitem{MuYo73}
N.~{O'Murchadha} and J.W. {York, Jr.}
\newblock Existence and uniqueness of solutions of the {Hamiltonian} constraint
  of general relativity on compact manifolds.
\newblock {\em J. Math. Phys.}, 14(11):1551--1557, November, 1973.

\bibitem{OrRh70}
J.~M. Ortega and W.~C. Rheinboldt.
\newblock {\em Iterative Solution of Nonlinear Equations in Several Variables}.
\newblock Academic Press, New York, NY, 1970.

\bibitem{Riva84}
M.C. Rivara.
\newblock Algorithms for refining triangular grids suitable for adaptive and
  multigrid techniques.
\newblock {\em International Journal for Numerical Methods in Engineering},
  20:745--756, 1984.

\bibitem{Riva91}
M.C. Rivara.
\newblock Local modification of meshes for adaptive and/or multigrid
  finite-element methods.
\newblock {\em Journal of Computational and Applied Mathematics}, 36:79--89,
  1991.

\bibitem{RoSt75}
I.G. Rosenberg and F.~Stenger.
\newblock A lower bound on the angles of triangles constructed by bisecting the
  longest side.
\newblock {\em Math.\ Comp.}, 29:390--395, 1975.

\bibitem{Rose97}
S.~Rosenberg.
\newblock {\em The {Laplacian} on a {Riemannian} Manifold}.
\newblock Cambridge University Press, Cambridge, MA, 1997.

\bibitem{RuSt87}
J.~W. Ruge and K.~St{\"u}ben.
\newblock Algebraic multigrid ({AMG}).
\newblock In S.~F. McCormick, editor, {\em Multigrid Methods}, volume~3 of {\em
  Frontiers in Applied Mathematics}, pages 73--130. SIAM, Philadelphia, PA,
  1987.

\bibitem{Scha74}
A.~H. Schatz.
\newblock An oberservation concerning {Ritz-Galerkin} methods with indefinite
  bilinear forms.
\newblock {\em Math.\ Comp.}, 28(128):959--962, 1974.

\bibitem{ScWa00}
A.~H. Schatz and J.~Wang.
\newblock Some new error estimates for {Ritz-Galerkin} methods with minimal
  regularity assumptions.
\newblock {\em Math.\ Comp.}, 62:445--475, 2000.

\bibitem{Schw91}
G.~Schwarz.
\newblock {\em Hodge Decomposition: A Method for Solving Boundary Value
  Problems}.
\newblock Springer-Verlag, New York, NY, 1991.

\bibitem{ScZh90}
L.~R. Scott and S.~Zhang.
\newblock Finite element interpolation of nonsmooth functions satisfying
  boundary conditions.
\newblock {\em Math.\ Comp.}, 54(190):483--493, 1990.

\bibitem{Styn80}
M.~Stynes.
\newblock On faster convergence of the bisection method for all triangles.
\newblock {\em Math.\ Comp.}, 35:1195--1201, 1980.

\bibitem{VMB94}
P.~Vanek, J.~Mandel, and M.~Brezina.
\newblock Algebraic multigrid on unstructured meshes.
\newblock Technical Report UCD/CCM 34, Center for Computational Mathematics,
  University of Colorado at Denver, 1994.

\bibitem{Verf94}
R.~Verf{\"u}rth.
\newblock A posteriori error estimates for nonlinear problems. {Finite} element
  discretizations of elliptic equations.
\newblock {\em Math.\ Comp.}, 62(206):445--475, 1994.

\bibitem{Verf96}
R.~Verf{\"u}rth.
\newblock {\em A Review of {\em A Posteriori} Error Estimation and Adaptive
  Mesh-Refinement Techniques}.
\newblock John Wiley \& Sons Ltd, New York, NY, 1996.

\bibitem{Wald84}
R.~M. Wald.
\newblock {\em General Relativity}.
\newblock University of Chicago Press, Chicago, IL, 1984.

\bibitem{Wlok92}
J.~Wloka.
\newblock {\em Partial Differential Equations}.
\newblock Cambridge University Press, Cambridge, MA, 1992.

\bibitem{Xu92a}
J.~Xu.
\newblock Iterative methods by space decomposition and subspace correction.
\newblock {\em SIAM Review}, 34(4):581--613, 1992.

\bibitem{Xu92b}
J.~Xu.
\newblock A novel two-grid method for semilinear elliptic equations.
\newblock Technical report, Dept. of Mathematics, Penn State University, 1992.

\bibitem{Xu92c}
J.~Xu.
\newblock Two-grid finite element discretization for nonlinear elliptic
  equations.
\newblock Technical report, Dept. of Mathematics, Penn State University, 1992.

\bibitem{XuZh97}
J.~Xu and A.~Zhou.
\newblock Local and parallel finite element algorithms based on two-grid
  discretizations.
\newblock {\em Math.\ Comp.}, 69:881--909, 2000.

\bibitem{Yosi80}
K.~Yosida.
\newblock {\em Functional Analysis}.
\newblock Springer-Verlag, Berlin, Germany, 1980.

\end{thebibliography}


\vspace*{0.5cm}

\end{document}